\documentclass[10pt,reqno]{amsart}
\usepackage{amsmath,amssymb,latexsym,esint,cite,mathrsfs}
\usepackage{verbatim,relsize,color,soul,enumitem}
\usepackage[colorlinks=true,urlcolor=blue, citecolor=red,linkcolor=blue,
linktocpage,pdfpagelabels, bookmarksnumbered,bookmarksopen]{hyperref}
\usepackage[hyperpageref]{backref}

\usepackage{graphicx}
\usepackage{tikz}
\usepackage[font=small,skip=5pt]{caption}
\usepackage{lmodern}
 
\usepackage[T1]{fontenc}

\usepackage{geometry}
\numberwithin{equation}{section}
\newtheorem{theorem}{Theorem}[section]
\newtheorem{lemma}[theorem]{Lemma}
\newtheorem{proposition}[theorem]{Proposition}

\newtheoremstyle{remarkstyle}
{}{}{}{}{\bfseries}{.}{ }{\thmname{#1}\thmnumber{ #2}\thmnote{ (#3)}}
\theoremstyle{remarkstyle}
\newtheorem{remark}{Remark}[section]

\newtheorem{observation}{Observation}[section]

\newcommand{\N}{\mathbb N}
\newcommand{\Z}{\mathbb Z}

\newcommand{\R}{\mathbb R}
\newcommand{\C}{\mathbb C}

\newcommand{\Ac}{\mathcal A}
\newcommand{\Bc}{\mathcal B}

\newcommand{\Dc}{\mathcal D}

\newcommand{\Gc}{\mathcal G}

\newcommand{\Mcal}{\mathcal M}

\newcommand{\vareps}{\varepsilon}

\newcommand{\cdotb}{\boldsymbol{\cdot}}
\newcommand{\scal}[1]{\left\langle #1 \right\rangle}

\DeclareMathOperator*{\curl}{curl}
\DeclareMathOperator*{\loc}{loc}

\DeclareMathOperator*{\opt}{opt}
\DeclareMathOperator*{\dist}{dist}
\DeclareMathOperator*{\supp}{supp}

\DeclareMathOperator*{\rea}{Re}
\DeclareMathOperator*{\ima}{Im}
\DeclareMathOperator*{\sigc}{\sigma_c}

\DeclareMathOperator*{\spec}{spec}

\title[3D Magnetic NLS]
{The 3D nonlinear Schr\"odinger equation with a constant magnetic field revisited}

\author[V. D. Dinh]{Van Duong Dinh}
\address[V. D. Dinh]{Ecole Normale Sup\'erieure de Lyon \& CNRS, UMPA (UMR 5669), France
	and 
	Department of Mathematics, Ho Chi Minh City University of Education, 280 An Duong Vuong, Ho Chi Minh City, Vietnam}
\email{contact@duongdinh.com}


\begin{document}
	
	\begin{abstract}
	In this paper, we revisit the Cauchy problem for the three dimensional nonlinear Schr\"od-inger equation with a constant magnetic field. We first establish sufficient conditions that ensure the existence of global in time and finite time blow-up solutions. In particular, we derive sharp thresholds for global existence versus blow-up for the equation with mass-critical and mass-supercritical nonlinearities. We next prove the existence and orbital stability of normalized standing waves which extend the previous known results to the mass-critical and mass-supercritical cases. To show the existence of normalized solitary waves, we present a new approach that avoids the celebrated concentration-compactness principle. Finally, we study the existence and strong instability of ground state standing waves which greatly improve the previous literature.
	\end{abstract}
		
	\maketitle

	
	\section{Introduction}
	\label{S1}
	\setcounter{equation}{0}	
	The present paper concerns with the Cauchy problem for nonlinear Schr\"odinger equations with a constant magnetic field in three dimensions
	\begin{align} \label{mag-NLS}
	\left\{
	\begin{array}{rcl}
	i \partial_t u + (\nabla+iA)^2 u &=& -|u|^\alpha u, \\
	\left. u \right|_{t=0} &=& u_0,
	\end{array}
	\right. \quad (t,x) \in \R^+ \times \R^3,
	\end{align}
	where
	\begin{align} \label{defi-A}
	A(x)= \frac{b}{2}(-x_2, x_1, 0), \quad x=(x_1, x_2, x_3) \in \R^3
	\end{align}
	is a vector-valued potential modeling the effect of an external magnetic field 
	\begin{align} \label{defi-B}
	B= \curl(A)=(0,0,b), \quad b \ne 0.
	\end{align}
	
	The Schr\"odinger equation with a constant magnetic field is an effective model describing properties of a single non-relativistic quantum particle in the presence of an electromagnetic field (see e.g., \cite{LS}). A rigorous mathematical investigation of the linear Schr\"odinger operator with a constant magnetic field was studied by J. Avron, I. Herbst, and B. Simon \cite{AHS-1, AHS-2, AHS-3}. 
	
	The nonlinear Schr\"odinger equation with a constant magnetic field \eqref{mag-NLS} can be regarded as a special case of the Gross-Pitaevskii equation describing the Bose-Einstein condensation with a critical rotational speed and a partial harmonic confinement potential (see e.g., \cite{BC}), namely
	\[
	i \partial_t u + \Delta u - bL_z u - \frac{b^2}{4} (x_1^2+x_2^2) u = - |u|^2 u, \quad (t,x) \in \R^+ \times \R^3,
	\]
	where 
	\begin{align} \label{defi-Lz}
	L_z := i (x_2 \partial_{x_1} - x_1 \partial_{x_2})
	\end{align}
	is the third component of the angular momentum vector 
	\[
	-ix \wedge \nabla = \left(L_x, L_y, L_z\right) = i\left( x_3 \partial_{x_2}-x_2 \partial_{x_3}, x_1 \partial_{x_3}-x_3 \partial_{x_1}, x_2 \partial_{x_1}-x_1 \partial_{x_2} \right).
	\]
	
	To our knowledge, the first paper addressed \eqref{mag-NLS} belongs to M. J. Esteban and P.-L. Lions \cite{EL}, where the existence of normalized standing waves related to \eqref{mag-NLS} was proved. T. Cazenave and M. J. Esteban \cite{CE} later established the local well-posedness for \eqref{mag-NLS}. As a consequence, they showed that normalized standing waves obtained in \cite{EL} are indeed orbitally stable under the flow of \eqref{mag-NLS}. Note that these existence and stability results hold with a mass-subcritical nonlinearity, i.e., $0<\alpha<\frac{4}{3}$. In the mass-(super)critical case, i.e., $\frac{4}{3}\leq \alpha <4$, J. M. Gon{\c{c}}alves Ribeiro \cite{Ribeiro} proved the existence of finite time blow-up solutions to \eqref{mag-NLS} with negative energy. Also with this regime of nonlinearity, the orbital instability of (rotational invariant) ground state standing waves was studied by J. M. Gon{\c{c}}alves Ribeiro \cite{Ribeiro-insta} and R. Fukuizumi and M. Ohta \cite{FO}. Recently, a new blow-up result for \eqref{mag-NLS} was found by T. F. Kieffer and M. Loss \cite{KL}.    	
	
	The main purposes of this paper are three folds:
	\begin{itemize}
		\item First, we investigate sufficient conditions for the existence of global in time and finite time blow-up solutions. In particular, we derive sharp thresholds for global existence versus blow-up for the equation with mass-(super)critical nonlinearities.
		\item Second, we study the existence and orbital stability of normalized standing waves.
		\item Finally, we address the existence and strong instability of ground state standing waves.
	\end{itemize}   
	
	\subsection{Global existence and finite blow-up}
	Before stating our results in this direction, let us recall the local theory for \eqref{mag-NLS}. The local well-posedness for \eqref{mag-NLS} with initial data in $H^1_A(\R^3)$ was established by T. Cazenave and M. J. Esteban \cite{CE} (see also \cite[Section 9.1]{Cazenave}),
	where 
	\[
	H^1_A(\R^3):= \left\{f \in L^2(\R^3) \ : \ |(\nabla+ iA) f| \in L^2(\R^3) \right\}
	\]
	is a Hilbert space equipped the norm
	\[
	\|f\|_{H^1_A}^2 = \|(\nabla+iA)f\|^2_{L^2} + \|f\|^2_{L^2}.
	\]
	
	\begin{proposition}[LWP \cite{CE}]
		Let $0<\alpha<4$ and $u_0 \in H^1_A(\R^3)$. Then there exist $T^* \in (0,\infty]$ and a unique maximal solution 
		\[
		u \in C([0,T^*), H^1_A(\R^3)) \cap C^1([0,T^*), H^{-1}_A(\R^3)),
		\]
		where $H^{-1}_A(\R^3)$ is the dual space of $H^1_A(\R^3)$. The maximal time of existence satisfies the blow-up alternative: if $T^*<\infty$, then $\lim_{t \nearrow T^*} \|u(t)\|_{H^1_A} =\infty$. In addition, there are conservation laws of mass and energy, namely
		\begin{align*}
		M(u(t)) &= \|u(t)\|^2_{L^2} = M(u_0), \tag{Mass} \\
		E(u(t)) &= \frac{1}{2} \|(\nabla + iA) u(t)\|^2_{L^2} - \frac{1}{\alpha+2} \|u(t)\|^{\alpha+2}_{L^{\alpha+2}} = E(u_0), \tag{Energy}
		\end{align*}
		for all $t\in [0,T^*)$. 
	\end{proposition}
	
	In the mass-subcritical case, it was proved in \cite{CE} that solutions to \eqref{mag-NLS} exist globally in time, i.e., $T^*=\infty$. In the mass-(super)critical cases, there exist solutions to \eqref{mag-NLS} which blow up in finite time, i.e., $T^*<\infty$ (see e.g., \cite{Ribeiro, Garcia, KL}). To state blow-up results for \eqref{mag-NLS}, let us introduce the following Hilbert space
	\begin{align} \label{Sigma-A}
	\Sigma_A(\R^3):= \left\{ f\in H^1_A(\R^3) \ : \ |x| f \in L^2(\R^3)\right\}
	\end{align}
	endowed with the norm
	\[
	\|f\|^2_{\Sigma_A} := \|(\nabla+iA)f\|^2_{L^2} + \|xf\|^2_{L^2} + \|f\|^2_{L^2}.
	\]
	We will see in Remark \ref{rem-equi-norm} that $\Sigma_A(\R^3) \equiv \Sigma(\R^3)$, where
	\begin{align} \label{Sigma}
	\Sigma(\R^3):= \left\{ f\in H^1(\R^3) \ : \ |x| f \in L^2(\R^3)\right\}
	\end{align}
	equipped with the norm
	\[
	\|f\|^2_{\Sigma}:= \|\nabla f\|^2_{L^2} + \|xf\|^2_{L^2} + \|f\|^2_{L^2}.
	\]
	
	Thanks to this fact, we have the following useful identity
	\begin{align} \label{mag-norm}
	\|(\nabla + iA) f\|_{L^2}^2 = \|\nabla f\|^2_{L^2} + bR(f) + \frac{b^2}{4} \|\rho f\|^2_{L^2},
	\end{align}
	where $\rho:=\sqrt{x_1^2+x_2^2}$ and
	\begin{align} \label{defi-R}
	R(f) := i \int (x_2\partial_{x_1} f - x_1 \partial_{x_2} f) \overline{f} dx = \int L_zf \overline{f} dx,
	\end{align}
	where $L_z$ is as in \eqref{defi-Lz}. Note that, by H\"older's inequality, it is straightforward to see that the functional $R$ is well-defined on $\Sigma(\R^3)$.  
	
	By making use of virial identity related to \eqref{mag-NLS} (see Lemma \ref{lem-viri-iden}), the existence of finite time blow-up solutions to \eqref{mag-NLS} was showed by J. M. Gonçalves Ribeiro \cite{Ribeiro} (see also \cite{Garcia} for a more general magnetic potential).

	\begin{proposition}[\cite{Ribeiro}] \label{prop-blow-1}
		Let $\frac{4}{3} \leq \alpha<4$. Let $u_0 \in \Sigma_A(\R^3)$ be such that $E(u_0)<0$. Then the corresponding solution to \eqref{mag-NLS} blows up in finite time, i.e., $T^*<\infty$. 
	\end{proposition}
	
	Recently, T. F. Kieffer and M. Loss \cite{KL} showed the following blow-up result for \eqref{mag-NLS}. 
	
	\begin{proposition}[\cite{KL}] \label{prop-blow-2}
		Let $\frac{4}{3} \leq \alpha<4$. Let $u_0 \in \Sigma_A(\R^3)$ and $u:[0,T^*) \times \R^3 \rightarrow \C$ be the corresponding solution to \eqref{mag-NLS}. Then the solution blows up in finite time, i.e., $T^*<\infty$ provided that one of the following conditions holds:  
		\begin{itemize} [leftmargin=5mm]
			\item[(1)] $E_0(u_0)<0$;
			\item[(2)] $E_0(u_0)=0$ and $\ima \mathlarger{\int} x \cdot \nabla u_0(x) \overline{u}_0(x)dx <0$;
			\item[(3)] $E_0(u_0)>0$ and $\ima \mathlarger{\int} x \cdot \nabla u_0(x) \overline{u}_0(x) dx <-\sqrt{2E_0(u_0)} \|xu_0\|_{L^2}$.
		\end{itemize}
		Here 
		\begin{align} \label{defi-E0}
		E_0(f):= \frac{1}{2} \|\nabla f\|^2_{L^2} +\frac{b^2}{8} \|\rho f\|^2_{L^2} - \frac{1}{\alpha+2} \|f\|^{\alpha+2}_{L^{\alpha+2}}.
		\end{align}
	\end{proposition}
	
	The proof of this blow-up result is based on the virial identity and the following observation.
	
	\begin{lemma} \label{lem-cons-angu}
		Let $0<\alpha<4$ and $u_0 \in \Sigma_A(\R^3)$. Let $u:[0,T^*) \times \R^3 \rightarrow \C$ be the corresponding solution to \eqref{mag-NLS}. Then the angular momentum $R(u(t))$ is real-valued and conserved along the flow of \eqref{mag-NLS}, i.e.,
		\[
		R(u(t)) = R(u_0), \quad \forall t\in [0,T^*).
		\] 
		In particular, we have
		\[
		E_0(u(t)) = E_0(u_0), \quad \forall t\in [0,T^*),
		\]
		where $E_0$ is as in \eqref{defi-E0}.
	\end{lemma}
	For the reader's convenience, we give a proof of this result in Section \ref{S3}.
	
	\begin{remark}
		In \cite{KL}, a relationship between $E(u_0)$ and $E_0(u_0)$ has been analyzed. In particular, for a magnetic field with the strength $b>0$, we have
		\[
		\left\{
		\begin{array}{ccc}
		E(u_0) > E_0(u_0) &\text{if}& R(u_0) >0, \\
		E(u_0) < E_0(u_0) &\text{if}& R(u_0) <0, \\
		E(u_0) = E_0(u_0) &\text{if}& R(u_0) =0. 
		\end{array}
		\right.
		\]
		Depending on the sign of the magnetic strength and the angular momentum, the blow-up condition given in Proposition \ref{prop-blow-1} may be better or weaker than the one of Proposition \ref{prop-blow-2} and vice versa.
	\end{remark}
	
	Our first result is the following sharp threshold for global existence versus finite time blow-up in the mass-critical case.
	
	\begin{proposition} \label{prop-thre-mass}
		Let $\alpha=\frac{4}{3}$. 
		\begin{itemize}[leftmargin=5mm]
			\item[(1)] If $u_0 \in H^1_A(\R^3)$ satisfies $\|u_0\|_{L^2} < \|Q\|_{L^2}$, where $Q$ is the unique positive radial solution to 
			\begin{align} \label{defi-Q}
			-\Delta Q + Q - |Q|^{\alpha} Q =0,
			\end{align}
			then the corresponding solution to \eqref{mag-NLS} exists globally in time, i.e., $T^*=\infty$.
			\item[(2)] For $c > \|Q\|_{L^2}$, there exists $u_0 \in \Sigma_A(\R^3)$ such that the corresponding solution to \eqref{mag-NLS} with initial data $\left. u\right|_{t=0} =u_0$ blows up in finite time, i.e., $T^*<\infty$.
		\end{itemize}
	\end{proposition}
	
	\begin{remark}
		It is not clear to us at the moment that whether or not there exists a blow-up solution to the mass-critical \eqref{mag-NLS} with the minimal mass $\|u_0\|_{L^2} = \|Q\|_{L^2}$. 
	\end{remark}
	
	Our next results are the following global existence in the mass-supercritical case. 
	
	\begin{proposition} \label{prop-gwp-1}
		Let $\frac{4}{3}<\alpha<4$. Let $u_0 \in H^1_A(\R^3)$ be such that $E(u_0) \geq 0$ and
		\begin{align} 
		E(u_0) [M(u_0)]^{\sigc} &< E^0(Q) [M(Q)]^{\sigc}, \label{cond-gwp-1-1} \\
		\|(\nabla+iA) u_0\|_{L^2} \|u_0\|_{L^2}^{\sigc} &< \|\nabla Q\|_{L^2} \|Q\|^{\sigc}_{L^2}, \label{cond-gwp-1-2}
		\end{align}
		where
		\begin{align} \label{defi-E^0-sigc}
		E^0(f):= \frac{1}{2}\|\nabla f\|^2_{L^2} - \frac{1}{\alpha+2} \|f\|^{\alpha+2}_{L^{\alpha+2}}, \quad\sigc:= \frac{4-\alpha}{3\alpha-4}.
		\end{align}
		Then the corresponding solution to \eqref{mag-NLS} exists globally in time, i.e., $T^*=\infty$, and satisfies
		\begin{align*}
		\|(\nabla+iA) u(t)\|_{L^2} \|u(t)\|^{\sigc}_{L^2} < \|\nabla Q\|_{L^2} \|Q\|_{L^2}^{\sigc}
		\end{align*}
		for all $t\in [0,\infty)$.
	\end{proposition}
	
	\begin{proposition} \label{prop-gwp-2}
		Let $\frac{4}{3}<\alpha<4$. Let $u_0 \in H^1_A(\R^3)$ be such that
		\begin{align} 
		E(u_0) [M(u_0)]^{\sigc} &= E^0(Q) [M(Q)]^{\sigc}, \label{cond-gwp-2-1} \\
		\|(\nabla +iA) u_0\|_{L^2} \|u_0\|_{L^2}^{\sigc} &< \|\nabla Q\|_{L^2} \|Q\|^{\sigc}_{L^2}. \label{cond-gwp-2-2}
		\end{align}
		Then the corresponding solution to \eqref{mag-NLS} exists globally in time, i.e., $T^*=\infty$ and satisfies
		\begin{align*}
		\|(\nabla +iA) u(t)\|_{L^2} \|u(t)\|_{L^2}^{\sigc} < \|\nabla Q\|_{L^2} \|Q\|^{\sigc}_{L^2}
		\end{align*}
		for all $t\in [0,\infty)$.
	\end{proposition}
	
	The following result gives a sharp threshold for global existence versus finite time blow-up in the mass-supercritical case. 
	
	\begin{theorem} \label{theo-dyn-below}
		Let $\frac{4}{3}<\alpha<4$. Let $u_0 \in \Sigma_A(\R^3)$ be such that $E_0(u_0) \geq 0$ and
		\begin{align} \label{cond-ener-below}
		E_0(u_0) [M(u_0)]^{\sigc} < E^0(Q) [M(Q)]^{\sigc},
		\end{align}
		where $E_0$ and $E^0$ are as in \eqref{defi-E0} and \eqref{defi-E^0-sigc} respectively.		
		\begin{itemize} [leftmargin=5mm]
			\item[(1)] If 
			\begin{align} \label{cond-gwp-below}
			\|\nabla u_0\|_{L^2} \|u_0\|_{L^2}^{\sigc} < \|\nabla Q\|_{L^2} \|Q\|^{\sigc}_{L^2},
			\end{align}
			then the corresponding solution to \eqref{mag-NLS} exists globally in time, i.e., $T^*=\infty$, and satisfies
			\begin{align*} 
			\|\nabla u(t)\|_{L^2} \|u(t)\|^{\sigc}_{L^2} < \|\nabla Q\|_{L^2} \|Q\|_{L^2}^{\sigc}
			\end{align*}
			for all $t\in [0,\infty)$.
			\item[(2)] If 
			\begin{align} \label{cond-blow-below}
			\|\nabla u_0\|_{L^2} \|u_0\|_{L^2}^{\sigc} > \|\nabla Q\|_{L^2} \|Q\|^{\sigc}_{L^2},
			\end{align}
			then the corresponding solution to \eqref{mag-NLS} satisfies
			\begin{align*} 
			\|\nabla u(t)\|_{L^2} \|u(t)\|^{\sigc}_{L^2}> \|\nabla Q\|_{L^2} \|Q\|^{\sigc}_{L^2}
			\end{align*}
			for all $t\in [0,T^*)$. Moreover, the solution blows up in finite time, i.e., $T^*<\infty$.
		\end{itemize}
	\end{theorem}
	
	\begin{remark} \label{rem-dyn-below}
		Here we only consider data with $E_0(u_0)\geq 0$ since solutions to \eqref{mag-NLS} with $E_0(u_0)<0$ blow up in finite time according to Proposition \ref{prop-blow-2}. Moreover, as we see from \eqref{est-G-2}, there is no $u_0 \in \Sigma_A(\R^3)$ satisfying \eqref{cond-ener-below} and 
		\[
		\|\nabla u_0\|_{L^2} \|u_0\|^{\sigc}_{L^2} = \|\nabla Q\|_{L^2} \|Q\|_{L^2}^{\sigc}.
		\]
		Hence Theorem \ref{theo-dyn-below} indeed gives a sharp threshold for global existence versus finite time blow-up for \eqref{mag-NLS}. 
	\end{remark}
	
	\begin{remark}
		In the case of no magnetic potential, this type of result was proved by J. Holmer and S. Roudenko \cite{HR}. They also proved that global solutions scatter to the linear ones as time tends to infinity. The later result on the scattering is not expected to hold in the presence of a constant magnetic field since Strichartz estimates associated to the magnetic Schr\"odinger operator are available only for finite times (see e.g., \cite{CE}). 
	\end{remark}
	
	\begin{theorem} \label{theo-dyn-at}
		Let $\frac{4}{3}<\alpha<4$. Let $u_0 \in \Sigma_A(\R^3)$ be such that
		\begin{align} \label{cond-ener-at}
		E_0(u_0) [M(u_0)]^{\sigc} = E^0(Q) [M(Q)]^{\sigc}.
		\end{align}
		\begin{itemize} [leftmargin=5mm]
			\item[(1)] If
			\begin{align} \label{cond-gwp-at}
			\|\nabla u_0\|_{L^2} \|u_0\|^{\sigc}_{L^2} < \|\nabla Q\|_{L^2} \|Q\|^{\sigc}_{L^2},
			\end{align}
			then the corresponding solution to \eqref{mag-NLS} exists globally in time.
			\item[(2)] If 
			\begin{align} \label{cond-blow-at}
			\|\nabla u_0\|_{L^2} \|u_0\|^{\sigc}_{L^2} > \|\nabla Q\|_{L^2} \|Q\|^{\sigc}_{L^2},
			\end{align}
			then the corresponding solution to \eqref{mag-NLS} either blows up in finite time, i.e., $T^*<\infty$, or there exist a time sequence $t_n \rightarrow \infty$ and $(y_n)_{n\geq 1} \subset \R^3$ such that
			\[
			u(t_n, \cdot+y_n) \rightarrow e^{i\theta}  \lambda Q \text{ strongly in } H^1(\R^3)
			\]
			for some $\theta \in \R$ and $\lambda = \frac{\|u_0\|_{L^2}}{\|Q\|_{L^2}}$ as $n\rightarrow \infty$. 
		\end{itemize}
	\end{theorem}
	
	\begin{remark}
		It was proved in Observation \ref{observation} that there is no data $u_0 \in \Sigma_A(\R^3)$ satisfying \eqref{cond-ener-at} and 
		\[
		\|\nabla u_0\|_{L^2} \|u_0\|^{\sigc}_{L^2} = \|\nabla Q\|_{L^2} \|Q\|^{\sigc}_{L^2}.
		\]
		Thus Theorem \ref{theo-dyn-at} give a description on long time behaviors of solutions to \eqref{mag-NLS} with initial data lying at the mass-energy threshold. 
	\end{remark}
	
	\begin{theorem} \label{theo-dyn-above}
		Let $\frac{4}{3}<\alpha<4$. Let $u_0 \in \Sigma_A(\R^3)$ be such that
		\begin{align} 
		E_0(u_0) [M(u_0)]^{\sigc} &\geq E^0(Q)[M(Q)]^{\sigc}, \label{cond-above-1} \\
		\frac{E_0(u_0)[M(u_0)]^{\sigc}}{E^0(Q) [M(Q)]^{\sigc}} &\left(1-\frac{(F'(u_0))^2}{8 E_0(u_0) F(u_0)} \right) \leq 1, \label{cond-above-2} 
		\end{align}
		and 
		\begin{align}
		\|u_0\|^{\alpha+2}_{L^{\alpha+2}} \|u_0\|^{2\sigc}_{L^2} &> \|Q\|^{\alpha+2}_{L^{\alpha+2}} \|Q\|^{2\sigc}_{L^2}, \label{cond-above-3} \\
		\ima \mathlarger{\int} x \cdot \nabla u_0(x) \overline{u}_0(x)dx  &\leq 0. \label{cond-above-4}
		\end{align}
		Then the corresponding solution to \eqref{mag-NLS} blows up in finite time, i.e., $T^*<\infty$.
	\end{theorem}
	
	\begin{remark}
		In Theorems \ref{theo-dyn-below}, \ref{theo-dyn-at}, and \eqref{theo-dyn-above}, we show the existence of finite time blow-up solutions to \eqref{mag-NLS} having $E_0(u_0) \geq 0$. Hence our results do not fall into the framework of the blow-up result proven recently by T. F. Kieffer and M. Loss (see Proposition \ref{prop-blow-2}).
	\end{remark}

	\subsection{Normalized standing waves}
	Next we are interested in the existence and stability of prescribed mass standing waves for \eqref{mag-NLS}. By standing waves, we mean solutions to \eqref{mag-NLS} of the form $u(t,x)=e^{i\omega t} \phi(x)$, where $\omega\in \R$ and $\phi$ is a solution to 
	\begin{align} \label{ell-equ}
	-(\nabla +iA)^2 \phi + \omega \phi - |\phi|^\alpha \phi =0.
	\end{align}
	The existence of standing waves for \eqref{mag-NLS} can be obtained by minimizing the energy functional $E(f)$ over the mass-constraint 
	\[
	S(c):= \left\{f \in H^1_A(\R^3) \ : \ M(f) =c\right\}
	\]
	with $c>0$. More precisely, we consider the minimization problem
	\[
	I(c):= \inf \left\{ E(f) \ : \ f \in S(c)\right\}.
	\]
	
	In the mass-subcritical case, the existence of minimizers for $I(c)$ was proved by M. J. Esteban and P.-L. Lions \cite{EL}. Moreover, the orbital stability of standing waves was showed by T. Cazenave and M. J. Esteban \cite{CE}. 
	
	\begin{proposition}[\cite{EL,CE}] \label{prop-exis-stab-mass-sub}
		Let $0<\alpha<\frac{4}{3}$. Then for any $c>0$, there exists a minimizer for $I(c)$. Moreover, the set
		\[
		\Mcal(c):= \left\{ \phi \in S(c) \ : \ E(\phi) = I(c)\right\}
		\]
		is orbitally stable under the flow of \eqref{mag-NLS} in the sense that for any $\vareps>0$, there exists $\delta>0$ such that for any initial data $u_0 \in H^1_A(\R^3)$ satisfying
		\[
		\inf_{\phi \in \Mcal(c)} \|u_0 - \phi\|_{H^1_A} \leq \delta,
		\]
		then the corresponding solution to \eqref{mag-NLS} exists globally in time and satisfies
		\[
		\inf_{\phi \in \Mcal(c)} \inf_{y\in \R^3} \|e^{iA(y) \cdot \cdotb} u(t,\cdotb+y) - \phi\|_{H^1_A} \leq \vareps, \quad \forall t\geq 0. 
		\]
	\end{proposition}
	
	In \cite{EL}, the existence of minimizers for $I(c)$ was claimed without proof and the proof was referred to \cite{Lions} for a similar argument using the concentration-compactness principle. However, an important point seems to be missing in order to preclude the vanishing scenario. In fact, if the vanishing occurs, then it is well-known (see \cite{Lions}) that the minimizing sequence $f_n \rightarrow 0$ strongly in $L^r(\R^3)$ for all $2<r<6$. Thus the mass-constraint, namely $M(f_n)=c>0$ for all $n\geq 1$, is not enough to rule out the vanishing. In the case of non magnetic potential, i.e., $A=0$, the vanishing can be precluded by using the fact that $I(c)<0$ for all $c>0$, which can be proved easily using a scaling argument. However, due to the appearance of the magnetic potential, this scaling argument does not work to show the negativity of $I(c)$. Indeed, it may happen that $I(c)$ is non-negative. 
	
	In this paper, we present an alternative simple method that avoids the concentration-compactness argument. Our main contributions in this direction are the following existence and stability in the mass-critical and mass-supercritical cases.
	
	\begin{theorem} \label{theo-exis-stab-mass-cri}
		Let $\alpha=\frac{4}{3}$. Then for any $0<c<M(Q)$, where $Q$ is the unique positive radial solution to \eqref{defi-Q}, there exists a minimizer for $I(c)$. Moreover, the set of minimizers for $I(c)$ is orbitally stable in the sense of Proposition \ref{prop-exis-stab-mass-sub}.
	\end{theorem}
	
	\begin{remark}
		The main difficulty in showing the existence of minimizers for $I(c)$ is the lack of compactness. To overcome it, the authors in \cite{EL} made use of a variant of the celebrated concentration-compactness principle adapted to the magnetic Sobolev space $H^1_A(\R^3)$. However, due to the non-negativity of $I(c)$, it is not clear from \cite{EL} how to exclude the vanishing possibility. Here we rule out the vanishing scenario by showing that every minimizing sequence for $I(c)$ has $L^{\alpha+2}$-norm bounded away from zero (see Lemma \ref{lem-lowe-boun}). This is done by using an $L^2$-bound of the magnetic-Sobolev norm (see \eqref{L2-bound}) and a suitable scaling argument. We refer to Section \ref{S4} for more details. 
	\end{remark}

	We also have the following non-existence results.
	
	\begin{proposition} \label{prop-non-exis} 
		\begin{itemize}[leftmargin=5mm]
			\item[(1)] Let $\alpha=\frac{4}{3}$. If $c\geq M(Q)$, where $Q$ is the unique positive radial solution to \eqref{defi-Q}, then there is no minimizer for $I(c)$. 
			\item[(2)] Let $\frac{4}{3}<\alpha<4$. Then for any $c>0$, there is no minimizer for $I(c)$.
		\end{itemize}
	\end{proposition}
	
	We are next interested in finding normalized solutions to \eqref{ell-equ} in the mass-supercritical case. By Proposition \ref{prop-non-exis}, we are not able to find minimizers for the energy functional under the mass-constraint $S(c)$. Inspired by a recent work of J. Bellazzini, N. Boussa\"id, L. Jeanjean, and N. Visciglia \cite{BBJV}, we consider the minimizing problem
	\[
	I^m(c):= \inf \left\{ E(f) \ : \ f \in S(c) \cap D(m)\right\},
	\] 
	where 
	\[
	D(m):= \left\{ f \in H^1_A(\R^3) \ : \ \|(\nabla+iA)f\|^2_{L^2} \leq m\right\}. 
	\]
	
	\begin{theorem} \label{theo-exis-stab-mass-sup}
		Let $\frac{4}{3}<\alpha<4$. Then for any $m>0$, there exists $c_0=c_0(m)>0$ sufficiently small such that:
		\begin{itemize}[leftmargin=5mm]
			\item[(1)] There exists a minimizer for $I^m(c)$ for all $0<c<c_0$. Moreover, the set of minimizers for $I^m(c)$ defined by
			\[
			\Mcal^m(c):= \left\{ \phi \in S(c) \cap D(m) \ : \ E(\phi) = I^m(c)\right\}
			\]
			satisfies
			\[
			\emptyset \ne \Mcal^m(c) \subset D(m/2).
			\]
			In particular, $\phi$ is a solution to \eqref{ell-equ} with $\omega$ the corresponding Lagrange multiplier. In addition, we have
			\begin{align} \label{est-omega}
			-|b|<\omega \leq -|b|\left(1- K c^{\frac{4-\alpha}{4}} m^{\frac{3\alpha-4}{4}} \right)
			\end{align}
			for some constant $K>0$ independent of $c$ and $m$. 
			\item[(2)] The set $\Mcal^m(c)$ with $0<c<c_0$ is orbitally stable under the flow of \eqref{mag-NLS} in the sense of Proposition \ref{prop-exis-stab-mass-sub}.
		\end{itemize}
	\end{theorem}
	
	\begin{remark}
		The proof of Theorem \ref{theo-exis-stab-mass-sup} is inspired by an idea of \cite{BBJV}. However, comparing to \cite{BBJV}, there are two main different points:
		\begin{itemize}[leftmargin=5mm]
			\item[(1)] In \cite{BBJV}, the existence of minimizers for $I^m(c)$ relies on the following inequality 
			\[
			\inf \left\{ E(f) \ : \ f \in S(c) \cap D(mc/2) \right\} < \inf \left\{ E(f) \ : \ f \in S(c) \cap \left( D(m) \backslash D(mc)\right)\right\}.
			\]
			Here the notation has been modified according to our definitions. If we use the above inequality, then for $f \in S(c) \cap D(mc/2)$, it follows from \eqref{L2-bound} that
			\[
			c=M(f) \leq \frac{1}{|b|} \|(\nabla+iA) f\|^2_{L^2} \leq \frac{mc}{2|b|}
			\]
			which yields $m \geq 2|b|$. Thus the argument of \cite{BBJV} does not apply to all $m>0$. Here our proof relies instead on the following inequality:
			\begin{align} \label{est-inf-intro}
			\inf \left\{E(f) \ : \ f \in S(c) \cap D(m/4)\right\} &< \inf \left\{ E(f) \ : \ f \in S(c) \cap \left( D(m) \backslash D(m/2)\right) \right\}.
			\end{align}
			We prove in Lemma \ref{lem-c0} that for any $m>0$, there exists $c_0=c_0(m)>0$ sufficiently small such that for all $0<c<c_0$, both 
			\[
			\left\{ f \ : \ f \in S(c) \cap D(m/4) \right\}, \quad \left\{ f \ : \ f \in  S(c) \cap \left( D(m) \backslash D(m/2)\right) \right\}
			\]
			are not empty and \eqref{est-inf-intro} holds. 
			\item[(2)] The orbital stability of normalized standing waves implicitly requires the solution exists globally in time. This global existence result was not showed in \cite{BBJV}. In Lemma \ref{lem-gwp-super}, we show a global existence result that supports the orbital stability given in Theorem \ref{theo-exis-stab-mass-sup}. The proof of this result is based on a standard continuity argument. 
		\end{itemize}
	\end{remark}
	
	Our next result shows that for a fixed constant $m>0$ and $c>0$ sufficiently small, minimizers of $I^m(c)$ are indeed normalized ground states related to \eqref{ell-equ}.
	 
	\begin{proposition} \label{prop-norm-grou-stat}
		Let $\frac{4}{3}<\alpha<4$. Let $m>0$ be a fixed constant, $c>0$ sufficiently small, and $\phi \in \Mcal^m(c)$. Then $\phi$ is a normalized ground state related to \eqref{ell-equ}, i.e., 
		\[
		\left.E'\right|_{S(c)}(\phi) =0, \quad  E(\phi) = \inf \left\{E(f) \ : \ f \in S(c), \left.E'\right|_{S(c)} (f)=0\right\}.
		\]
	\end{proposition}

	\subsection{Ground state standing waves}
	
	We are also interested in the existence and stability of ground state standing waves related to \eqref{mag-NLS}. Recall that a non-zero solution $\phi$ to \eqref{ell-equ} is called a ground state related to \eqref{ell-equ} if it minimizes the action functional
	\[
	S_\omega(f) :=E(f) + \frac{\omega}{2} M(f) = \frac{1}{2} \|(\nabla+iA)f\|^2_{L^2} + \frac{\omega}{2} \|f\|^2_{L^2} - \frac{1}{\alpha+2} \|f\|^{\alpha+2}_{L^{\alpha+2}}
	\]
	over all non-trivial solutions to \eqref{ell-equ}. Note that \eqref{ell-equ} can be written as $S'_\omega(\phi)=0$. Thus we denote the set of non-trivial solutions to \eqref{ell-equ} by
	\[
	\Ac(\omega):= \left\{ f\in H^1_A(\R^3) \ : \ S'_\omega(f)=0 \right\}
	\]
	and the set of ground states related to \eqref{ell-equ} by
	\[
	\Gc(\omega) := \left\{ \phi \in \Ac(\omega) \ : \ S_\omega(\phi) \leq S_\omega(f), \forall f \in \Ac(\omega)\right\}.
	\]
	Our last results concern with the existence of ground states related to \eqref{mag-NLS} and the strong instability of ground state standing waves in the mass-supercritical case. 
	
	\begin{theorem} \label{theo-exis-grou}
		Let $0<\alpha<4$ and $\omega>-|b|$. Then there exists a ground state related to \eqref{ell-equ}. Moreover, the set of ground states $\Gc(\omega)$ is characterized by 
		\[
		\Gc(\omega)= \left\{ \phi \in H^1_A(\R^3) \backslash \{0\} \ : \ S_\omega(\phi) = d(\omega), K_\omega(\phi)=0 \right\},
		\]
		where
		\begin{align} \label{defi-d-omega}
		d(\omega):= \inf\left\{ S_\omega(f) \ : \ f \in H^1_A(\R^3) \backslash \{0\}, K_\omega(f) =0 \right\}
		\end{align}
		with
		\begin{align} \label{defi-K-omega}
		K_\omega(f):= \|(\nabla+iA) f\|^2_{L^2} + \omega \|f\|^2_{L^2} - \|f\|^{\alpha+2}_{L^{\alpha+2}}.
		\end{align}
	\end{theorem}
	
	\begin{remark}
		In \cite[Section 4]{FO}, R. Fukuizumi and M. Ohta proved the existence of ground states related to \eqref{mag-NLS} in a subspace $H^1_{A,0}(\R^3)$ of $H^1_A(\R^3)$, namely
		\begin{align} \label{defi-H-A0}
		H^1_{A,0}(\R^3) = \left\{ f \in H^1(\R^3) \ : \ \rho f \in L^2(\R^3), f=f(\rho,z) \text{ does not depend on } \theta \right\},
		\end{align}
		where $(\rho,\theta, z)$ is the cylindrical coordinates in $\R^3$, i.e., $x_1=\rho \cos \theta, x_2=\rho \sin \theta$, and $x_3=z$. Our result extends the one in \cite[Section 4]{FO} to the whole energy space $H^1_A(\R^3)$.
	\end{remark}
	
	\begin{theorem} \label{theo-insta-grou}
		Let $\frac{4}{3}<\alpha<4$, $\omega>-|b|$, and $\phi \in \Gc(\omega)$. If $\left. \partial^2_\lambda S_\omega(\phi^\lambda)\right|_{\lambda=1} \leq 0$, where 
		\begin{align} \label{defi-phi-lambda}
		\phi^\lambda(x)=\lambda^{\frac{3}{2}} \phi(\lambda x),
		\end{align} 
		then the ground state standing wave $e^{i\omega t} \phi(x)$ is strongly unstable by blow-up in the sense that for any $\vareps>0$, there exists $u_0 \in \Sigma_A(\R^3)$ such that $\|u_0-\phi\|_{\Sigma_A} <\vareps$ and the corresponding solution to \eqref{mag-NLS} with initial data $\left.u\right|_{t=0} =u_0$ blows up in finite time.
	\end{theorem}
	
	\begin{remark}
		In \cite{FO}, the orbital instability of ground state standing waves for \eqref{mag-NLS} in the subspace $H^1_{A,0}(\R^3)$ (see \eqref{defi-H-A0}) was proven (see also \cite{Ribeiro-insta} for an earlier similar result). Here we extend their results and show the strong instability of ground state standing waves for \eqref{mag-NLS}. 
	\end{remark}
	
	We end the introduction by reporting some recent results related to magnetic nonlinear Schr\"odinger equations. After the pioneering works of M. J. Esteban and P.-L. Lions \cite{EL} and T. Cazenave and M. J. Esteban \cite{CE}, the nonlinear Schr\"odinger equations (NLS) with magnetic potential has attracted much of interest in the last decades. For the time-dependent magnetic NLS with an external potential, we mention the works of L. Fanelli and L. Vega \cite{FV} and P. D'Ancona, L. Fanelli, and L. Vega \cite{DFV} on virial identities and Strichartz estimates; A. Garcia \cite{Garcia} for the existence of finite time blow-up solutions; J. Colliander, M. Czubak, and J. Lee \cite{CCL} for the interaction Morawetz estimate and its application to the global existence theory. For the time-independent magnetic NLS with potential, we refer to the works of G. Arioli and A. Szulkin \cite{AS} and J. Chabrowski and A. Szulkin \cite{CS} for the existence and qualitative properties of ground state solutions; K. Kurata \cite{Kurata}, S. Cingolani \cite{Cingolani}, S. Cingolani and S. Secchi \cite{CS}, S. Cingolani, L. Jeanjean, and S. Secchi \cite{CJS}, and C. O. Alves, G. M. Figueiredo, and M. F. Furtado \cite{AFF} for the existence of semiclassical solutions.

	This paper is organized as follows. In Section \ref{S2}, we recall some basic properties of the magnetic Sobolev space $H^1_A(\R^3)$ and prove some preliminary results which are needed in the sequel. Section \ref{S3} is devoted to long time dynamics such as global existence and finite time blow-up of solutions to \eqref{mag-NLS}. In Section \ref{S4}, we study the existence and orbital stability of normalized standing waves related to \eqref{mag-NLS}. Finally, the existence and strong instability of ground state standing waves will be investigated in Section \ref{S5}. 
	
	\section{Preliminaries}
	\label{S2}
	\setcounter{equation}{0}
	In this section, we recall some basic properties of the magnetic Sobolev space $H^1_A(\R^3)$ and prove some preliminary results which are needed in the sequel.  
	
	\begin{lemma} [\cite{EL}]
		Let $A \in L^2_{\loc}(\R^3, \R^3)$. Then $H^1_A(\R^3)$ equipped with the inner product 
		\[
		\scal{f,g}_{H^1_A}:= \int f \overline{g} dx + \int (\nabla+iA)f \cdot \overline{(\nabla+iA) g} dx 
		\]
		is a Hilbert space.
	\end{lemma}
	
	\begin{lemma} [Diamagnetic inequality \cite{LL}]
		Let $A\in L^2_{\loc}(\R^3,\R^3)$ and $f \in H^1_A(\R^3)$. Then $|f| \in H^1(\R^3)$. In particular, we have
		\begin{align} \label{diag-ineq}
		|\nabla |f|(x)| \leq |(\nabla+iA)f(x)| \quad \text{a.e. } x \in \R^3.
		\end{align}
	\end{lemma}
	
	\begin{lemma}[\cite{EL}] \label{lem-prop-H1-A}
		Let $A \in L^2_{\loc}(\R^3,\R^3)$. Then the following properties hold:
		\begin{itemize}[leftmargin=5mm]
			\item[(1)] $C^\infty_0(\R^3)$ is dense in $H^1_A(\R^3)$.
			\item[(2)] $H^1_A(\R^3)$ is continuously embedded in $L^r(\R^3)$ for all $2\leq r\leq 6$. 
			\item[(3)] Assume that $A$ is linear, i.e., $A(x+y) =A(x)+A(y)$ for all $x, y \in \R^3$. Let $y \in \R^3$, $f \in H^1_A(\R^3)$, and set
			\[
			\tilde{f}(x):= e^{iA(y)\cdot x} f(x+y), \quad x \in \R^3.
			\] 
			Then $(\nabla+iA) \tilde{f}(x) = e^{iA(y) \cdot x} (\nabla+iA)f(x+y)$. In particular, 
			\[
			\|(\nabla+iA)\tilde{f}\|_{L^2}=\|(\nabla+iA)f\|_{L^2}.
			\]
			\item[(4)] If $A \in L^3_{\loc}(\R^3,\R^3)$, then $H^1_A(\R^3)$ is continuously embedded in $H^1_{\loc}(\R^3)$. In particular, $H^1_A(\R^3)$ is compactly embedded in $L^r_{\loc}(\R^3)$ for all $2\leq r<6$. 
		\end{itemize}
	\end{lemma}
	
	\begin{lemma} [\cite{AHS-1}]
		Let $A \in W^{1,\infty}_{\loc}(\R^3, \R^3)$ and $j, k \in \{1, \cdots, 3\}$. Then for any $f \in C^\infty_0(\R^3)$, we have
		\[
		\left|\int(\partial_j A_k - \partial_k A_j) f \overline{f} dx\right| \leq \|(\partial_j + iA_j) f\|^2_{L^2} + \|(\partial_k+iA_k) f\|^2_{L^2}.
		\]
		In particular, if $A$ is as in \eqref{defi-A}, then 
		\begin{align} \label{L2-bound}
		|b|\|f\|^2_{L^2} \leq \|(\nabla +iA) f\|^2_{L^2}.
		\end{align}
	\end{lemma}
	
	\begin{lemma} \label{lem-refi-fatou}
		Let $A \in L^2_{\loc}(\R^3, \R^3)$ and $(f_n)_{n\geq 1}$ be a bounded sequence in $H^1_A(\R^3)$. Assume that $f_n \rightharpoonup f$ weakly in $H^1_A(\R^3)$. Then we have
		\begin{align*}
		\|(\nabla+iA) f_n\|^2_{L^2} &= \|(\nabla+iA)f\|^2_{L^2} + \|(\nabla+iA) (f_n-f)\|^2_{L^2} + o_n(1), \\
		\|f_n\|^r_{L^r} &= \|f\|^r_{L^r} + \|f_n-f\|^r_{L^r} + o_n(1), \quad 2\leq r \leq 6.
		\end{align*}
	\end{lemma}
	
	\begin{proof}
		As $H^1_A(\R^3)$ is continuously embedding in $L^r(\R^3)$ for all $2\leq r\leq 6$. The second identity is a direct consequence of the refined Fatou's lemma due to H. Br\'ezis and  E. H. Lieb \cite{BL}. Let us prove the first identity. Set $g_n:= f_n -f$. We see that $g_n \rightharpoonup 0$ weakly in $H^1_A(\R^3)$. We compute
		\begin{align*}
		\|(\nabla +iA) f_n\|^2_{L^2} &= \|(\nabla +iA) (f+g_n)\|^2_{L^2} \\
		&= \|(\nabla+iA) f\|^2_{L^2} + \|(\nabla +iA) g_n\|^2_{L^2} + 2 \rea \int \overline{(\nabla+iA) f} \cdot (\nabla+iA) g_n dx.
		\end{align*}
		Let $\epsilon>0$. Since $C^\infty_0(\R^3)$ is dense in $H^1_A(\R^3)$, we take $\varphi \in C^\infty_0(\R^3)$ so that $\|(\nabla+iA)(f-\varphi)\|_{L^2} <\epsilon/2C$, where $C:=\sup_{n\geq 1} \|g_n\|_{H^1_A} <\infty$. Since $g_n \rightharpoonup 0$ weakly in $H^1_A(\R^3)$, we see that
		\[
		\left|\int \overline{(\nabla+iA) \varphi} \cdot (\nabla+iA) g_n dx\right| \rightarrow 0 \text{ as } n\rightarrow \infty.
		\]
		Thus there exists $n_0 \in \N$ such that for $n\ge n_0$,
		\begin{align*}
		\Big|\int \overline{(\nabla+iA) f} &\cdot(\nabla+iA) g_n dx\Big|\\ 
		&\leq \left|\int \overline{(\nabla+iA) (f-\varphi)} \cdot (\nabla+iA) g_n dx\right| + 	\left|\int \overline{(\nabla+iA) \varphi} \cdot (\nabla +iA) g_n dx\right| \\
		&\leq \|(\nabla+iA)(f-\varphi)\|_{L^2} \|(\nabla+iA) g_n\|_{L^2} + \epsilon/2 <\epsilon.
		\end{align*}
		The proof is complete.
	\end{proof}

	\begin{lemma} \label{lem-weak-conv}
		Let $A \in L^3_{\loc}(\R^3, \R^3)$ be linear. Let $(f_n)_{n\geq 1}$ be a bounded sequence in $H^1_A(\R^3)$, i.e., $\sup_{n\geq 1} \|f_n\|_{H^1_A} <\infty$. Assume that there exists $\vareps_0>0$ such that
		\begin{align} \label{lowe-bound-fn-1}
		\inf_{n\geq 1} \|f_n\|_{L^r} \geq \vareps_0
		\end{align}
		for some $2<r<6$. Then up to a subsequence, there exist $f \in H^1_A(\R^3)\backslash \{0\}$ and $(y_n)_{n\geq 1} \subset \R^3$ such that
		\[
		e^{iA(y_n) \cdot x} f_n(x+y_n) \rightharpoonup f \text{ weakly in } H^1_A(\R^3).
		\]
	\end{lemma}
	
	\begin{proof}
		The proof is based on an argument of \cite[Lemma 3.4]{BBJV}. By interpolation, we infer from \eqref{lowe-bound-fn-1} that
		\begin{align} \label{lowe-bound-fn-2}
		\inf_{n\geq 1} \|f_n\|_{L^{\frac{10}{3}}} \geq \vareps_1 >0.
		\end{align}
		By the Sobolev embedding
		\[
		\|f\|^{\frac{10}{3}}_{L^{\frac{10}{3}}(Q_k)} \leq C \|f\|^{\frac{4}{3}}_{L^2(Q_k)} \|f\|^2_{H^1(Q_k)},
		\]
		where
		\[
		Q_k:= (k, k+1)^3, \quad k \in \Z
		\]
		and taking the sum over $k\in \Z$, we get
		\[
		\|f\|^{\frac{10}{3}}_{L^{\frac{10}{3}}} \leq C\left(\sup_{k\in \Z} \|f\|_{L^2(Q_k)} \right)^{\frac{4}{3}} \|f\|^2_{H^1}.
		\]
		Replacing $f$ by $|f|$ and using the diamagnetic inequality \eqref{diag-ineq}, we have
		\begin{align*}
		\|f\|^{\frac{10}{3}}_{L^{\frac{10}{3}}} &\leq C\left( \sup_{k\in \Z} \|f\|_{L^2(Q_k)} \right)^{\frac{4}{3}} \left( \|\nabla |f|\|^2_{L^2} + \|f\|^2_{L^2} \right) \\
		&\leq C\left( \sup_{k\in \Z} \|f\|_{L^2(Q_k)} \right)^{\frac{4}{3}} \left(\|(\nabla +iA) f\|^2_{L^2} +\|f\|^2_{L^2} \right).
		\end{align*}
		Applying the above inequality to $f_n$ and using \eqref{lowe-bound-fn-2} together with the fact that $\sup_{n\geq 1} \|f_n\|_{H^1_A} <\infty$, there exists $(k_n)_{n\geq 1} \subset \Z$ such that
		\[
		\inf_{n\geq 1} \|f_n\|_{L^2(Q_{k_n})} \geq C
		\]
		for some constant $C>0$. Set $y_n=(-k_n, -k_n, -k_n)$ and
		\[
		\tilde{f}_n(x):= e^{iA(y_n) \cdot x} f_n(x+y_n).
		\]
		By Lemma \ref{lem-prop-H1-A}, we have
		\[
		\|\tilde{f}_n\|^2_{L^2}=\|f_n\|^2_{L^2}, \quad \|(\nabla+iA) \tilde{f}_n\|^2_{L^2} = \|(\nabla +iA) f_n\|^2_{L^2}.
		\]
		Thus we get $\sup_{n\geq 1} \|\tilde{f}_n\|_{H^1_A} <\infty$ and 
		\begin{align} \label{lowe-bound-gn}
		\inf_{n\geq 1} \|\tilde{f}_n\|^2_{L^2(Q_0)} \geq C>0.
		\end{align}
		Since the embedding $H^1_A(\R^3) \hookrightarrow L^2(Q_0)$ is compact (see again Lemma \ref{lem-prop-H1-A}), there exist $f \in H^1_A(\R^3)$ and a subsequence still denoted by $(\tilde{f}_n)_{n\geq 1}$ such that $\tilde{f}_n \rightharpoonup f$ weakly in $H^1_A(\R^3)$ and $\tilde{f}_n \rightarrow f$ strongly in $L^2(Q_0)$. By \eqref{lowe-bound-gn}, we have $f \ne 0$. The proof is complete.
	\end{proof}

	\section{Global existence and finite time blow-up}
	\label{S3}
	\setcounter{equation}{0}
	In this section, we study the existence of global in time and finite time blow-up solutions to \eqref{mag-NLS}. Let us start with the following result.
	
	\begin{remark} \label{rem-equi-norm}
		Let $A$ be as in \eqref{defi-A}. Then $\Sigma_A(\R^3) \equiv \Sigma(\R^3)$, where $\Sigma_A(\R^3)$ and $\Sigma(\R^3)$ are defined as in \eqref{Sigma-A} and \eqref{Sigma-A} respectively.
	\end{remark}
	
	\begin{proof}
		We first recall the following identity due to \cite{Ribeiro}:
		\begin{align} \label{iden-mag}
		\|(\nabla+iA)f\|^2_{L^2}= \|\nabla f\|^2_{L^2} -2 \rea \int (\nabla+iA)f \cdot \overline{i A f} dx - \|Af\|^2_{L^2},
		\end{align}
		From \eqref{iden-mag}, we have from H\"older's and Cauchy-Schwarz' inequalities that
		\begin{align*}
		\|\nabla f\|^2_{L^2} \leq 2 \left(\|(\nabla+iA)f\|^2_{L^2} + \|Af\|^2_{L^2} \right) \leq C(b) \|f\|^2_{\Sigma_A},
		\end{align*}
		hence $\Sigma_A(\R^3)\subset \Sigma(\R^3)$. On the other hand, by \eqref{iden-mag}, we have
		\begin{align*}
		\|(\nabla+iA)f\|^2_{L^2} &\leq \|\nabla f\|^2_{L^2} + 2\|(\nabla+iA) f\|_{L^2} \|Af\|_{L^2} + \|Af\|^2_{L^2} \\
		&\leq \|\nabla f\|^2_{L^2} + \frac{1}{2} \|(\nabla+iA)f\|^2_{L^2} + 3 \|Af\|^2_{L^2}
		\end{align*}
		which implies that 
		\[
		\|(\nabla+iA)f\|^2_{L^2} \leq 2\|\nabla f\|^2_{L^2} + 6 \|Af\|^2_{L^2} \leq C(b) \|f\|^2_{\Sigma},
		\]
		so $\Sigma(\R^3) \subset \Sigma_A(\R^3)$. The proof is complete.		
	\end{proof}
	
	We next prove the angular momentum conservation given in Lemma \ref{lem-cons-angu}. 
	
	\begin{proof} [Proof of Lemma \ref{lem-cons-angu}]
		The proof is essentially given in \cite{KL}. For the reader's convenience, we recall some details. We first observe that
		\begin{align} \label{prop-Lz}
		\int L_z f g dx = - \int L_z g f dx,  \quad \overline{L_z f} = - L_z \overline{f}
		\end{align}
		which yields
		\[
		R(f) = \int L_z f \overline{f} dx = - \int f L_z \overline{f} dx = \int f \overline{L_z f} dx = \overline{R(f)}
		\]
		or $R(f)$ is real-valued. We next give formal computations to show the conservation of angular momentum. The rigorous proof needs the standard approximation argument (see \cite{Ribeiro}) and we omit the details. From \eqref{mag-norm}, \eqref{defi-R}, and \eqref{prop-Lz}, we have
		\begin{align} \label{nabla-A}
		(\nabla +iA)^2 = \Delta - b L_z - \frac{b^2}{4}\rho^2
		\end{align}
		and $[\Delta, L_z]=0$. It follows that
		\begin{align*}
		\frac{d}{dt} R(u(t))&= \int L_z \partial_t u(t) \overline{u}(t) dx + i \int L_z u(t) \partial_t \overline{u}(t) dx \\
		&= \int L_z (i \Delta u - i b L_z u - i\frac{b^2}{4} \rho^2 u +i|u|^\alpha u) \overline{u} dx \\
		&\mathrel{\phantom{=}}+ \int L_z u (-i\Delta \overline{u} - i b L_z \overline{u} + i \frac{b^2}{4}\rho^2 \overline{u} - i |u|^\alpha \overline{u}) dx \\
		&= i \int L_z \Delta u \overline{u} - L_z u \Delta \overline{u} dx - ib \int L^2_z u \overline{u} + L_z u L_z \overline{u} dx \\
		&\mathrel{\phantom{=}} -i \frac{b^2}{4}\int L_z(\rho^2 u) \overline{u} - L_z u \rho^2 \overline{u} dx + i \int L_z(|u|^\alpha u) \overline{u} - L_zu |u|^\alpha \overline{u} dx \\
		&= (1) +(2) + (3) + (4).
		\end{align*}
		We see that
		\[
		(1) = i \int L_z \Delta u \overline{u} - \Delta L_z u \overline{u} dx = i \int [L_z, \Delta] u \overline{u} dx =0
		\]
		and
		\[
		(2) = ib \int L^2_z u \overline{u} + L_z u L_z \overline{u} dx = ib \int L_z u L_z \overline{u} - L_z u L_z \overline{u} dx =0.
		\]
		As $L_z(\rho^2 u) = \rho^2 L_z u$, we readily see that $(3)=0$. Here we have used the fact that $L_z = -i \partial_\theta$, where where $x_1=\rho \cos \theta$ and $x_2 =\rho \sin \theta$ with $\rho=\sqrt{x_1^2+x_2^2}$ and $\theta \in [0,2\pi)$. Finally, we have
		\[
		L_z(|u|^{\alpha+2}) = -i\partial_\theta(|u|^\alpha u \overline{u}) = -i \partial_\theta(|u|^\alpha u) \overline{u} -i |u|^\alpha u \partial_\theta \overline{u} = L_z(|u|^\alpha u) \overline{u} + |u|^\alpha u L_z \overline{u}
		\]
		which shows that
		\[
		(4) = i \int L_z(|u|^{\alpha+2}) -  |u|^\alpha u L_z \overline{u} - L_z u |u|^\alpha \overline{u} dx = - i \int |u|^\alpha (u L_z \overline{u} + \overline{u} L_z u) dx.
		\]
		On the other hand, we have
		\[
		L_z(|u|^{\alpha+2}) = (\alpha+2) |u|^{\alpha+1} L_z(|u|) =\frac{\alpha+2}{2} |u|^\alpha ( u L_z \overline{u} + \overline{u} L_z u),
		\]
		where $L_z(|u|^2) = 2|u|L_z(|u|)= u L_z \overline{u} + \overline{u}L_z u$. It follows that
		\[
		(4) = \frac{-2i}{\alpha+2} \int L_z(|u|^{\alpha+2}) dx =\frac{-2}{\alpha+2}\int \partial_\theta(|u|^{\alpha+2})dx=0.
		\]
		The proof is complete.
	\end{proof}
	
	We next recall the following virial identity related to \eqref{mag-NLS} (see e.g., \cite[Theorem 1.2]{Ribeiro}) which plays an important role in proving the existence of finite time blow-up solutions.
	 
	\begin{lemma}[\cite{Ribeiro}] \label{lem-viri-iden}
		Let $0<\alpha<4$ and $u_0 \in \Sigma_A(\R^3)$. Let $u:[0,T^*) \times \R^3 \rightarrow \C$ be the corresponding solution to \eqref{mag-NLS}. Set
		\begin{align} \label{defi-F}
		F(u(t)):= \int |x|^2 |u(t,x)|^2 dx.
		\end{align}
		Then the function $[0,T^*) \ni t \mapsto F(u(t))$ is in $C^2([0,T^*))$ and 
		\begin{align*}
		F'(u(t)) &
		= 4 \ima \int x \cdot \nabla u(t,x) \overline{u}(t,x) dx,\\
		F''(u(t)) &= 8 \|\nabla u(t)\|^2_{L^2} -2 b^2 \|\rho u(t)\|^2_{L^2} - \frac{12\alpha}{\alpha+2} \|u(t)\|^{\alpha+2}_{L^{\alpha+2}},
		\end{align*}
		for all $t\in [0,T^*)$.
	\end{lemma}
	
	\begin{proof} [Proof of Proposition \ref{prop-blow-2}]
		Let $F(u(t))$ is as in \eqref{defi-F}. By Lemma \ref{lem-viri-iden}, we have
		\begin{align*}
		F''(u(t)) &= 8 \|\nabla u(t)\|^2_{L^2} - 2b^2 \|\rho u(t)\|^2_{L^2} - \frac{12\alpha}{\alpha+2} \|u(t)\|^{\alpha+2}_{L^{\alpha+2}} \\
		&=16 E_0(u(t)) -4b^2 \|\rho u(t)\|^2_{L^2} - \frac{4(3\alpha-4)}{\alpha+2} \|u(t)\|^{\alpha+2}_{L^{\alpha+2}}
		\end{align*}
		for all $t\in [0,T^*)$. By Lemma \ref{lem-cons-angu}, we infer that
		\[
		F''(u(t)) \leq 16 E_0(u_0), \quad \forall t\in [0,T^*).
		\]
		Integrating this inequality, we get
		\[
		F(u(t)) \leq \|xu_0\|^2_{L^2} + 4 \left(\ima \int x \cdot \nabla u_0(x) \overline{u}_0(x)dx\right) t + 8 E_0(u_0) t^2, \quad \forall t\in [0,T^*).
		\]
		If one of the conditions given in Proposition \ref{prop-blow-2} holds, then there exists $t_1>0$ such that $F(u(t_1))<0$ which is a contradiction. The proof is complete.
	\end{proof}
	
	Now we prove the sharp threshold for global existence versus blow-up for \eqref{mag-NLS} in the mass-critical case given in Proposition \ref{prop-thre-mass}.
	\begin{proof} [Proof of Proposition \ref{prop-thre-mass}]
		(1) By the Gagliardo-Nirenberg inequality and the diamagnetic inequality \eqref{diag-ineq}, we have
		\begin{align} \label{mag-GN-ineq-mass-cri}
		\|f\|^{\frac{10}{3}}_{L^{\frac{10}{3}}} \leq \frac{5}{3} \left(\frac{\|f\|_{L^2}}{\|Q\|_{L^2}} \right)^{\frac{4}{3}} \|\nabla |f|\|^2_{L^2} \leq \frac{5}{3} \left(\frac{\|f\|_{L^2}}{\|Q\|_{L^2}} \right)^{\frac{4}{3}} \|(\nabla+iA)f\|^2_{L^2},
		\end{align}
		where $Q$ is the unique positive radial solution to \eqref{defi-Q} with $\alpha=\frac{4}{3}$. From this inequality and the conservation laws of mass and energy, we infer that 
		\begin{align*}
		E(u_0) = E(u(t)) &\geq \frac{1}{2} \|(\nabla +iA) u(t)\|^2_{L^2} - \frac{1}{2} \left(\frac{\|u(t)\|_{L^2}}{\|Q\|_{L^2}} \right)^{\frac{4}{3}} \|(\nabla+iA)u(t)\|^2_{L^2} \\
		&= \frac{1}{2} \left(1-\left(\frac{\|u_0\|_{L^2}}{\|Q\|_{L^2}}\right)^{\frac{4}{3}}\right) \|(\nabla +iA) u(t)\|^2_{L^2}
		\end{align*}
		for all $t\in [0,T^*)$. As $\|u_0\|_{L^2}<\|Q\|_{L^2}$, we have $\sup_{t\in [0,T^*)} \|(\nabla +iA) u(t)\|_{L^2} \leq C$ which, by the blow-up alternative, implies that $T^*=\infty$. 
		
		(2) Let $c >\|Q\|_{L^2}$. We define
		\[
		u_0 (x):= a \lambda^{\frac{3}{2}} Q(\lambda x),
		\]
		where $a:= \frac{c}{\|Q\|_{L^2}}>1$ and $\lambda>0$ will be chosen later. As $Q$ decays exponentially at infinity, it is clear that $Q \in \Sigma_A(\R^3)$ (see also Remark \ref{rem-equi-norm}). Moreover, we have
		\begin{align*}
		\|u_0\|^2_{L^2} &= a^2 \|Q\|^2_{L^2}=c^2, & \|\nabla u_0\|^2_{L^2} &= a^2\lambda^2  \|\nabla Q\|^2_{L^2}, \\ \|u_0\|^{\frac{10}{3}}_{L^{\frac{10}{3}}} &= a^{\frac{10}{3}} \lambda^2 \|Q\|^{\frac{10}{3}}_{L^{\frac{10}{3}}}, & \|\rho u_0\|^2_{L^2} &= a^2 \lambda^{-2} \|\rho Q\|^2_{L^2}.
		\end{align*}
		It follows that
		\begin{align*}
		E_0(u_0) &= \frac{1}{2}\|\nabla u_0\|^2_{L^2} + \frac{b^2}{8} \|\rho u_0\|^2_{L^2} -\frac{3}{10} \|u_0\|^{\frac{10}{3}}_{L^{\frac{10}{3}}} \\
		&= a^2 \lambda^2 \left(\frac{1}{2} \|\nabla Q\|^2_{L^2} + \frac{b^2}{2} \lambda^{-4} \|\rho Q\|^2_{L^2} -\frac{3}{10} a^{\frac{4}{3}} \|Q\|^{\frac{10}{3}}_{L^{\frac{10}{3}}}\right).
		\end{align*}
		Using the Pohozaev's identity (see e.g., \cite{Cazenave}): 
		\begin{align} \label{poho-iden-mass}
		\|\nabla Q\|_{L^2}^2 = \frac{3}{5} \|Q\|^{\frac{10}{3}}_{L^{\frac{10}{3}}} = \frac{3}{2} \|Q\|^2_{L^2},
		\end{align}
		we infer that
		\[
		E_0(u_0) = a^2\lambda^2 \left(\frac{b^2}{2} \lambda^{-4} \|\rho Q\|^2_{L^2} - \frac{3}{10} \left(a^{\frac{4}{3}}-1\right) \|Q\|^{\frac{10}{3}}_{L^{\frac{10}{3}}} \right).
		\]
		Taking $\lambda>0$ sufficiently large, we have $E_0(u_0)<0$. By Proposition \ref{prop-blow-2}, the corresponding solution to \eqref{mag-NLS} with initial data $\left. u\right|_{t=0} = u_0$ blows up in finite time. The proof is complete. 
	\end{proof}
	
	Before studying the global existence and finite time blow-up for \eqref{mag-NLS} in the mass-supercritical case, let us recall the following properties of the unique positive radial solution $Q$ to \eqref{defi-Q}. 
	
	\begin{remark} 
		Let $\frac{4}{3}<\alpha<4$ and $Q$ be the unique positive radial solution to \eqref{defi-Q}. It is well-known that $Q$ optimizes the Gagliardo-Nirenberg inequality
		\begin{align} \label{GN-ineq-super}
		\|f\|^{\alpha+2}_{L^{\alpha+2}} \leq C_{\opt} \|\nabla f\|^{\frac{3\alpha}{2}}_{L^2} \|f\|^{\frac{4-\alpha}{2}}_{L^2}, \quad \forall f\in H^1(\R^3).
		\end{align}
		In particular, we have
		\[
		C_{\opt} = \|Q\|^{\alpha+2}_{L^{\alpha+2}} \div \left[\|\nabla Q\|^{\frac{3\alpha}{2}}_{L^2} \|Q\|^{\frac{4-\alpha}{2}}_{L^2} \right].
		\]
		Thanks to the following Pohozaev's identities (see e.g., \cite{Cazenave}):
		\begin{align} \label{poho-Q-super}
		\|Q\|^2_{L^2} = \frac{4-\alpha}{3\alpha} \|\nabla Q\|^2_{L^2} = \frac{4-\alpha}{2(\alpha+2)} \|Q\|^{\alpha+2}_{L^{\alpha+2}},
		\end{align}
		we have
		\begin{align} \label{proper-E0-Q}
		E^0(Q) [M(Q)]^{\sigc} = \frac{3\alpha-4}{6\alpha} \left(\|\nabla Q\|_{L^2} \|Q\|^{\sigc}_{L^2}\right)^2, \quad C_{\opt} = \frac{2(\alpha+2)}{3\alpha} \left( \|\nabla Q\|_{L^2} \|Q\|_{L^2}^{\sigc}\right)^{-\frac{3\alpha-4}{2}}.
		\end{align}
	\end{remark}
 
	We are now able to prove the global existence given in Propositions \ref{prop-gwp-1} and \ref{prop-gwp-2}.
	
	\begin{proof} [Proof of Proposition \ref{prop-gwp-1}]
		Let $u:[0,T^*)\times \R^3 \rightarrow \C$ be the corresponding solution to \eqref{mag-NLS}. By the Gagliardo-Nirenberg inequality \eqref{GN-ineq-super} and the diamagnetic inequality \eqref{diag-ineq}, we have 
		\[
		\|f\|^{\alpha+2}_{L^{\alpha+2}} \leq C_{\opt} \|\nabla |f|\|^{\frac{3\alpha}{2}}_{L^2} \|f\|^{\frac{4-\alpha}{2}}_{L^2} \leq C_{\opt} \|(\nabla + iA) f\|^{\frac{3\alpha}{2}}_{L^2} \|f\|^{\frac{4-\alpha}{2}}_{L^2}, \quad \forall f\in H^1_A(\R^3).
		\]
		It follows that
		\begin{align*}
		E(u(t)) &[M(u(t))]^{\sigc} \\
		&\geq \frac{1}{2} \left(\|(\nabla+iA) u(t)\|_{L^2} \|u(t)\|^{\sigc}_{L^2}\right)^2 - \frac{C_{\opt}}{\alpha+2} \| (\nabla +iA)u(t)\|^{\frac{3\alpha}{2}}_{L^2} \|u(t)\|^{\frac{4-\alpha}{2} + 2\sigc}_{L^2} \nonumber \\
		&= G\left( \|(\nabla+iA) u(t)\|_{L^2} \|u(t)\|^{\sigc}_{L^2}\right) 
		\end{align*}
		for all $t\in [0,T^*)$, where
		\begin{align} \label{defi-G}
		G(\lambda):= \frac{1}{2}\lambda^2 - \frac{C_{\opt}}{\alpha+2} \lambda^{\frac{3\alpha}{2}}.
		\end{align}
		Using \eqref{poho-Q-super} and \eqref{proper-E0-Q}, we see that 
		\begin{align} \label{iden-G}
		G\left( \|\nabla Q\|_{L^2} \|Q\|^{\sigc}_{L^2}\right) = \frac{3\alpha-4}{6\alpha} \left( \|\nabla Q\|_{L^2} \|Q\|^{\sigc}_{L^2}\right)^2 = E^0(Q) [M(Q)]^{\sigc}.
		\end{align}
		By the conservation of energy and \eqref{cond-ener-below}, we have
		\begin{align}
		G\left( \|(\nabla+iA) u(t)\|_{L^2} \|u(t)\|^{\sigc}_{L^2}\right) &\leq E(u_0) [M(u_0)]^{\sigc} \nonumber\\
		&<E^0(Q)[M(Q)]^{\sigc}= G\left( \|\nabla Q\|_{L^2} \|Q\|^{\sigc}_{L^2}\right)  \label{est-G}
		\end{align}
		for all $t\in [0,T^*)$. By \eqref{cond-gwp-below}, the continuity argument implies
		\[
		\|(\nabla+iA) u(t)\|_{L^2} \|u(t)\|^{\sigc}_{L^2} < \|\nabla Q\|_{L^2} \|Q\|^{\sigc}_{L^2}
		\]
		for all $t\in [0,T^*)$. This estimate together with the blow-up alternative and the conservation of mass yield $T^*=\infty$. 
	\end{proof}
	
	\begin{proof} [Proof of Proposition \ref{prop-gwp-2}]
		It suffices to prove that
		\[
		\|(\nabla +iA) u(t)\|_{L^2} \|u(t)\|_{L^2}^{\sigc} < \|\nabla Q\|_{L^2} \|Q\|^{\sigc}_{L^2}
		\]
		for all $t\in [0,T^*)$. This together with the blow-up alternative shows that $T^*=\infty$. Assume by contradiction that there exists $t_0>0$ such that 
		\[
		\|(\nabla +iA) u(t_0)\|_{L^2} \|u(t_0)\|_{L^2}^{\sigc} \geq \|\nabla Q\|_{L^2} \|Q\|^{\sigc}_{L^2}.
		\]
		By \eqref{cond-gwp-2-2} and the continuity argument, there exists $t_1 \in (0, t_0]$ such that
		\[
		\|(\nabla +iA) u(t_1)\|_{L^2} \|u(t_1)\|_{L^2}^{\sigc} = \|\nabla Q\|_{L^2} \|Q\|^{\sigc}_{L^2}.
		\]
		Denote $f= u(t_1)$. We have from \eqref{cond-gwp-2-1} and the conservation laws of mass and energy that
		\[
		E(f)[M(f)]^{\sigc} = E^0(Q) [M(Q)]^{\sigc}, \quad \|(\nabla+iA)f\|_{L^2} \|f\|^{\sigc}_{L^2} = \|\nabla Q\|_{L^2} \|Q\|^{\sigc}_{L^2}.
		\]
		We take $\lambda>0$ such that $\|f\|_{L^2} = \lambda \|Q\|_{L^2}$. It follows that
		\[
		E(f) = \lambda^{-2\sigc} E^0(Q), \quad \|(\nabla +iA)f\|_{L^2} = \lambda^{-\sigc}\|\nabla Q\|_{L^2}
		\]
		which yields
		\begin{align*}
		\|f\|^{\alpha+2}_{L^{\alpha+2}} &= (\alpha+2) \left(\frac{1}{2} \|(\nabla +iA)f\|^2_{L^2} - E(f) \right) \\
		&=(\alpha+2) \left(\frac{\lambda^{-2\sigc}}{2} \|\nabla Q\|^2_{L^2} - \lambda^{-2\sigc} E^0(Q)\right) \\
		&= \lambda^{-2\sigc} \|Q\|^{\alpha+2}_{L^{\alpha+2}}.
		\end{align*}
		Thus we get
		\begin{align*}
		\left[\|f\|^{\alpha+2}_{L^{\alpha+2}}\right] &\div \left[\|(\nabla +iA)f\|^{\frac{3\alpha}{2}}_{L^2} \|f\|^{\frac{4-\alpha}{2}}_{L^2}\right] \\
		&= \left[\lambda^{-2\sigc} \|Q\|^{\alpha+2}_{L^{\alpha+2}} \right] \div \left[ \left( \lambda^{-\sigc} \|\nabla Q\|_{L^2}\right)^{\frac{3\alpha}{2}} \left(\lambda \|Q\|_{L^2} \right)^{\frac{4-\alpha}{2}}\right] \\
		&= \left[ \|Q\|^{\alpha+2}_{L^{\alpha+2}}\right] \div \left[ \|\nabla Q\|_{L^2}^{\frac{3\alpha}{2}} \|Q\|_{L^2}^{\frac{4-\alpha}{2}} \right] = C_{\opt},
		\end{align*}
		where $C_{\opt}$ is as in \eqref{GN-ineq-super}. From this, the diamagnetic inequality, and \eqref{GN-ineq-super}, we see that
		\[
		\|f\|^{\alpha+2}_{L^{\alpha+2}} = C_{\opt} \|(\nabla +iA)f\|^{\frac{3\alpha}{2}}_{L^2} \|f\|^{\frac{4-\alpha}{2}}_{L^2} \geq C_{\opt} \|\nabla |f|\|^{\frac{3\alpha}{2}}_{L^2} \|f\|^{\frac{4-\alpha}{2}}_{L^2} \geq \|f\|^{\alpha+2}_{L^{\alpha+2}}.
		\]
		In particular, we have
		\[
		\|(\nabla +iA) f\|_{L^2}= \|\nabla |f|\|_{L^2}.
		\]
		Using the fact (see e.g., \cite[Theorem 7.21]{LL}) that
		\[
		|\nabla |f|| = \left| \rea \left(\nabla f \frac{\overline{f}}{|f|}\right) \right| = \left| \rea \left( (\nabla +iA) f \frac{\overline{f}}{|f|} \right) \right| \leq |(\nabla+iA) f|,
		\]
		we infer that
		\[
		\ima \left( (\nabla +iA) f \frac{\overline{f}}{|f|} \right) =0 \Longleftrightarrow A= -\ima \left( \frac{\nabla f}{f}\right) \text{ a.e. in } \R^3,
		\]
		hence $\curl A= (0,0,0)$ which contradicts \eqref{defi-B}.
	\end{proof}
	
	We next give the proof of the sharp threshold for global existence and blow-up for \eqref{mag-NLS} in the mass-supercritical case given in Theorem \ref{theo-dyn-below}.
	
	\begin{proof}[Proof of Theorem \ref{theo-dyn-below}]
		(1) Let us consider $u_0 \in \Sigma_A(\R^3)$ satisfying \eqref{cond-ener-below} and \eqref{cond-gwp-below}. Let $u:[0,T^*)\times \R^3 \rightarrow \C$ be the corresponding solution to \eqref{mag-NLS}. By the Gagliardo-Nirenberg inequality \eqref{GN-ineq-super}, we have 
		\begin{align*}
		E_0(u(t)) [M(u(t))]^{\sigc} &\geq \frac{1}{2} \left(\|\nabla u(t)\|_{L^2} \|u(t)\|^{\sigc}_{L^2}\right)^2 + \frac{b^2}{8} \|\rho u(t)\|_{L^2} \|u(t)\|^{2\sigc}_{L^2} \nonumber\\
		&\mathrel{\phantom{\geq \frac{1}{2} \left(\|\nabla u(t)\|_{L^2} \|u(t)\|^{\sigc}_{L^2}\right)^2}}- \frac{C_{\opt}}{\alpha+2} \| \nabla u(t)\|^{\frac{3\alpha}{2}}_{L^2} \|u(t)\|^{\frac{4-\alpha}{2} + 2\sigc}_{L^2} \nonumber \\
		&\geq G\left( \|\nabla u(t)\|_{L^2} \|u(t)\|^{\sigc}_{L^2}\right) 
		\end{align*}
		for all $t\in [0,T^*)$, where $G$ is as in \eqref{defi-G}. Using \eqref{iden-G}, Lemma \ref{lem-cons-angu}, and \eqref{cond-ener-below}, we have
		\begin{align} \label{est-G-2}
		G\left( \|\nabla u(t)\|_{L^2} \|u(t)\|^{\sigc}_{L^2}\right) &\leq E_0(u_0) [M(u_0)]^{\sigc} \nonumber \\
		&<E^0(Q) [M(Q)]^{\sigc}= G\left( \|\nabla Q\|_{L^2} \|Q\|^{\sigc}_{L^2}\right)
		\end{align}
		for all $t\in [0,T^*)$. By \eqref{cond-gwp-below}, the continuity argument implies
		\[
		\|\nabla u(t)\|_{L^2} \|u(t)\|^{\sigc}_{L^2} < \|\nabla Q\|_{L^2} \|Q\|^{\sigc}_{L^2}
		\]
		for all $t\in [0,T^*)$. By the conservation of mass, we infer that 
		\[
		\sup_{t\in [0,T^*)} \|\nabla u(t)\|_{L^2} \leq C(\|u_0\|_{L^2}, \|Q\|_{L^2}, \|\nabla Q\|_{L^2}).
		\]
		On the other hand, by Lemma \ref{lem-cons-angu} and \eqref{GN-ineq-super}, we have
		\begin{align*}
		\frac{b^2}{8} \|\rho u(t)\|^2_{L^2} &\leq E_0(u(t)) + \frac{1}{\alpha+2} \|u(t)\|^{\alpha+2}_{L^{\alpha+2}} \\
		&\leq E_0(u_0) + \frac{C_{\opt}}{\alpha+2} \|\nabla u(t)\|^{\frac{3\alpha}{2}}_{L^2}\|u(t)\|^{\frac{4-\alpha}{2}}_{L^2} \leq C(E_0(u_0), M(u_0), \|Q\|_{L^2}, \|\nabla Q\|_{L^2})
		\end{align*}
		for all $t\in [0,T^*)$. From \eqref{mag-norm}, Remark \ref{rem-equi-norm}, and Lemma \ref{lem-cons-angu}, we have
		\[
		\sup_{t \in [0,T^*)} \|(\nabla + iA) u(t)\|_{L^2} \leq C(E_0(u_0), M(u_0), \|Q\|_{L^2}, \|\nabla Q\|_{L^2})
		\]
		which, by the blow-up alternative, implies that $T^*=\infty$.
		
		(2) Let us now consider $u_0 \in \Sigma_A(\R^3)$ satisfying \eqref{cond-ener-below} and \eqref{cond-blow-below}. By the same argument as above, we see that
		\begin{align} \label{est-blow-below}
		\|\nabla u(t)\|_{L^2} \|u(t)\|^{\sigc}_{L^2} > \|\nabla Q\|_{L^2} \|Q\|^{\sigc}_{L^2}
		\end{align}
		for all $t\in [0,T^*)$. We next show that the solution blows up in finite time. From \eqref{cond-ener-below}, we take $\vartheta =\vartheta(u_0,Q)>0$ such that 
		\[
		E_0(u_0) [M(u_0)]^{\sigc} \leq (1-\vartheta) E^0(Q) [M(Q)]^{\sigc}.
		\]
		We also denote
		\begin{align} \label{defi-H}
		\begin{aligned}
		H(f):&= \|\nabla f\|^2_{L^2} - \frac{b^2}{4} \|\rho f\|^2_{L^2} - \frac{3\alpha}{2(\alpha+2)} \|f\|^{\alpha+2}_{L^{\alpha+2}} \\
		&=\frac{3\alpha}{2} E_0(f) - \frac{3\alpha-4}{4} \|\nabla f\|^2_{L^2} - \frac{(3\alpha+4)b^2}{16} \|\rho f\|^2_{L^2}.
		\end{aligned}
		\end{align}
		By Lemma \ref{lem-cons-angu}, \eqref{est-blow-below}, and the conservation of mass, we see that
		\begin{align*}
		H(u(t)) [M(u(t))]^{\sigc} &\leq \frac{3\alpha}{2} E_0(u(t)) [M(u(t))]^{\sigc} - \frac{3\alpha-4}{4}\left( \|\nabla u(t)\|_{L^2} \|u(t)\|^{\sigc}_{L^2}\right)^2 \\
		&\leq \frac{3\alpha}{2}E_0(u_0) [M(u_0)]^{\sigc} - \frac{3(p-1)-4}{4} \left(\|\nabla Q\|_{L^2}\|Q\|^{\sigc}_{L^2}\right)^2 \\
		&\leq \frac{3\alpha}{2}(1-\vartheta) E^0(Q) [M(Q)]^{\sigc} -\frac{3\alpha-4}{4} \left(\|\nabla Q\|_{L^2}\|Q\|^{\sigc}_{L^2}\right)^2 \\
		& = -\frac{3\alpha-4}{4} \vartheta \left(\|\nabla Q\|_{L^2}\|Q\|^{\sigc}_{L^2}\right)^2
		\end{align*}
		for all $t\in [0,T^*)$. It follows from Lemma \ref{lem-viri-iden} that
		\[
		F''(u(t)) = 8 H(u(t)) \leq -2(3\alpha-4)\vartheta \left(\frac{\|Q\|_{L^2}}{\|u_0\|_{L^2}}\right)^{2\sigc} \|\nabla Q\|^2_{L^2}<0
		\]
		for all $t\in [0,T^*)$. This shows that $T^*<\infty$. The proof is complete.
	\end{proof}
	
	Next we study the long time behaviors of solutions to \eqref{mag-NLS} with data lying at the mass-energy threshold given in Theorem \ref{theo-dyn-at}.
	
	\begin{proof}[Proof of Theorem \ref{theo-dyn-at}] Let us start with the following observation. 
		
		\begin{observation} \label{observation}
			There is no $f\in \Sigma_A(\R^3)$ satisfying
			\[
			E_0(f) [M(f)]^{\sigc} = E^0(Q) [M(Q)]^{\sigc}, \quad \|\nabla f\|_{L^2} \|f\|^{\sigc}_{L^2} = \|\nabla Q\|_{L^2} \|Q\|^{\sigc}_{L^2}.
			\]
		\end{observation}
		In fact, we take $\lambda>0$ such that $\|f\|_{L^2} = \lambda \|Q\|_{L^2}$. It follows that
		\begin{align} \label{obser-f}
		E_0(f) = \lambda^{-2\sigc} E^0(Q), \quad \|\nabla f\|_{L^2} = \lambda^{-\sigc} \|\nabla Q\|_{L^2}.
		\end{align}
		Using the Gagliardo-Nirenberg inequality \eqref{GN-ineq-super} and \eqref{proper-E0-Q}, we see that
		\begin{align*}
		\|f\|^{\alpha+2}_{L^{\alpha+2}} \|f\|^{2\sigc}_{L^2} &\leq C_{\opt} \|\nabla f\|^{\frac{3\alpha}{2}}_{L^2} \|f\|^{2\sigc+\frac{4-\alpha}{2}}_{L^2} \\
		& = \frac{2(\alpha+2)}{3\alpha} \left(\|\nabla Q\|_{L^2} \|Q\|^{\sigc}_{L^2}\right)^{-\frac{3\alpha-4}{2}}  \left(\|\nabla f\|_{L^2}\|f\|^{\sigc}_{L^2}\right)^{\frac{3\alpha}{2}} \\
		& = \frac{2(\alpha+2)}{3\alpha} \left(\|\nabla Q\|_{L^2}\|Q\|_{L^2}^{\sigc}\right)^2.
		\end{align*}
		This implies
		\[
		\|f\|^{\alpha+2}_{L^{\alpha+2}} \leq \frac{2(\alpha+2)}{3\alpha} \lambda^{-2\sigc} \|\nabla Q\|^2_{L^2} = \lambda^{-2\sigc} \|Q\|^{\alpha+2}_{L^{\alpha+2}}.
		\]
		Using \eqref{obser-f}, we infer that
		\begin{align*}
		0\leq \frac{b^2}{2} \|\rho f\|^2_{L^2} &= E_0(f) - \frac{1}{2}\|\nabla f\|^2_{L^2}  +\frac{1}{\alpha+2} \|f\|^{\alpha+2}_{L^{\alpha+2}} \\
		&\leq \frac{1}{\alpha+2}\|f\|^{\alpha+2}_{L^{\alpha+2}} - \frac{1}{\alpha+2} \lambda^{-2\sigc} \|Q\|^{\alpha+2}_{L^{\alpha+2}} \leq 0.
		\end{align*}
		This shows that $f =0$ which is a contradiction. 
		
		(1) Let $u_0\in \Sigma_A(\R^3)$ satisfy \eqref{cond-ener-at} and \eqref{cond-gwp-at}. Let $u:[0,T^*)\times \R^3 \rightarrow \C$ be the corresponding solution to \eqref{mag-NLS}. 
		We will show that 
		\begin{align*}
		\|\nabla u(t)\|_{L^2} \|u(t)\|^{\sigc}_{L^2} < \|\nabla Q\|_{L^2} \|Q\|^{\sigc}_{L^2}
		\end{align*}
		for all $t\in [0,T^*)$. Assume by contradiction that there exists $t_0 \in [0,T^*)$ such that 
		\[
		\|\nabla u(t_0)\|_{L^2} \|u(t_0)\|^{\sigc}_{L^2} \geq \|\nabla Q\|_{L^2} \|Q\|^{\sigc}_{L^2}.
		\] 
		By the continuity using \eqref{cond-gwp-at}, there exists $t_1\in (0,t_0]$ such that 
		\[
		\|\nabla u(t_1)\|_{L^2} \|u(t_1)\|^{\sigc}_{L^2} = \|\nabla Q\|_{L^2} \|Q\|^{\sigc}_{L^2}.
		\]
		By Lemma \ref{lem-cons-angu} and \eqref{cond-ener-at}, we have
		\[
		E_0(u(t_1)) [M(u(t_1))]^{\sigc} = E^0(Q) [M(Q)]^{\sigc}
		\]
		which contradicts Observation \ref{observation}.
		
		(2) Let $u_0 \in \Sigma_A(\R^3)$ satisfy \eqref{cond-ener-at} and \eqref{cond-blow-at}. By the same argument as above, we prove that
		\[
		\|\nabla u(t)\|_{L^2} \|u(t)\|^{\sigc}_{L^2} > \|\nabla Q\|_{L^2} \|Q\|^{\sigc}_{L^2}
		\]
		for all $t\in [0,T^*)$. If $T^*<\infty$, then we are done. If $T^*=\infty$, then we consider two cases.
		
		{\bf Case 1.} If 
		\[
		\sup_{t\in [0,\infty)} \|\nabla u(t)\|_{L^2} \|u(t)\|^{\sigc}_{L^2} > \|\nabla Q\|_{L^2} \|Q\|^{\sigc}_{L^2},
		\] 
		then there exists $\eta>0$ such that for all $t \in [0,\infty)$,
		\[
		\|\nabla u(t)\|_{L^2} \|u(t)\|^{\sigc}_{L^2} \geq (1+\eta) \|\nabla Q\|_{L^2} \|Q\|^{\sigc}_{L^2}.
		\]
		It follows that
		\begin{align*}
		H(u(t)) [M(u(t))]^{\sigc} & \leq \frac{3\alpha}{2} E_0(u(t))[M(u(t))]^{\sigc} - \frac{3\alpha-4}{4} \left( \|\nabla u(t)\|_{L^2} \|u(t)\|^{\sigc}_{L^2} \right)^2 \\
		&\leq \frac{3\alpha}{2} E_0(u_0) [M(u_0)]^{\sigc} - \frac{3\alpha-4}{4} (1+\eta)^2 \left(\|\nabla Q\|_{L^2} \|Q\|^{\sigc}_{L^2}\right)^2 \\
		&= \frac{3\alpha}{2} E^0(Q) [M(Q)]^{\sigc} - \frac{3\alpha-4}{4}(1+\eta)^2 \left(\|\nabla Q\|_{L^2} \|Q\|^{\sigc}_{L^2}\right)^2 \\
		&= \frac{3\alpha-4}{4} \left(1-(1+\eta)^2\right) \left(\|\nabla Q\|_{L^2} \|Q\|^{\sigc}_{L^2}\right)^2 <0
		\end{align*}
		for all $t\in [0,\infty)$, where the functional $H$ is as in \eqref{defi-H}. Thus we have
		\[
		F''(u(t))=8H(u(t)) \leq -2(3\alpha-4) \left((1+\eta)^2-1\right) \left(\frac{\|Q\|_{L^2}}{\|u_0\|_{L^2}}\right)^{2\sigc} \|\nabla Q\|_{L^2}^2
		\]
		for all $t\in [0,\infty)$. Integrating this inequality, there exists $t_0>0$ such that $F(t_0)<0$ which is a contradiction.
		
		{\bf Case 2.} We must have 
		\[
		\sup_{t\in [0,\infty)} \|\nabla u(t)\|_{L^2} \|u(t)\|_{L^2}^{\sigc} = \|\nabla Q\|_{L^2} \|Q\|^{\sigc}_{L^2}.
		\]
		Thus there exists $(t_n)_{n\geq 1} \subset [0,\infty)$ such that 
		\[
		\lim_{n\rightarrow \infty} \|\nabla u(t_n)\|_{L^2} \|u(t_n)\|_{L^2}^{\sigc} = \|\nabla Q\|_{L^2} \|Q\|^{\sigc}_{L^2}.
		\]
		By the conservation laws of mass and Lemma \ref{lem-cons-angu}, we have
		\[
		E_0(u(t_n))[M(u(t_n))]^{\sigc}= E^0(Q) [M(Q)]^{\sigc}.
		\]
		Note that $t_n$ must tend to infinity. Otherwise, there exists $t_0 \in [0,\infty)$ such that up to a subsequence, $t_n \rightarrow t_0$ as $n\rightarrow \infty$. By continuity of the solution maps $t \ni [0,\infty) \mapsto u(t) \in \Sigma_A(\R^3)$ and $\Sigma_A(\R^3)\subset H^1(\R^3)$, we have
		\[
		E_0(u(t_0))[M(u(t_0))]^{\sigc}= E^0(Q) [M(Q)]^{\sigc}, \quad \|\nabla u(t_0)\|_{L^2} \|u(t_0)\|_{L^2}^{\sigc} = \|\nabla Q\|_{L^2} \|Q\|^{\sigc}_{L^2} 
		\]
		which is impossible due to Observation \ref{observation}. Now, we take $\lambda>0$ so that $\|u(t_n)\|_{L^2} = \lambda \|Q\|_{L^2}$. Note that $\lambda$ is independent of $n$ due to the conservation of mass. It follows that
		\[
		E_0(u(t_n)) = \lambda^{-2\sigc} E^0(Q), \quad \lim_{n\rightarrow \infty} \|\nabla u(t_n)\|_{L^2} = \lambda^{-\sigc} \|\nabla Q\|_{L^2}.
		\]
		By the Gagliardo-Nirenberg inequality \eqref{GN-ineq-super}, we see that
		\begin{align*}
		\|u(t_n)\|^{\alpha+2}_{L^{\alpha+2}} &\leq C_{\opt} \|\nabla u(t_n)\|^{\frac{3\alpha}{2}}_{L^2} \|u(t_n)\|^{\frac{4-\alpha}{2}}_{L^2} \\
		&= \frac{2(\alpha+2)}{3\alpha} \left(\|\nabla Q\|_{L^2} \|Q\|^{\sigc}_{L^2}\right)^{-\frac{3\alpha-4}{2}} \|\nabla u(t_n)\|^{\frac{3\alpha}{2}}_{L^2} \left(\lambda \|Q\|_{L^2}\right)^{\frac{4-\alpha}{2}}
		\end{align*}
		which implies
		\[
		\lim_{n\rightarrow \infty} \|u(t_n)\|^{\alpha+2}_{L^{\alpha+2}} \leq \frac{2(\alpha+2)}{3\alpha} \lambda^{-2\sigc} \|\nabla Q\|^2_{L^2} = \lambda^{-2\sigc} \|Q\|^{\alpha+2}_{L^{\alpha+2}}.
		\]
		Thus we have
		\[
		\lambda^{-2\sigc} E^0(Q) \leq \lim_{n\rightarrow \infty} E^0(u(t_n)) \leq E_0(u(t_n)) = \lambda^{-2\sigc} E^0(Q)
		\]
		which implies
		\[
		\lim_{n\rightarrow \infty} E^0(u(t_n)) = \lambda^{-2\sigc} E^0(Q).
		\]
		We have proved that there exists a time sequence $t_n \rightarrow \infty$ such that 
		\[
		\|u(t_n)\|_{L^2} = \lambda \|Q\|_{L^2}, \quad \lim_{n\rightarrow \infty} \|\nabla u(t_n)\|_{L^2} = \lambda^{-\sigc} \|\nabla Q\|_{L^2}, \quad \lim_{n\rightarrow\infty} E^0(u(t_n))= \lambda^{-2\sigc} E^0(Q)
		\]
		for some $\lambda>0$. By the concentration-compactness lemma of P.-L. Lions \cite{Lions}, there exists a subsequence still denoted by $(u(t_n))_{n\geq 1}$ satisfying one of the following three possibilities: vanishing, dichotomy and compactness. 
		
		The vanishing cannot occur. In fact, suppose that the vanishing occurs. Then it was shown in \cite{Lions} that $u(t_n) \rightarrow 0$ strongly in $L^r(\R^3)$ for any $2<r<6$. This however contradicts to the fact that
		\[
		\lim_{n\rightarrow \infty} \|u(t_n)\|^{\alpha+2}_{L^{\alpha+2}} = \lambda^{-2\sigc} \|Q\|^{\alpha+2}_{L^{\alpha+2}}>0.
		\]
		
		The dichotomy cannot occur. Indeed, suppose the dichotomy occurs, then there exist $\mu \in (0, \lambda \|Q\|_{L^2})$ and sequences $(f^1_n)_{n\geq 1}, (f^2_n)_{n\geq 1}$ bounded in $H^1(\R^3)$ such that
		\[
		\left\{
		\renewcommand*{\arraystretch}{1.3}
		\begin{array}{l}
		\|u(t_n) - f^1_n - f^2_n\|_{L^r} \rightarrow 0 \text{ as } n\rightarrow \infty \text{ for any } 2 \leq r < 6, \\
		\|f^1_n\|_{L^2} \rightarrow \mu, \quad \|f^2_n\|_{L^2} \rightarrow \lambda \|Q\|_{L^2} - \mu \text{ as } n\rightarrow \infty, \\
		\dist(\supp(f^1_n), \supp(f^2_n)) \rightarrow \infty \text{ as } n \rightarrow \infty, \\
		\liminf_{n\rightarrow \infty} \|\nabla u(t_n)\|^2_{L^2} - \|\nabla f^1_n\|^2_{L^2} - \|\nabla f^2_n\|^2_{L^2} \geq 0.
		\end{array}
		\right.	
		\]
		By the Gagliardo-Nirenberg inequality, we have
		\[
		\|f^1_n\|^{\alpha+2}_{L^{\alpha+2}} \leq C_{\opt} \|\nabla f^1_n\|^{\frac{3\alpha}{2}}_{L^2} \|f^1_n\|^{\frac{4-\alpha}{2}}_{L^2} < C_{\opt} \|\nabla f^1_n\|^{\frac{3\alpha}{2}}_{L^2} \|u(t_n)\|^{\frac{4-\alpha}{2}}_{L^2}
		\]
		for $n$ sufficiently large. Similarly, we have for $n$ large enough,
		\[
		\|f^2_n\|^{\alpha+2}_{L^{\alpha+2}} < C_{\opt} \|\nabla f^2_n\|^{\frac{3\alpha}{2}}_{L^2} \|u(t_n)\|^{\frac{4-\alpha}{2}}_{L^2}.
		\]
		It follows that
		\begin{align*}
		\lambda^{-2\sigc} \|Q\|^{\alpha+2}_{L^{\alpha+2}}= \lim_{n\rightarrow \infty} \|u(t_n)\|^{\alpha+2}_{L^{\alpha+2}} &= \lim_{n\rightarrow \infty} \|f^1_n\|^{\alpha+2}_{L^{\alpha+2}} + \|f^2_n\|^{\alpha+2}_{L^{\alpha+2}} \\
		&< C_{\opt} \lim_{n\rightarrow \infty} \left(\|\nabla f^1_n\|^{\frac{3\alpha}{2}}_{L^2} + \|\nabla f^2_n\|^{\frac{3\alpha}{2}}_{L^2} \right) \|u(t_n)\|^{\frac{4-\alpha}{2}}_{L^2} \\
		&\leq C_{\opt} \lim_{n\rightarrow \infty} \left(\|\nabla f^1_n\|_{L^2}^2 +\|\nabla f^2_n\|^2 \right)^{\frac{3\alpha}{4}} \|u(t_n)\|^{\frac{4-\alpha}{2}}_{L^2}\\
		&\leq C_{\opt} \lim_{n\rightarrow \infty} \|\nabla u(t_n)\|^{\frac{3\alpha}{2}}_{L^2} \|u(t_n)\|^{\frac{4-\alpha}{2}}_{L^2} \\
		&= C_{\opt} \left(\lambda^{-\sigc} \|\nabla Q\|_{L^2} \right)^{\frac{3\alpha}{2}} \left( \lambda \|Q\|_{L^2}\right)^{\frac{4-\alpha}{2}} \\
		&= \lambda^{-2\sigc} \|Q\|^{\alpha+2}_{L^{\alpha+2}}
		\end{align*}
		which is a contradiction. 
		
		Therefore, the compactness must occur. By \cite{Lions}, there exist a subsequence still denoted by $(u(t_n))_{n\geq 1}$, a function $f \in H^1(\R^3)$ and a sequence $(y_n)_{n\geq 1} \subset \R^3$ such that $u(t_n,\cdot+y_n) \rightarrow f$ strongly in $L^r(\R^3)$ for any $2\leq r<6$ and weakly in $H^1(\R^3)$. We have
		\[
		\|f\|_{L^2} = \lim_{n\rightarrow \infty} \|u(t_n, \cdot+y_n)\|_{L^2} = \lambda \|Q\|_{L^2}
		\]
		and
		\[
		\|f\|^{\alpha+2}_{L^{\alpha+2}} = \lim_{n\rightarrow \infty} \|u(t_n, \cdot+y_n)\|^{\alpha+2}_{L^{\alpha+2}} = \lambda^{-2\sigc}\|Q\|^{\alpha+2}_{L^{\alpha+2}}
		\]
		and 
		\[
		\|\nabla f\|_{L^2} \leq \liminf_{n\rightarrow \infty} \|\nabla u(t_n, \cdot+y_n)\|_{L^2} = \lambda^{-\sigc} \|\nabla Q\|_{L^2}.
		\]
		On the other hand, by the Gagliardo-Nirenberg inequality \eqref{GN-ineq-super}, we have
		\[
		\|\nabla f\|^{\frac{3\alpha}{2}}_{L^2} \geq \frac{\|f\|^{\alpha+2}_{L^{\alpha+2}}}{C_{\opt} \|f\|^{\frac{4-\alpha}{2}}_{L^2}} = \frac{ \lambda^{-2\sigc} \|Q\|^{\alpha+2}_{L^{\alpha+2}}}{ C_{\opt} \left( \lambda \|Q\|_{L^2}\right)^{\frac{4-\alpha}{2}}} = \left(\lambda^{-\sigc} \|\nabla Q\|_{L^2}\right)^{\frac{3\alpha}{2}}
		\]
		hence $\|\nabla f\|_{L^2} = \lim_{n\rightarrow \infty} \|\nabla u(t_n,\cdot+y_n)\|_{L^2} = \lambda^{-\sigc}\|\nabla Q\|_{L^2}$. In particular, $u(t_n,\cdot+y_n) \rightarrow f$ strongly in $H^1(\R^3)$. It is easy to see that
		\[
		\frac{\|f\|^{\alpha+2}_{L^{\alpha+2}}}{\|\nabla f\|^{\frac{3\alpha}{2}}_{L^2} \|f\|^{\frac{4-\alpha}{2}}_{L^2}} = \frac{\|Q\|^{\alpha+2}_{L^{\alpha+2}}}{\|\nabla Q\|^{\frac{3\alpha}{2}}_{L^2} \|Q\|^{\frac{4-\alpha}{2}}_{L^2}} = C_{\opt}.
		\]
		This shows that $f$ is an optimizer for the Gagliardo-Nirenberg inequality \eqref{GN-ineq-super}. By the characterization of ground state (see e.g., \cite{Lions}) with the fact $\|f\|_{L^2} = \lambda \|Q\|_{L^2}$, we have $f(x) = e^{i\theta} \lambda Q(x - x_0)$ for some $\theta \in \R$, $\mu>0$ and $x_0 \in \R^3$. Redefining the variable, we prove that there exists a sequence $(y_n)_{n\geq 1} \subset \R^3$ such that
		\[
		u(t_n, \cdot +y_n) \rightarrow e^{i\theta} \lambda Q \text{ strongly in } H^1(\R^3)
		\]
		as $n\rightarrow \infty$. 
		The proof is complete.
	\end{proof}
	
	We end this section by giving the proof of the blow-up above the mass-energy threshold given in Theorem \ref{theo-dyn-above}.
	
	\begin{proof}[Proof of Theorem \ref{theo-dyn-above}]
		We follow an argument of T. Duyckaerts and S. Roudenko \cite{DR}. Let $u:[0,T^*) \times \R^3 \rightarrow \C$ be the corresponding solution to \eqref{mag-NLS}. We will proceed in two steps.
		
		{\bf Step 1. Reduction of conditions.} Let us start with the following Cauchy-Schwarz inequality:
		\begin{align} \label{cauchy-schwarz}
		\left( \ima \int \overline{f} x \cdot \nabla f dx \right)^2 \leq \|xf\|^2_{L^2} \left(\|\nabla f\|^2_{L^2} - \left( \frac{\|f\|^{\alpha+2}_{L^{\alpha+2}}}{C_{\opt} \|f\|_{L^2}^{\frac{4-\alpha}{2}}} \right)^{\frac{4}{3\alpha}} \right)
		\end{align}
		for all $f\in H^1(\R^3)$. To see it, we have from \eqref{GN-ineq-super} that
		\[
		\|\nabla f\|^2_{L^2} \geq \left(\frac{\|f\|^{\alpha+2}_{L^{\alpha+2}}}{C_{\opt} \|f\|^{\frac{4-\alpha}{2}}_{L^2}}\right)^{\frac{4}{3\alpha}}.
		\] 
		This implies that 
		\[
		4\lambda^2 \|xf\|^2_{L^2} -4 \lambda \ima \int \overline{f} x \cdot \nabla f  dx + \|\nabla f\|^2_{L^2}=\|\nabla (e^{i\lambda |x|^2} f)\|^2_{L^2} \geq \left(\frac{\|f\|^{\alpha+2}_{L^{\alpha+2}}}{C_{\opt} \|f\|^{\frac{4-\alpha}{2}}_{L^2}}\right)^{\frac{4}{3\alpha}}
		\]
		for all $\lambda \in \R$. This shows \eqref{cauchy-schwarz}. We also recall the following identities:
		\begin{align*}
		F''(u(t)) &= 8 \|\nabla u(t)\|^2_{L^2} -2b^2\|\rho u(t)\|^2_{L^2} - \frac{12\alpha}{\alpha+2} \|u(t)\|^{\alpha+2}_{L^{\alpha+2}} \\
		&= 16 E_0(u(t)) - 4b^2 \|\rho u(t)\|^2_{L^2} - \frac{4(3\alpha-4)}{\alpha+2} \|u(t)\|^{\alpha+2}_{L^{\alpha+2}} \\
		&= 12\alpha E_0(u(t))- 2(3\alpha-4) \|\nabla u(t)\|^2_{L^2} - \frac{(3\alpha+4)b^2}{2} \|\rho u(t)\|^2_{L^2},
		\end{align*}
		where $F(u(t))$ is as in \eqref{defi-F}. In particular, we have
		\begin{align*}
		\|u(t)\|^{\alpha+2}_{L^{\alpha+2}} &= \frac{\alpha+2}{4(3\alpha-4)} \left(16E_0(u(t)) - 4b^2 \|\rho u(t)\|^2_{L^2} - F''(u(t))\right), \\
		\|\nabla u(t)\|^2_{L^2} & = \frac{1}{2(3\alpha-4)} \left(12\alpha E_0(v(t)) - \frac{(3\alpha+4)b^2}{2} \|\rho u(t)\|^2_{L^2} - F''(u(t))\right).
		\end{align*}
		Note that since $\|u(t)\|_{L^{\alpha+2}} \geq 0$, we have 
		\[
		F''(u(t)) +4b^2 \|\rho u(t)\|^2_{L^2} \leq 16E_0(u(t)).
		\]
		Moreover, inserting the above identities into \eqref{cauchy-schwarz}, we get
		\begin{multline*}
		\left(\frac{F'(u(t))}{4}\right)^2 \leq F(u(t)) \Big[ \frac{1}{2(3\alpha-4)} \Big( 12\alpha E_0(u(t)) - \frac{(3\alpha+4)b^2}{2} \|\rho u(t)\|^2_{L^2} - F''(u(t)) \Big) \\
		- \left(\frac{(\alpha+2) \Big( 16E_0(u(t)) - 4b^2 \|\rho u(t)\|^2_{L^2} - F''(u(t)) \Big)}{4(3\alpha-4)C_{\opt} \|u(t)\|^{\frac{4-\alpha}{2}}_{L^2}} \right)^{\frac{4}{3\alpha}}
		\Big]
		\end{multline*}
		Since $3\alpha >4$, we infer that
		\begin{align} \label{est-deri-z}
		(z'(t))^2 \leq 4K\left(F''(u(t)) + 4b^2 \|\rho u(t)\|^2_{L^2} \right),
		\end{align}
		where
		\[
		z(t):= \sqrt{F(v(t))}
		\]
		and
		\[
		K(\lambda):= \frac{1}{2(3\alpha-4)} \left(12 \alpha E_0 - \lambda\right) - \left(\frac{(\alpha+2)\left(16E_0 -\lambda\right)}{4(3\alpha-4) C_{\opt} M^{\frac{4-\alpha}{4}}} \right)^{\frac{4}{3\alpha}}
		\]
		with $\lambda \leq 16E_0$, $E_0=E_0(u(t))= E_0(u_0)$ and $M= M(u(t)) = M(u_0)$. Since $3\alpha>4$, we readily check that $K(\lambda)$ is decreasing on $(-\infty,\lambda_0)$ and increasing on $(\lambda_0, 16E_0)$, where $\lambda_0$ satisfies
		\begin{align} \label{lambda-0-1}
		\frac{3\alpha C_{\opt} M^{\frac{4-\alpha}{4}}}{2(\alpha+2)} = \left( \frac{(\alpha+2)(16E_0-\lambda_0)}{4(3\alpha-4) C_{\opt} M^{\frac{4-\alpha}{4}}}\right)^{\frac{4-3\alpha}{3\alpha}}.
		\end{align}
		This implies that
		\begin{align*}
		K(\lambda_0)= \frac{1}{2(3\alpha-4)} ( 12\alpha E_0-\lambda_0) - \frac{3\alpha(16E_0-\lambda_0)}{8(3\alpha-4)} = \frac{\lambda_0}{8}.  
		\end{align*}
		Note that \eqref{lambda-0-1} can be rewritten as
		\[
		\left(\frac{3\alpha C_{\opt}}{2(\alpha+2)} \right)^{\frac{4}{3\alpha}}  =\left(\frac{8(3\alpha-4)}{3\alpha (16E_0-\lambda_0) M^{\sigc}}\right)^{\frac{3\alpha-4}{3\alpha}}
		\]
		which together with the fact (see \eqref{proper-E0-Q})
		\[
		C_{\opt}=\frac{2(\alpha+2)}{3\alpha} \left(\frac{6\alpha}{3\alpha-4} E^0(Q)[M(Q)]^{\sigc}\right)^{-\frac{3\alpha-4}{4}}
		\]
		imply
		\[
		\frac{(16E_0-\lambda_0) M^{\sigc}}{16E^0(Q)[M(Q)]^{\sigc}} =1
		\]
		or
		\begin{align} \label{lambda-0-2}
		\frac{E_0 M^{\sigc}}{E^0(Q)[M(Q)]^{\sigc}} \left(1-\frac{\lambda_0}{16E_0}\right)=1.
		\end{align}
		As a result, we see that \eqref{cond-above-1} is equivalent to 
		\begin{align} \label{cond-lambda-0-1}
		\lambda_0 \geq 0
		\end{align}
		and \eqref{cond-above-2} is equivalent to $(F'(u_0))^2 \geq 2 F(u_0) \lambda_0$ or
		\begin{align} \label{cond-lambda-0-2}
		(z'(0))^2 \geq \frac{\lambda_0}{2} = 4K(\lambda_0).
		\end{align}
		Moreover, \eqref{cond-above-4} is equivalent to $z'(0) \leq 0$. Finally, \eqref{cond-above-3} is equivalent to 
		\[
		F''(u_0) +4b^2 \|\rho u_0\|^2_{L^2} <\lambda_0.
		\] 
		Indeed, if \eqref{cond-above-3} holds, then 
		\begin{align*}
		F''(u_0) + 4b^2 \|\rho u_0\|^2_{L^2} &= 16E_0 - \frac{4(3\alpha-4)}{\alpha+2} \|u_0\|^{\alpha+2}_{L^{\alpha+2}} \\
		&<16E_0 - \frac{4(3\alpha-4)}{\alpha+2} \frac{ \|Q\|^{\alpha+2}_{L^{\alpha+2}} [M(Q)]^{\sigc}}{M^{\sigc}} \\
		&= 16\left( E_0- \frac{E^0(Q)[M(Q)]^{\sigc}}{M^{\sigc}}\right) \\
		&= 16E_0 \left(1-\frac{E^0(Q)[M(Q)]^{\sigc}}{E_0M^{\sigc}}\right) \\
		&= \lambda_0,
		\end{align*}
		where the last equality comes from \eqref{lambda-0-2}.
		
		{\bf Step 2. Finite time blow-up.} Let $u_0 \in \Sigma_A(\R^3)$ satisfy \eqref{cond-above-1}, \eqref{cond-above-2}, \eqref{cond-above-3}, and \eqref{cond-above-4}. By Step 1, we have
		\begin{align} \label{cond-blow-above-equiv}
		\lambda_0 \geq 0, \quad (z'(0))^2 \geq \frac{\lambda_0}{2} = 4K(\lambda_0), \quad z'(0)\leq 0, \quad F''(u_0) + 4b^2\|\rho u_0\|^2_{L^2} <\lambda_0.
		\end{align}
		We claim that
		\begin{align} \label{claim-blow-above}
		z''(t) <0, \quad \forall t\in [0,T^*).
		\end{align}
		Using the fact
		\begin{align} \label{iden-deri-z}
		z''(t) = \frac{1}{z(t)} \left(\frac{F''(u(t))}{2} - (z'(t))^2\right),
		\end{align}
		we have $z''(0)<0$. Assume that \eqref{claim-blow-above} does not hold. Then there exists $t_0 \in (0,T^*)$ such that 
		\[
		z''(t) <0, \quad \forall t\in [0,t_0), \quad z''(t_0) =0.
		\]
		By \eqref{cond-blow-above-equiv}, we have
		\[
		z'(t) <z'(0) \leq - 2\sqrt{K(\lambda_0)}, \quad \forall t\in (0,t_0].
		\]
		Hence $(z'(t))^2 >4K(\lambda_0)$ which together with \eqref{est-deri-z} imply
		\[
		K\left( F''(u(t)) + 4b^2\|\rho u(t)\|^2_{L^2} \right) > K(\lambda_0), \quad \forall t\in (0,t_0].
		\]
		It follows that 
		\[
		F''(u(t)) + 4b^2\|\rho u(t)\|^2_{L^2} \ne \lambda_0, \quad \forall t\in (0,t_0]
		\]
		which, by continuity, implies
		\[
		F''(u(t)) + 4b^2 \|\rho u(t)\|^2_{L^2} <\lambda_0, \quad \forall t\in [0,t_0].
		\]
		By \eqref{iden-deri-z}, we obtain
		\[
		z''(t_0) = \frac{1}{z(t_0)} \left( \frac{F''(u(t_0))}{2} - (z'(t_0))^2\right) < \frac{1}{z(t_0)} \left(\frac{\lambda_0}{2} -\frac{\lambda_0}{2}\right) =0
		\]
		which is a contradiction. This proves \eqref{claim-blow-above}. Now we assume by contradiction that $T^*=\infty$. Then by \eqref{claim-blow-above}, we have
		\[
		z'(t) \leq z'(1)<z'(0)\leq 0, \quad \forall t\in [1,\infty).
		\]
		This contradicts with the fact $z(t)$ is positive for all $t\in [0,\infty)$. The proof is complete.
	\end{proof}
	
	\section{Existence and stability of normalized standing waves}
	\label{S4}
	\setcounter{equation}{0}
	In this section, we prove the existence and orbital stability of normalized standing waves related to \eqref{mag-NLS}. To this end, we need the following result which plays a crucial role in ruling out the vanishing possibility.	
	
	\begin{lemma} \label{lem-lowe-boun}
		Let $A$ be as in \eqref{defi-A} and $0<\alpha<4$. Let $c>0$ and $(f_n)_{n\geq 1}$ be a minimizing sequence for $I(c)$. Then there exists $C>0$ such that
		\[
		\liminf_{n\rightarrow \infty} \|f_n\|_{L^{\alpha+2}} \geq C>0.
		\]
	\end{lemma}
	
	\begin{proof}
		Assume by contradiction that there exists a subsequence still denoted by $(f_n)_{n\geq 1}$ satisfying $\lim_{n\rightarrow \infty} \|f_n\|_{L^{\alpha+2}} =0$. Thanks to \eqref{L2-bound}, we see that
		\begin{align} \label{lowe-boun-Ic}
		I(c) = \lim_{n\rightarrow \infty} E(f_n) = \lim_{n\rightarrow \infty} \frac{1}{2} \|(\nabla +iA) f_n\|_{L^2}^2 \geq \lim_{n\rightarrow \infty}\frac{|b|}{2} \|f_n\|^2_{L^2} = \frac{|b| c}{2}.
		\end{align}
		Denote $x=(x_\perp, x_3)$ with $x_\perp = (x_1, x_2) \in \R^2$ and $x_3 \in \R$, and set $g(x_\perp):= \sqrt{\frac{|b|}{2\pi}} e^{-\frac{|b|}{4} |x_\perp|^2}$. One can readily check that
		\[
		\|g\|_{L^2(\R^2)} = 1, \quad \|\nabla_\perp g\|^2_{L^2(\R^2)} +\frac{b^2}{4} \|\rho g\|^2_{L^2(\R^2)} = |b|.
		\]
		Let $h \in C^\infty_0(\R)$ be such that $\|h\|^2_{L^2(\R)} =c$ and set
		\begin{align} \label{defi-f-lambda}
		f_\lambda(x) = g(x_\perp) h_\lambda(x_3), \quad h_\lambda(x_3) = \lambda^{\frac{1}{2}} h(\lambda x_3)
		\end{align}
		with $\lambda>0$ to be chosen later. We have $\|f_\lambda\|^2_{L^2}=c$ for all $\lambda>0$. Using \eqref{mag-norm}, we see that
		\begin{align*}
		\|(\nabla+iA)f_\lambda\|^2_{L^2} &= \|\nabla f_\lambda\|^2_{L^2} + bR(f_\lambda) + \frac{b^2}{4} \|\rho f_\lambda\|^2_{L^2} \\
		&= \|\nabla_\perp g\|^2_{L^2(\R^2)} \|h_\lambda\|^2_{L^2(\R)} + \|g\|^2_{L^2(\R^2)} \|\partial_3 h_\lambda\|^2_{L^2(\R)} \\
		&\mathrel{\phantom{=}}+ b \left(\int_{\R^2} L_z g \overline{g} dx_\perp\right) \|h_\lambda\|^2_{L^2(\R)} + \frac{b^2}{4} \|\rho g\|^2_{L^2(\R^2)} \|h_\lambda\|^2_{L^2(\R)}\\
		&=  c\left(\|\nabla_\perp g\|^2_{L^2(\R^2)} + \frac{b^2}{4} \|\rho g\|^2_{L^2(\R^2)} \right) + \lambda^2 \|\partial_3 h\|^2_{L^2(\R)} \\
		&= c|b| + \lambda^2 \|\partial_3 h\|^2_{L^2(\R)}.
		\end{align*}
		Here we note that $\mathlarger{\int}_{\R^2} L_z g \overline{g} dx_\perp =0$ as $g$ is radially symmetric. It follows that
		\[
		E(f_\lambda) = \frac{|b|c}{2} + \frac{\lambda^2}{2} \|\partial_3 h\|^2_{L^2(\R)} - \frac{\lambda^{\frac{\alpha}{2}}}{\alpha+2}\|g\|^{\alpha+2}_{L^{\alpha+2}(\R^2)} \|h\|^{\alpha+2}_{L^{\alpha+2}(\R)}. 
		\]
		As $\alpha<4$, by taking $\lambda>0$ sufficiently small, we have $E(f_\lambda)<\frac{|b|c}{2}$. In particular, $I(c)< \frac{|b|c}{2}$ which contradicts \eqref{lowe-boun-Ic}. The proof is complete.
	\end{proof}

	\begin{proof} [Proof of Theorem \ref{theo-exis-stab-mass-cri}] We proceed in two steps.
		
		{\bf Step 1. Existence of minimizers.} Let $0<c<M(Q)$. We first show that $I(c)$ is well-defined, i.e., $I(c)>-\infty$. Let $f \in S(c)$. By the Gagliardo-Nirenberg inequality \eqref{mag-GN-ineq-mass-cri}, we have
		\begin{align*}
		E(f) &\geq \frac{1}{2} \|(\nabla +iA) f\|^2_{L^2} - \frac{1}{2} \left(\frac{M(f)}{M(Q)} \right)^{\frac{2}{3}} \|(\nabla +iA) f\|^2_{L^2} \\
		&= \frac{1}{2} \left(1-\left(\frac{c}{M(Q)}\right)^{\frac{2}{3}}\right) \|(\nabla+iA)f \|^2_{L^2} \geq 0
		\end{align*}
		for all $f \in S(c)$. This shows that $I(c) \geq 0$. 
		
		Now let $(f_n)_{n\geq 1}$ be a minimizing sequence for $I(c)$. From the above estimate, we have
		\[
		\frac{1}{2} \left(1-\left(\frac{c}{M(Q)}\right)^{\frac{2}{3}}\right) \|(\nabla+iA)f_n\|^2_{L^2} \leq E(f_n) \rightarrow I(c) \text{ as } n \rightarrow \infty.
		\]
		This shows that $(f_n)_{n\geq 1}$ is a bounded sequence in $H^1_A(\R^3)$. Moreover, by Lemma \ref{lem-lowe-boun}, we see that up to a subsequence, 
		\[
		\inf_{n\geq 1} \|f_n\|_{L^{\frac{10}{3}}} \geq C >0.
		\]
		By Lemma \ref{lem-weak-conv}, up to a subsequence, there exist $f \in H^1_A(\R^3)\backslash \{0\}$ and $(y_n)_{n\geq 1}\subset \R^3$ such that
		\[
		\tilde{f}_n(x):= e^{iA(y_n) \cdot x} f_n(x+y_n) \rightharpoonup f \text{ weakly in } H^1_A(\R^3).
		\]
		By the weak convergence in $H^1_A(\R^3)$, we have
		\[
		0<\|f\|^2_{L^2} \leq \liminf_{n\rightarrow \infty} \|\tilde{f}_n\|^2_{L^2} = \liminf_{n\rightarrow \infty} \|f_n\|^2_{L^2}=c
		\]
		and
		\[
		\|(\nabla +iA) f\|^2_{L^2} \leq \liminf_{n\rightarrow \infty} \|(\nabla+iA) \tilde{f}_n\|^2_{L^2} = \liminf_{n\rightarrow \infty} \|(\nabla +iA) f_n\|^2_{L^2}.
		\]
		Next we claim that 
		\begin{align} \label{claim}
		\|f\|^2_{L^2}=c.
		\end{align}
		Let us postpone the verification of \eqref{claim} for the moment and finish the proof of Theorem \ref{theo-exis-stab-mass-cri}. By the weak convergence in $H^1_A(\R^3)$ and \eqref{claim}, we infer that $\tilde{f}_n \rightarrow f$ strongly in $L^2(\R^3)$. Using this strong convergence and the magnetic Gagliardo-Nirenberg inequality 
		\[
		\|f\|^{\frac{10}{3}}_{L^{\frac{10}{3}}} \leq C_{\opt} \|(\nabla +iA) f\|^2_{L^2} \|f\|^{\frac{4}{3}}_{L^2},
		\]
		we see that $\tilde{f}_n \rightarrow f$ strongly in $L^{\frac{10}{3}}(\R^3)$. Thus we get
		\[
		I(c) \leq E(f) \leq \liminf_{n\rightarrow \infty} E(\tilde{f}_n) = \liminf_{n\rightarrow \infty} E(f_n) = I(c),
		\]
		hence $E(f) = I(c)$ or $f$ is a minimizer for $I(c)$. This also implies that $\tilde{f}_n \rightarrow f$ strongly in $H^1_A(\R^3)$. 
		
		It remains to prove \eqref{claim}. Assume by contradiction that it is not true, i.e., $0<\|f\|^2_{L^2}<c$. We have for any $\lambda>0$,
		\[
		E(\lambda f) = \lambda^2 E(f) + \frac{\lambda^2(1-\lambda^\alpha)}{\alpha+2} \|f\|^{\alpha+2}_{L^{\alpha+2}}
		\]
		or
		\[
		E(f) = \frac{1}{\lambda^2} E(\lambda f) + \frac{\lambda^\alpha-1}{\alpha+2} \|f\|^{\alpha+2}_{L^{\alpha+2}}.
		\]
		Set $\lambda_0 = \frac{\sqrt{c}}{\|f\|_{L^2}}>1$. We have $\|\lambda_0 f\|^2_{L^2}=c$ and
		\[
		E(f) = \frac{\|f\|^2_{L^2}}{c} E(\lambda_0 f) + \frac{\lambda_0^\alpha-1}{\alpha+2} \|f\|^{\alpha+2}_{L^{\alpha+2}} > \frac{\|f\|^2_{L^2}}{c} I(c)
		\]
		as $f \ne 0$ and $\lambda_0>1$. Similarly, set $\lambda_n:= \frac{\sqrt{c}}{\|\tilde{f}_n-f\|_{L^2}}$. By Lemma \ref{lem-refi-fatou}, we have $\|\tilde{f}_n-f\|^2_{L^2} \rightarrow c-\|f\|^2_{L^2}$ as $n\rightarrow \infty$, hence $\lambda_n \rightarrow \frac{\sqrt{c}}{\sqrt{c-\|f\|^2_{L^2}}}  >1$ as $n\rightarrow \infty$. In particular, we have
		\[
		\lim_{n\rightarrow \infty} E(\tilde{f}_n-f) = \lim_{n\rightarrow \infty} \frac{1}{\lambda_n^2} E(\lambda_n(\tilde{f}_n-f)) + \frac{\lambda_n^\alpha-1}{\alpha+2} \|\tilde{f}_n-f\|^{\alpha+2}_{L^{\alpha+2}} \geq \frac{c-\|f\|^2_{L^2}}{c} I(c).
		\]
		Using the refined Fatou's lemma (see Lemma \ref{lem-refi-fatou}), we get
		\[
		I(c) = \lim_{n\rightarrow \infty} E(f_n) = \lim_{n\rightarrow \infty} E(\tilde{f}_n) = E(f) + \lim_{n\rightarrow \infty} E(\tilde{f}_n-f) > \frac{\|f\|^2_{L^2}}{c} I(c) + \frac{c-\|f\|^2_{L^2}}{c} I(c) =I(c)
		\]
		which is a contradiction. This proves \eqref{claim} and the existence of minimizers for $I(c)$.
		
		{\bf Step 2. Orbital stability.} Let us now show that the set of minimizers $\Mcal(c)$ is orbitally stable in the sense of Proposition \ref{prop-exis-stab-mass-sub}. We follow an argument of \cite{CE}. Assume by contradiction that it is not true. Then there exist $\vareps_0>0$, $\phi_0 \in \Mcal(c)$, and a sequence of initial data $(u_{0,n})_{n \geq 1} \subset H^1_A(\R^3)$ such that
		\begin{align} \label{sta-prof-1}
		\lim_{n\rightarrow \infty} \|u_{0,n} -\phi_0\|_{H^1_A} =0
		\end{align} 
		and a sequence of time $(t_n)_{n\geq 1} \subset [0,\infty)$ such that 
		\begin{align} \label{sta-prof-2}
		\inf_{\phi \in \Mcal(c)} \inf_{y\in \R^3} \|e^{iA(y)\cdot \cdotb}u_n(t_n, \cdotb +y)- \phi\|_{H^1_A} \geq \vareps_0,
		\end{align}
		where $u_n$ is the solution to \eqref{mag-NLS} with initial data $\left. u_n\right|_{t=0} = u_{0,n}$. Note that the solutions exist globally in time by Proposition \ref{prop-thre-mass}.
		
		Since $\phi_0 \in \Mcal(c)$, we have $E(\phi_0) = I(c)$. From \eqref{sta-prof-1} and the Sobolev embedding, we infer that
		\[
		\|u_{0,n}\|^2_{L^2} \rightarrow \|\phi_0\|_{L^2}^2=c, \quad E(u_{0,n}) \rightarrow E(\phi_0)=I(c) \text{ as } n\rightarrow \infty.
		\]
		By the conservation laws of mass and energy, we have
		\[
		\|u_n(t_n)\|^2_{L^2} \rightarrow c, \quad E(u_n(t_n)) \rightarrow I(c) \text{ as } n \rightarrow \infty.
		\]
		In particular, $(u_n(t_n))_{n\geq 1}$ is a minimizing sequence for $I(c)$. Arguing as in Step 1, we see that up to a subsequence, there exist $\phi \in \Mcal(c)$ and $(y_n)_{n\geq 1}\subset \R^3$ such that 
		\[
		\|e^{iA(y_n)\cdot \cdotb}u_n(t_n, \cdotb+y_n) - \phi\|_{H^1_A} \rightarrow 0 \text{ as } n\rightarrow \infty.
		\]
		This however contradicts \eqref{sta-prof-2}. The proof is complete.
	\end{proof}

	\begin{proof} [Proof of Proposition \ref{prop-non-exis}]
		(1) We first consider the case $\alpha=\frac{4}{3}$. Let $\varphi \in C^\infty_0(\R^3)$ be radially symmetric satisfying $\varphi(x)=1$ for $|x| \leq 1$. We define
		\[
		f_\lambda(x):=   B_\lambda  \lambda^{\frac{3}{2}} \varphi(x) Q_0(\lambda x), \quad \lambda>0,
		\]
		where $Q_0(x)= \frac{Q(x)}{\|Q\|_{L^2}}$ and $B_\lambda>0$ is such that $\|f_\lambda\|^2_{L^2} =c$ for all $\lambda>0$. By the definition, we have
		\begin{align*}
		B_\lambda^{-2} =\frac{1}{c} \int \varphi^2(\lambda^{-1} x) Q^2_0(x) dx.
		\end{align*}
		Since $Q_0$ decays exponentially at infinity, we see that for $\lambda>0$ sufficiently large and any $\delta>0$,
		\[
		\left| \int \left(1-\varphi^2(\lambda^{-1} x)\right) Q_0^2(x) dx \right| \lesssim \int_{|x| \geq \lambda} e^{-C |x|} dx \lesssim \int_{|x| \geq \lambda} |x|^{-3-\delta} dx \lesssim \lambda^{-\delta}.
		\]
		In particular, we have $B^2_\lambda = c + O(\lambda^{-\infty})$ as $\lambda \rightarrow \infty$, where $D_\lambda = O(\lambda^{-\infty})$ means that $|D_\lambda| \leq C \lambda^{-\delta}$ for any $\delta>0$ with some constant $C>0$ independent of $\lambda$. Using \eqref{mag-norm}, we have
		\[
		\|(\nabla+iA) f_\lambda\|^2_{L^2}= \|\nabla f_\lambda\|^2_{L^2} + b R(f_\lambda) + \frac{b^2}{4} \|\rho f_\lambda\|^2_{L^2} = \|\nabla f_\lambda\|^2_{L^2} + \frac{b^2}{4} \|\rho f_\lambda\|^2_{L^2}, 
		\]
		where $R(f_\lambda)=0$ as $f_\lambda$ is radially symmetric. We have
		\begin{align*}
		\|\nabla f_\lambda\|^2_{L^2} = B_\lambda^2 \Big( \int |\nabla \varphi(\lambda^{-1} x)|^2 Q_0^2(x) dx &+ \lambda^2 \int \varphi^2(\lambda^{-1} x) |\nabla Q_0(x)|^2dx \\
		& + 2 \lambda \rea \int \varphi(\lambda^{-1} x) Q_0(x) \nabla \varphi(\lambda^{-1} x) \cdot \nabla Q_0(x) dx \Big).
		\end{align*}
		As $|\nabla Q_0|$ also decays exponentially at infinity and $B_\lambda^2 = c + O(\lambda^{-\infty})$ as $\lambda \rightarrow \infty$, we infer that
		\[
		\|\nabla f_\lambda\|^2_{L^2} = c \lambda^2  \|\nabla Q_0\|^2_{L^2} + O(\lambda^{-\infty})
		\]
		as $\lambda \rightarrow \infty$. On the other hand, since $\lambda^3 Q_0^2(\lambda x)$ converges weakly to the Dirac delta function at zero when $\lambda \rightarrow \infty$, we infer that
		\[
		\int \rho^2(x) |f_\lambda (x)|^2 dx = B_\lambda^2 \int \rho^2(x) \varphi^2(x) \lambda^3 Q_0^2(\lambda x) dx \rightarrow 0
		\]
		as $\lambda \rightarrow \infty$, where $\rho(x)=\sqrt{x_1^2+x_2^2}$. We also have
		\[
		\|f_\lambda\|^{\frac{10}{3}}_{L^{\frac{10}{3}}} =  c^{\frac{5}{3}} \lambda^2 \|Q_0\|^{\frac{10}{3}}_{L^{\frac{10}{3}}} + O(\lambda^{-\infty})
		\]
		as $\lambda \rightarrow \infty$. It follows that
		\begin{align}
		I(c) \leq E(f_\lambda) &= \frac{1}{2} \|(\nabla+iA)f_\lambda\|^2_{L^2} - \frac{3}{10} \|f_\lambda\|^{\frac{10}{3}}_{L^{\frac{10}{3}}} \nonumber \\
		&= \frac{1}{2}\|\nabla f_\lambda\|^2_{L^2} + \frac{b^2}{8}\|\rho f_\lambda\|^2_{L^2} - \frac{3}{10}\|f_\lambda\|^{\frac{10}{3}}_{L^{\frac{10}{3}}} \nonumber \\
		&= \frac{c}{2}\lambda^2 \left( \|\nabla Q_0\|^2_{L^2} - \frac{3}{5} c^{\frac{2}{3}} \|Q_0\|^{\frac{10}{3}}_{L^{\frac{10}{3}}} \right) + o_\lambda(1) \nonumber \\
		&= \frac{c}{2}\lambda^2 \|\nabla Q_0\|^2_{L^2} \left(1-\left(\frac{c}{M(Q)}\right)^{\frac{2}{3}} \right) + o_\lambda(1) \label{est-I-c}
		\end{align}
		as $\lambda \rightarrow \infty$, where $D_\lambda = o_\lambda(1)$ means that $|D_\lambda| \rightarrow 0$ as $\lambda \rightarrow \infty$. Here we have used \eqref{poho-iden-mass} to get
		\[
		\frac{3}{5}\|Q_0\|^{\frac{10}{3}}_{L^{\frac{10}{3}}} = \frac{\|\nabla Q_0\|^2_{L^2}}{\|Q\|^{\frac{4}{3}}_{L^2} }.
		\]
		
		In the case $c>M(Q)$, letting $\lambda \rightarrow \infty$ in \eqref{est-I-c}, we get $I(c) =-\infty$, hence there is no minimizer for $I(c)$. 
		
		In the case $c=M(Q)$, it follows from \eqref{est-I-c} that $I(M(Q)) \leq 0$. On the other hand, by the magnetic Gagliardo-Nirenberg inequality \eqref{mag-GN-ineq-mass-cri}, we have for any $f \in H^1_A$ satisfying $\|f\|_{L^2}^2=c=M(Q)$,
		\[
		E(f) \geq \frac{1}{2} \|(\nabla+iA) f\|^2_{L^2} - \frac{1}{2} \left(\frac{\|f\|_{L^2}}{\|Q\|_{L^2}}\right)^{\frac{4}{3}} \|(\nabla+iA) f\|^2_{L^2} = 0.
		\]
		This shows that $I(M(Q)) \geq 0$, hence $I(M(Q)) =0$. We will show that there is no minimizer for $I(M(Q))$. Assume by contradiction that there exists a minimizer for $I(M(Q))$, says $\phi$. We have
		\begin{align*}
		0= I(\|Q\|^2_{L^2}) = E(\phi) = \frac{1}{2} \|(\nabla+iA) \phi\|^2_{L^2}  - \frac{3}{10} \|\phi\|^{\frac{10}{3}}_{L^{\frac{10}{3}}}
		\end{align*}
		In particular, $\phi$ is an optimizer to the magnetic Gagliardo-Nirenberg inequality \eqref{mag-GN-ineq-mass-cri}. Arguing as in the proof of Proposition \ref{prop-gwp-2}, we get a contradiction. Thus there is no minimizer for $I(M(Q))$.
		
		(2) Let us consider the case $\frac{4}{3}<\alpha<4$. Let $f \in C^\infty_0(\R^3)$ be radially symmetric and satisfy $\|f\|^2_{L^2}=c$. Denote
		\[
		f_\lambda(x):= \lambda^{\frac{3}{2}} f(\lambda x), \quad \lambda>0.
		\]
		We see that $\|f_\lambda\|^2_{L^2}=\|f\|^2_{L^2}=c$ for all $\lambda>0$. We also have
		\begin{align*}
		E(f_\lambda)&=\frac{1}{2} \|(\nabla+iA) f_\lambda\|^2_{L^2} - \frac{1}{\alpha+2} \|f_\lambda\|^{\alpha+2}_{L^{\alpha+2}} \\
		&= \frac{1}{2} \|\nabla f_\lambda\|^2_{L^2} + \frac{b^2}{8} \|\rho f_\lambda\|^2_{L^2} -\frac{1}{\alpha+2} \|f_\lambda\|^{\alpha+2}_{L^{\alpha+2}} \\
		&= \frac{\lambda^2}{2} \|\nabla f\|^2_{L^2} + \frac{b^2 \lambda^{-2}}{8} \|\rho f\|^2_{L^2} - \frac{\lambda^{\frac{3\alpha}{2}}}{\alpha+2} \|f\|^{\alpha+2}_{L^{\alpha+2}}.
		\end{align*}
		As $\alpha>\frac{4}{3}$ or $\frac{3\alpha}{2} >2$, we see that $E(f_\lambda) \rightarrow -\infty$ as $\lambda \rightarrow \infty$. In particular, $I(c) =-\infty$.  
	\end{proof}

	Before giving the proof of Theorem \ref{theo-exis-stab-mass-sup}, we prepare some lemmas.
	
	\begin{lemma} \label{lem-c0}
		Let $\frac{4}{3}<\alpha<4$. Then for any $m>0$, there exists $c_0=c_0(m)>0$ sufficiently small such that for all $0<c<c_0$,
		\begin{align}
		S(c) \cap D(m) &\ne \emptyset, \label{non-empty} \\
		\inf \left\{E(f) \ : \ f \in S(c) \cap D(m/4)\right\} &< \inf \left\{ E(f) \ : \ f \in S(c) \cap \left( D(m) \backslash D(m/2)\right) \right\}. \label{est-inf}
		\end{align}
	\end{lemma}
	
	\begin{proof}
		We take $f_0 \in C^\infty_0(\R^3)$ satisfying $\|(\nabla+iA) f_0\|^2_{L^2} = m$. Denote $c_0=c_0(m) := \|f_0\|^2_{L^2}$ and set $f(x):= \sqrt{\frac{c}{c_0}} f_0(x)$. It follows that $\|f\|^2_{L^2} = c$ and $\|(\nabla+iA) f\|^2_{L^2} = \frac{mc}{c_0} <m$ for all $0<c<c_0$. In particular, $f\in S(c) \cap D(m)$, hence \eqref{non-empty} is proved. 
		
		To prove \eqref{est-inf}, we first observe that $S(c) \cap \left(D(m) \backslash D(m/2)\right) \ne \emptyset$ for $c>0$ sufficiently small. Indeed let $\varphi \in C^\infty_0(\R^3)$ be radially symmetric and satisfy $\|\varphi\|^2_{L^2}=1$. Denote $f^\lambda(x):= \sqrt{c} \lambda^{\frac{3}{2}} \varphi(\lambda x)$ with $\lambda>0$ to be chosen later. We have $\|f^\lambda\|^2_{L^2}=c$ and 
		\begin{align*}
		\|(\nabla+iA)f^\lambda\|^2_{L^2} = \|\nabla f^\lambda\|^2_{L^2} + \frac{b^2}{4} \|\rho f^\lambda\|^2_{L^2} =c \left( \lambda^2 \|\nabla \varphi\|^2_{L^2} + \frac{b^2 \lambda^{-2} }{4} \|\rho \varphi\|^2_{L^2}\right).
		\end{align*}
		For each $m>0$, by reducing $c_0=c_0(m)>0$ if necessary, there exists $\lambda_0>0$ such that 
		\begin{align} \label{cond-m}
		\|(\nabla+iA) f^{\lambda_0}\|^2_{L^2}=\frac{3m}{4}.
		\end{align} 
		In particular, $f^{\lambda_0}\in S(c) \cap \left(D(m) \backslash D(m/2)\right)$. In fact, we observe that \eqref{cond-m} is equivalent to
		\begin{align} \label{est-m}
		\lambda^2\|\nabla \varphi\|^2_{L^2} + \frac{b^2 \lambda^{-2} }{4} \|\rho \varphi\|^2_{L^2} =\frac{3m}{4c}.
		\end{align}
		As a function of $\lambda$, the left hand side of \eqref{est-m} takes values on $[|b|\|\nabla \varphi\|_{L^2} \|\rho \varphi\|_{L^2}, \infty)$. Thus if we take $c_0=c_0(m)>0$ sufficiently small so that $\frac{3m}{4c_0} \geq |b|\|\nabla \varphi\|_{L^2} \|\rho \varphi\|_{L^2}$, there exists $\lambda_0>0$ such that \eqref{est-m} holds. 
		
		Now we prove \eqref{est-inf}. By the Gagliardo-Nirenberg inequality \eqref{GN-ineq-super} and the diamagnetic inequality, we have
		\[
		E(f) \geq \frac{1}{2} \|(\nabla+iA) f\|^2_{L^2} - K \|(\nabla+iA)f\|^{\frac{3\alpha}{2}}_{L^2} \|f\|^{\frac{4-\alpha}{2}}_{L^2}
		\]
		for some constant $K>0$. In particular, we have
		\begin{align} \label{est-g-h-c}
		g_c\left(\|(\nabla+iA) f\|^2_{L^2}\right) \leq E(f) \leq h_c\left(\|(\nabla+iA)f\|^2_{L^2}\right), \quad \forall f \in S(c),
		\end{align}
		where
		\[
		g_c(\lambda):= \frac{1}{2} \lambda- K c^{\frac{4-\alpha}{4}} \lambda^{\frac{3\alpha}{4}}, \quad h_c(\lambda) = \frac{1}{2} \lambda.
		\]
		Thanks to \eqref{est-g-h-c}, \eqref{est-inf} is proved provided that there exists $c_0=c_0(m)>0$ sufficiently small such that for each $0<c<c_0$,
		\begin{align} \label{est-hc}
		h_c(m/4) < \inf_{\lambda \in (m/2,m)} g_c(\lambda).
		\end{align}
		Notice that
		\[
		g_c(\lambda) =\frac{1}{2} \lambda \left( 1- 2K c^{\frac{4-\alpha}{4}} \lambda^{\frac{3\alpha-4}{4}} \right) >\frac{1}{3} \lambda
		\]
		for $\lambda \in (0,m)$ and for $0<c<c_0$. We infer that
		\[
		\inf_{\lambda \in (m/2,m)} g_c(\lambda) \geq \frac{m}{6} > \frac{m}{8} = h_c(m/4).
		\]
		This proves \eqref{est-hc}, hence \eqref{est-inf}. 
	\end{proof}
	
	\begin{lemma} \label{lem-non-vani}
		Let $A$ be as in \eqref{defi-A}, $\frac{4}{3}<\alpha<4$, and $m>0$. Then there exists $c_0=c_0(m)>0$ sufficiently small such that for all $0<c<c_0$ and any minimizing sequence $(f_n)_{n\geq 1}$ of $I^m(c)$, there exists $C>0$ such that
		\[
		\liminf_{n\rightarrow \infty} \|f_n\|_{L^{\alpha+2}} \geq C>0.
		\]
	\end{lemma}
	
	\begin{proof}
		The proof is similar to that of Lemma \ref{lem-lowe-boun}. Suppose that there exists a subsequence still denoted by $(f_n)_{n\geq 1}$ such that $\lim_{n\rightarrow \infty} \|f_n\|_{L^{\alpha+2}} =0$. By \eqref{L2-bound}, we have
		\[
		I^m(c) =\lim_{n\rightarrow \infty} E(f_n) \geq \frac{|b|c}{2}.
		\]
		Let $f_\lambda$ be as in \eqref{defi-f-lambda} with $\lambda>0$ to be chosen shortly. We have $\|f_\lambda\|^2_{L^2} =c$ for all $\lambda>0$ and
		\[
		\|(\nabla+iA)f_\lambda\|^2_{L^2} = c|b| + \lambda^2 \|\partial_3 h\|^2_{L^2(\R)} \leq m
		\]
		provided $0<c<c_0(m) \ll 1$ and $0<\lambda \ll 1$. On the other hand, we have 
		\begin{align} \label{est-I-mc}
		E(f_\lambda)=\frac{|b|c}{2}+\frac{\lambda^2}{2}\|\partial_3 h\|^2_{L^2}-\frac{\lambda^{\frac{\alpha}{2}}}{\alpha+2} \|g\|^{\alpha+2}_{L^{\alpha+2}} \|h\|^{\alpha+2}_{L^{\alpha+2}}<\frac{|b|c}{2} 
		\end{align}
		for $\lambda>0$ sufficiently small. This shows that $I^m(c) < \frac{|b|c}{2}$ which is a contradiction. 
	\end{proof}
	
	\begin{lemma} \label{lem-gwp-super}
		Let $\frac{4}{3}<\alpha<4$. Let $m>0$ and $u_0 \in H^1_A(\R^3)$ be such that
		\[
		\|(\nabla+iA) u_0\|^2_{L^2} \leq m.
		\]
		Then there exists $c_0=c_0(m)>0$ sufficiently small such that for all $0<c<c_0$, if $M(u_0)=c$, then the corresponding solution to \eqref{mag-NLS} exists globally in time, i.e., $T^*=\infty$.
	\end{lemma}
	
	To prove this result, we recall the following simple continuity argument. 
	
	\begin{lemma}[Continuity argument] \label{lem-cont-argu}
		Let $I \subset \R$ be an interval and $X : I \rightarrow [0 ,\infty )$ be a continuous
		function satisfying for every $t \in I$,
		\[
		X(t) \leq \alpha + \beta [X(t)]^\theta,
		\]
		where $\alpha, \beta > 0$ and $\theta > 0$ are constants. Assume that
		\[
		X(t_0 ) \leq 2\alpha, \quad  \beta < 2^{-\theta} \alpha^{1-\theta}
		\]
		for some $t_0 \in I$. Then for every $t \in I$, we have
		\[
		X(t) \leq 2\alpha.
		\]
	\end{lemma}
	
	\begin{proof}[Proof of Lemma \ref{lem-gwp-super}]
		Let $u:[0,T^*)\times \R^3 \rightarrow \C$ be the corresponding solution to \eqref{mag-NLS}. Using \eqref{L2-bound}, we have
		\[
		\|u_0\|^2_{L^2} \leq \frac{1}{|b|} \|(\nabla+iA) u_0\|^2_{L^2} \leq \frac{m}{|b|}.
		\]
		By \eqref{GN-ineq-super} and the diamagnetic inequality, we have
		\begin{align*}
		|E(u_0)| \leq \frac{1}{2} \|(\nabla+iA)u_0\|^2_{L^2} + \frac{C}{\alpha+2} \|(\nabla+iA)u_0\|^{\frac{3\alpha}{2}}_{L^2} \|u_0\|^{\frac{4-\alpha}{2}}_{L^2} \leq \frac{m}{2} + \frac{C}{\alpha+2} \frac{m^{\frac{\alpha+2}{2}}}{|b|^{\frac{4-\alpha}{4}}}.
		\end{align*}
		Similarly, by the conservation of mass and energy, we have for all $t\in [0,T^*)$,
		\begin{align*}
		\|(\nabla+iA)u(t)\|^2_{L^2} &= 2E(u(t) +\frac{2}{\alpha+2} \|u(t)\|^{\alpha+2}_{L^{\alpha+2}} \\
		&\leq 2E(u(t)) + \frac{2C}{\alpha+2} \|(\nabla+iA)u(t)\|^{\frac{3\alpha}{2}}_{L^2} \|u(t)\|^{\frac{4-\alpha}{2}}_{L^2} \\
		&\leq 2|E(u_0)| + \frac{2C}{\alpha+2} \|(\nabla+iA)u(t)\|^{\frac{3\alpha}{2}}_{L^2} (M(u_0))^{\frac{4-\alpha}{4}}.
		\end{align*}
		Set $X(t):= \|(\nabla+iA)u(t)\|^2_{L^2}$ and
		\[
		\alpha:= 2|E(u_0)| + \frac{1}{2}\|(\nabla+iA)u_0\|^2_{L^2}, \quad \beta:= \frac{C}{\alpha+2} (M(u_0))^{\frac{4-\alpha}{4}}, \quad \theta:= \frac{3\alpha}{4}. 
		\]
		It follows that
		\[
		X(t) \leq \alpha + \beta [X(t)]^\theta, \quad \forall t\in [0,T^*).
		\]
		Since $X(0)\leq 2\alpha$, we have from Lemma \ref{lem-cont-argu} that
		\[
		X(t) \leq 2\alpha, \quad \forall t\in [0,T^*)
		\]
		provided that $\beta <2^{-\theta} \alpha^{1-\theta}$. As $\theta>1$ and $\alpha$ is bounded from above by a constant depending only on $m$, we see that $2^{-\theta} \alpha^{1-\theta}$ is bounded from below by some constant depending on $m$. Therefore, if $M(u_0)$ is sufficiently small depending on $m$, then $\sup_{t\in [0,T^*)} \|(\nabla+iA)u(t)\|_{L^2} <\infty$. The blow-up alternative yields $T^*=\infty$. The proof is complete.
	\end{proof}
	
	\begin{proof}[Proof of Theorem \ref{theo-exis-stab-mass-sup}]
		The proof is done in two steps.
		
		{\bf Step 1. Existence of minimizers.} Let $(f_n)_{n\geq 1}$ be a minimizing sequence for $I^m(c)$ with $0<c<c_0=c_0(m)\ll 1$. We see that $(f_n)_{n\geq 1}$ is a bounded sequence in $H^1_A(\R^3)$. By Lemma \ref{lem-non-vani}, we have $\liminf_{n\rightarrow \infty} \|f_n\|_{L^{\alpha+2}} \geq C>0$. From Lemma \ref{lem-weak-conv}, up to a subsequence, there exist $f\in H^1_A(\R^3) \backslash \{0\}$ and a sequence $(y_n)_{n\geq 1} \subset \R^3$ such that
		\[
		\tilde{f}_n(x):=e^{iA(y_n) \cdot x} f_n(x+y_n) \rightharpoonup f \text{ weakly in } H^1_A(\R^3).
		\]
		By the weak convergence, we have
		\[
		0< \|f\|^2_{L^2} \leq \liminf_{n\rightarrow \infty} \|\tilde{f}_n\|^2_{L^2} = \liminf_{n\rightarrow \infty} \|f_n\|^2_{L^2}=c
		\]
		and
		\[
		\|(\nabla+iA) f\|^2_{L^2} \leq \liminf_{n\rightarrow \infty} \|(\nabla+iA) \tilde{f}_n\|^2_{L^2}=\liminf_{n\rightarrow \infty} \|(\nabla+iA) f_n\|^2_{L^2} \leq m.
		\]
		Arguing as in the proof of Theorem \ref{theo-exis-stab-mass-cri}, we have $\|f\|^2_{L^2}=c$, hence $f\in S(c) \cap D(m)$. We also have $E(f) = I^m(c)$ or $f$ is a minimizer for $I^m(c)$. Moreover, $\tilde{f}_n\rightarrow f$ strongly in $H^1_A(\R^3)$.
		
		We next prove that $f \in D(m/2)$. Indeed, suppose that it is not true. By \eqref{est-inf}, we have
		\begin{align*}
		I^m(c) &\leq \inf \left\{E(f) \ : \ f \in S(c) \cap D(m/4)\right\} \\
		&< \inf \left\{ E(f) \ : \ f \in S(c) \cap \left(D(m) \backslash D(m/2)\right) \right\} \\
		&\leq E(f) = I^m(c)
		\end{align*}
		which is a contradiction. This shows that $\emptyset \ne \Mcal^m(c) \subset D(m/2)$. As $f$ does not belong to the boundary of $D(m)$, there exists a Lagrange multiplier $\omega \in \R$ such that $S'_\omega(f)[\varphi] =0$ for all $\varphi \in C^\infty_0(\R^3)$, where $S_\omega(f):= E(f) + \frac{\omega}{2} M(f)$. A direct computation shows that $f$ is a solution to 
		\[
		-(\nabla+iA)^2 f + \omega f - |f|^\alpha f =0
		\] 
		in the weak sense. From this, we infer that
		\[
		\omega \|f\|^2_{L^2} = - \|(\nabla+iA) f\|^2_{L^2} + \|f\|^{\alpha+2}_{L^{\alpha+2}} = - 2E(f) + \frac{\alpha}{\alpha+2}\|f\|^{\alpha+2}_{L^{\alpha+2}}> -2E(f).
		\]
		Thus we get 
		\[
		\omega> -\frac{2E(f)}{\|f\|^2_{L^2}} =-\frac{2I^m(c)}{c} > -|b|,
		\]
		where the last inequality follows from \eqref{est-I-mc}. On the other hand, by \eqref{GN-ineq-super} and \eqref{diag-ineq}, we have
		\begin{align*}
		\omega \|f\|^2_{L^2} &\leq - \|(\nabla+iA)f\|^2_{L^2} + K \|(\nabla+iA)f\|^{\frac{3\alpha}{2}}_{L^2} \|f\|^{\frac{4-\alpha}{2}}_{L^2} \\
		&\leq - \|(\nabla+iA)f\|^2_{L^2} \left(1- K\|(\nabla+iA)f\|^{\frac{3\alpha-4}{2}}_{L^2} \|f\|^{\frac{4-\alpha}{2}}_{L^2} \right)
		\end{align*}
		for some constant $K>0$. As $\|f\|^2_{L^2}=c$ and $f \in D(m/2)$ or $\|(\nabla+iA)f\|^2_{L^2} \leq m/2$, we get
		\[
		\omega \|f\|^2_{L^2} \leq - \|(\nabla+iA)f\|^2_{L^2} \left(1 - K c^{\frac{4-\alpha}{4}} m^{\frac{3\alpha-4}{4}} \right),
		\]
		where the constant $K$ may vary from line to line. Reducing the value of $c$ if necessary, we infer from \eqref{L2-bound} that
		\[
		\omega \leq -|b| \left(1 - K c^{\frac{4-\alpha}{4}} m^{\frac{3\alpha-4}{4}} \right).
		\]
		This shows \eqref{est-omega}. It completes the proof of Item (1).
		
		{\bf Step 2. Orbital stability.} As in the proof of Theorem \ref{theo-exis-stab-mass-cri}, we argue by contradiction. Suppose that $\Mcal^m(c)$ is not orbitally stable. There exist $\epsilon_0>0$, $\phi_0 \in \Mcal^m(c)$, a sequence of initial data $u_{0,n} \in H^1_A(\R^3)$ satisfying
		\begin{align}  \label{orbi-proof-sup-1}
		\lim_{n\rightarrow \infty} \|u_{0,n} - \phi_0\|_{H^1_A} = 0 
		\end{align}
		and a sequence of time $(t_n)_{n\geq 1} \subset [0,\infty)$ such that
		\begin{align} \label{orbi-proof-sup-2}
		\inf_{\phi \in \Mcal^m(c)} \inf_{y\in \R^3}\|e^{iA(y_n) \cdot \cdotb} u_n(t_n, \cdotb+y) - \phi\|_{H^1_A} \geq \epsilon_0,
		\end{align}
		where $u_n$ is the solution to \eqref{mag-NLS} with initial data $\left.u_n\right|_{t=0}= u_{0,n}$. Note that the solutions exist globally in time by Lemma \ref{lem-gwp-super}.
		
		Since $\phi_0 \in \Mcal^m(c)$, we have $E(\phi_0) = I^m(c)$. By \eqref{orbi-proof-sup-1} and the Sobolev embedding, we have $\|u_{0,n}\|^2_{L^2} \rightarrow \|\phi_0\|^2_{L^2}=c$ and
		\[
		\|(\nabla+iA) u_{0,n}\|^2_{L^2} \rightarrow \|(\nabla+iA)\phi_0\|^2_{L^2} \leq m, \quad E(u_{0,n}) \rightarrow E(\phi_0) = I^m(c).
		\]
		By conservation laws of mass and energy, we have
		\[
		\|u_n(t_n)\|^2_{L^2} \rightarrow c, \quad E(u_n(t_n)) \rightarrow I^m(c)
		\]
		as $n\rightarrow \infty$. We next claim that (up to a subsequence) $\|(\nabla+iA)u_n(t_n)\|^2_{L^2} \leq m$. Suppose that there exists $N\geq 1$ such that $\|(\nabla+iA)u_n(t_n)\|^2_{L^2} >m$ for every $n\geq N$. By continuity, there exists $t_n^* \in (0,t_n)$ such that $\|(\nabla+iA)u_n(t_n^*)\|^2_{L^2} =m$. Since 
		\[
		\|u_n(t^*_n)\|^2_{L^2} \rightarrow c, \quad \|(\nabla+iA) u_n(t^*_n)\|^2_{L^2} =m, \quad E(u_n(t^*_n)) \rightarrow I^m(c)
		\]
		as $n\rightarrow \infty$, we see that $u_n(t^*_n)$ is a minimizing sequence for $I^m(c)$. By Step 1, there exist $\phi \in \Mcal^m(c)$ and a sequence $(y_n)_{n\geq 1} \subset \R^3$ such that $e^{iA(y_n)\cdot \cdotb}u_n(t^*_n, \cdotb+y_n) \rightarrow \phi$ strongly in $H^1_A(\R^3)$. This is not possible since minimizers for $I^m(c)$ does not belong to the boundary of $S(c)\cap D(m)$. Thus there exists a subsequence $(t_{n_k})_{k\geq 1}$ such that $\|(\nabla+iA)u_{n_k}(t_{n_k})\|^2_{L^2} \leq m$ for all $k\geq 1$. This shows that $(u_{n_k}(t_{n_k}))_{k\geq 1}$ is a minimizing sequence for $I^m(c)$. Again, by Step 1, there exist $\phi \in \Mcal^m(c)$ and a sequence $(y_k)_{k\geq 1} \subset \R^3$ such that
		\[
		\|e^{iA(y_k)\cdot \cdotb} u_{n_k}(t_{n_k}, \cdotb+y_k) - \phi\|_{H^1_A} \rightarrow 0
		\]
		as $k\rightarrow \infty$. This contradicts \eqref{orbi-proof-sup-2}, and the proof is complete.	
	\end{proof}

	\begin{proof} [Proof of Proposition \ref{prop-norm-grou-stat}]
		The first point follows directly from Theorem \ref{theo-exis-stab-mass-sup}. Let us prove the second point. Assume by contradiction that there exists $f \in S(c)$ with $\left.E'\right|_{S(c)} (f)=0$ such that $E(f)<E(\phi)=I^m(c)$. As $\left.E'\right|_{S(c)} (f)=0$, there exists a Lagrange multiplier $\omega \in \R$ such that $f$ is a solution to \eqref{ell-equ}. It follows that
		\[
		\|(\nabla+iA)f\|^2_{L^2} + \omega \|f\|^2_{L^2} - \|f\|^{\alpha+2}_{L^{\alpha+2}} =0.
		\]
		In particular, we have
		\[
		\frac{\alpha}{2(\alpha+2)} \left( \|(\nabla+iA)f\|^2_{L^2} +\omega\|f\|^2_{L^2}\right) =E(f)+\frac{\omega}{2}\|f\|^2_{L^2} < I^m(c) + \frac{\omega}{2}c.
		\]
		We infer that
		\[
		\|(\nabla+iA)f\|^2_{L^2} < \frac{2(\alpha+2)}{\alpha} \left( I^m(c) + \frac{\omega}{2}c\right) - \omega c.
		\]
		Using \eqref{est-I-mc}, we see that
		\[
		\|(\nabla+iA)f\|^2_{L^2} < \frac{\alpha+2}{\alpha} (|b| + \omega ) c - \omega c 
		\]
		which, by \eqref{est-omega}, implies $\|(\nabla+iA)f\|^2_{L^2} \rightarrow 0$ as $c \rightarrow 0$. Thus for $c>0$ sufficiently small, we have $f \in S(c) \cap D(m)$. By the definition of $I^m(c)$, we get $I^m(c) \leq E(f)$ which is a contradiction. The proof is complete.
	\end{proof}

	\section{Existence and instability of ground state standing waves}
	\label{S5}
	\setcounter{equation}{0}
	This section is devoted to the existence of ground states related to \eqref{ell-equ} and the strong instability of ground state standing waves related to \eqref{mag-NLS} in the mass-supercritical case.
	
	Before giving the proof of Theorem \ref{theo-exis-grou}, we need the following observations.
	
	\begin{observation}
		Let $A$ be as in \eqref{defi-A} and $\omega>-|b|$. Then 
		\begin{align} \label{equi-norm}
		H_\omega(f):=\|(\nabla+iA) f\|^2_{L^2} + \omega \|f\|^2_{L^2} \simeq \|(\nabla+iA) f\|^2_{L^2} +  \|f\|^2_{L^2}.
		\end{align}
	\end{observation}
	In fact, we have
	\[
	\|(\nabla+iA) f\|^2_{L^2} + \omega \|f\|^2_{L^2} \leq (1+ |\omega|) \left(\|(\nabla+iA) f\|^2_{L^2} + \|f\|^2_{L^2}\right).
	\]
	On the other hand, by \eqref{L2-bound}, we see that
	\[
	\|(\nabla+iA) f\|^2_{L^2} + \omega \|f\|^2_{L^2}  \geq (\omega+|b|) \|f\|^2_{L^2}.
	\]
	It follows that
	\begin{align*}
	\|(\nabla+iA) f\|^2_{L^2} + \|f\|^2_{L^2} & \leq \|(\nabla+iA) f\|^2_{L^2} + \omega \|f\|^2_{L^2} + |1-\omega| \|f\|^2_{L^2} \\
	&\leq \left(1+\frac{|1-\omega|}{\omega+|b|} \right) \left(\|(\nabla+iA) f\|^2_{L^2} + \omega \|f\|^2_{L^2} \right).
	\end{align*}
	
	\begin{observation} 
		Let $A$ be as in \eqref{defi-A}, $0<\alpha<4$, and $\omega>-|b|$. Then there exists $f \in H^1_A(\R^3)$ such that $K_\omega(f)=0$.
	\end{observation}
	Indeed, for $f \in C^\infty_0(\R^3)$, we have
	\[
	K_\omega(\lambda f) = \lambda^2 H_\omega(f)- \lambda^{\alpha+2} \|f\|^{\alpha+2}_{L^{\alpha+2}}, \quad \lambda>0.
	\]
	It follows that $K_\omega(\lambda_0 f)=0$ with $\lambda_0 = \left(\frac{H_\omega(f)}{\|f\|^{\alpha+2}_{L^{\alpha+2}}}\right)^{\frac{1}{\alpha}}$.
	
	\begin{lemma} \label{lem-exis-d-omega}
		Let $A$ be as in \eqref{defi-A}, $0<\alpha<4$, and $\omega>-|b|$. Then there exists a minimizer for $d(\omega)$.
	\end{lemma}
	
	\begin{proof}
		The proof is done by several steps.
		
		{\bf Step 1.} We first show that $d(\omega)>0$. Let $f \in H^1_A(\R^3)$ be such that $K_\omega(f)=0$. By the Gagliardo-Nirenberg inequality, the diamagnetic inequality, and \eqref{equi-norm}, we have
		\begin{align*}
		H_\omega(f) = \|f\|^{\alpha+2}_{L^{\alpha+2}} &\lesssim \|(\nabla+iA) f\|^{\frac{3\alpha}{2}}_{L^2} \|f\|^{\frac{4-\alpha}{2}}_{L^2} \\
		&\lesssim \left(\|(\nabla+iA)f\|^2_{L^2} +\|f\|^2_{L^2} \right)^{\frac{\alpha}{2}+1} \lesssim (H_\omega(f))^{\frac{\alpha}{2}+1}.
		\end{align*}
		Thus we get $H_\omega(f) \geq C>0$. It follows that
		\[
		S_\omega(f) = \frac{\alpha}{2(\alpha+2)} H_\omega(f) \geq \frac{\alpha}{2(\alpha+2)}C>0.
		\]
		Taking the infimum over all $f\in H^1_A(\R^3)\backslash \{0\}$ satisfying $K_\omega(f)=0$, we obtain $d(\omega)>0$.
		
		{\bf Step 2.} We next show that there exists a minimizer for $d(\omega)$. Let $(f_n)_{n\geq 1}$ be a minimizing sequence for $d(\omega)$. We have 
		\[
		\frac{\alpha}{2(\alpha+2)} H_\omega(f_n) = S_\omega(f_n) \rightarrow d(\omega)>0 \text{ as } n \rightarrow \infty
		\]
		which, by \eqref{equi-norm}, implies that $(f_n)_{n\geq 1}$ is a bounded sequence in $H^1_A(\R^3)$. As $K_\omega(f_n)=0$, we have
		\[
		\|f_n\|^{\alpha+2}_{L^{\alpha+2}} = H_\omega(f_n) \rightarrow \frac{2(\alpha+2)}{\alpha} d(\omega)>0 \text{ as } n\rightarrow \infty.
		\]
		Thus up to a subsequence, we have $\inf_{n\geq 1} \|f_n\|_{L^{\alpha+2}} \geq C>0$. Here the constant $C$ may vary from line to line. Applying Lemma \ref{lem-weak-conv}, there exist $f \in H^1_A(\R^3)\backslash \{0\}$ and $(y_n)_{n\geq 1} \subset \R^3$ such that
		\[
		\tilde{f}_n(x):=e^{iA(y_n) \cdot x} f_n(x+y_n) \rightharpoonup f \text{ weakly in } H^1_A(\R^3).
		\]
		Thanks to Lemma \ref{lem-refi-fatou}, we have
		\begin{align}
		H_\omega(\tilde{f}_n) &= H_\omega(f) + H_\omega(\tilde{f}_n-f) +o_n(1), \label{pytha-H-omega} \\
		K_\omega(\tilde{f}_n)&= K_\omega(f) + K_\omega(\tilde{f}_n-f) + o_n(1). \label{pytha-K-omega}
		\end{align}
		We will show that $K_\omega(f)=0$. Indeed, if $K_\omega(f)<0$, then there exists $\lambda_0 \in (0,1)$ such that $K_\omega(\lambda_0 f) =0$. From the definition of $d(\omega)$, we have
		\begin{align*}
		d(\omega) \leq S_\omega(\lambda_0 f) = \frac{\alpha}{2(\alpha+2)} H_\omega(\lambda_0f) &= \frac{\alpha \lambda_0^2}{2(\alpha+2)} H_\omega(f) \\
		&< \frac{\alpha}{2(\alpha+2)} H_\omega(f) \\
		&\leq \frac{\alpha}{2(\alpha+2)} \liminf_{n\rightarrow \infty} H_\omega(\tilde{f}_n) = \liminf_{n\rightarrow \infty} H_\omega(f_n) = d(\omega)
		\end{align*}
		which is a contradiction. If $K_\omega(f)>0$, then, by \eqref{pytha-K-omega} and the fact that $K_\omega(\tilde{f}_n)= K_\omega(f_n) =0$, we have $K_\omega(\tilde{f}_n-f)<0$ for $n$ sufficiently large. Thus there exists $(\lambda_n)_{n\geq 1} \subset (0,1)$ such that $K_\omega(\lambda_n(\tilde{f}_n-f)) =0$. It follows that
		\begin{align*}
		d(\omega) \leq S_\omega(\lambda_n(\tilde{f}_n-f)) &= \frac{\alpha}{2(\alpha+2)} \lim_{n\rightarrow \infty} H_\omega(\lambda_n(\tilde{f}_n-f)) \\
		&= \frac{\alpha}{2(\alpha+2)} \lim_{n\rightarrow \infty} \lambda_n^2 H_\omega(\tilde{f}_n-f) \\
		&\leq \frac{\alpha}{2(\alpha+2)} \lim_{n\rightarrow \infty} H_\omega(\tilde{f}_n-f) \\
		&= \frac{\alpha}{2(\alpha+2)} \left(\lim_{n\rightarrow \infty} H_\omega(\tilde{f}_n) - H_\omega(f) \right) \\
		&= \frac{\alpha}{2(\alpha+2)}\lim_{n\rightarrow \infty} H_\omega(f_n) - \frac{\alpha}{2(\alpha+2)} H_\omega(f) \\
		&=d(\omega)-\frac{\alpha}{2(\alpha+2)} H_\omega(f) <d(\omega)
		\end{align*}
		which is also a contradiction. Here the fourth line follows from \eqref{pytha-H-omega}. Thus we have $K_\omega(f)=0$. 
		
		By the definition of $d(\omega)$, we have
		\begin{align*}
		d(\omega) \leq S_\omega(f) = \frac{\alpha}{2(\alpha+2)} H_\omega(f) &\leq \frac{\alpha}{2(\alpha+2)} \liminf_{n\rightarrow \infty} H_\omega(\tilde{f}_n) \\
		&= \frac{\alpha}{2(\alpha+2)} \liminf_{n\rightarrow \infty} H_\omega(f_n) = d(\omega).
		\end{align*}
		This shows that $S_\omega(f)=d(\omega)$ or $f$ is a minimizer for $d(\omega)$. The proof is complete.
	\end{proof}

	\begin{proof}[Proof of Theorem \ref{theo-exis-grou}]
		We set 
		\[
		\Dc(\omega):= \left\{ \phi \in H^1_A(\R^3)\backslash \{0\} \ : \ S_\omega(\phi) = d(\omega), K_\omega(\phi)=0\right\}.
		\]
		By Lemma \ref{lem-exis-d-omega}, we have $\Dc(\omega) \ne \emptyset$. We will show that $\Dc(\omega) \equiv \Gc(\omega)$.
		
		To see this, let $\phi \in \Dc(\omega)$. There exists a Lagrange multiplier $\lambda \in \R$ such that $S'_\omega(\phi) = \lambda K'_\omega(\phi)$. It follows that
		\begin{align*}
		K_\omega(\phi) = \scal{S'_\omega(\phi), \phi}_{L^2} &= \lambda \scal{K'_\omega(\phi), \phi}_{L^2} = \lambda \left(2 K_\omega(\phi) - \alpha \|f\|^{\alpha+2}_{L^{\alpha+2}} \right).
		\end{align*}
		As $K_\omega(f)=0$ and $\phi \ne 0$, we have $\lambda=0$ or $S'_\omega(\phi)=0$ or $\phi \in \Ac(\omega)$. Let $f\in \Ac(\omega)$. As $K_\omega(f)=0$, we have $S_\omega(f) \geq d(\omega) =S_\omega(\phi)$. This shows that $S_\omega(\phi)\leq S_\omega(f)$ for all $f\in \Ac(\omega)$ or $\phi \in \Gc(\omega)$. Thus $\Dc(\omega) \subset \Gc(\omega)$.
		
		Finally we show that $\Gc(\omega)\subset \Dc(\omega)$. Indeed, let $\phi \in \Gc(\omega)$ and take $f\in \Dc(\omega)\subset \Gc(\omega)$. We have $S_\omega(f)=S_\omega(\phi)=d(\omega)$. Since $\phi \in \Ac(\omega)$, we have $K_\omega(\phi)=0$. Thus $\phi \in \Dc(\omega)$. The proof is complete.
	\end{proof}
	
	\begin{proposition} \label{prop-expo-deca}
		Let $A$ be as in \eqref{defi-A}, $0<\alpha<$, $\omega>-|b|$, and $\phi \in \Gc(\omega)$. Then $\phi \in L^r(\R^3)$ for all $2\leq r \leq \infty$ and $\lim_{|x| \rightarrow \infty} \phi(x) =0$. Moreover, there exists $\delta>0$ such that $e^{\delta |x|} \phi \in L^2(\R^3)$. In particular, $\phi \in \Sigma_A(\R^3)$.
	\end{proposition}
	
	The proof of Proposition \ref{prop-expo-deca} is based on the following results of Chabrowski and A. Szulkin \cite{CS} and N. Raymond \cite{Raymond}.
	
	\begin{lemma}[\cite{CS}] \label{lem-Linfty}
		Let $A \in L^2_{\loc}(\R^3, \R^3)$. Let $\phi \in H^1_A(\R^3)$ be a solution to \eqref{ell-equ}. Then $\phi \in L^r(\R^3)$ for all $2\leq r \leq \infty$. Moreover, 
		\[
		\lim_{|x| \rightarrow \infty} \phi(x) =0.
		\]
	\end{lemma}
	
	\begin{lemma}[{\cite[Proposition 4.9]{Raymond}}] \label{lem-expo-deca-Raymond}
		Let $V \in C^0(\R^3,\R)$ be bounded from below and $A \in C^1(\R^3, \R^3)$. Assume that there exist $R_0>0$ and $\mu^* \in \R$ such that for all $f \in H^1_A(\R^3)$ with $\supp(f) \subset \R^3 \backslash B(0,R_0)$, we have
		\begin{align} \label{est-V}
			\int |(\nabla+iA)f|^2dx+ \int V|f|^2 dx \geq \mu^* \|f\|^2_{L^2}.
		\end{align}
		Then we have $\inf\spec (-(\nabla+iA)^2 +V) \geq \mu^*$. Moreover, if $\phi$ is an eigenfunction for $-(\nabla+iA)^2 +V$ with eigenvalue $\mu<\mu^*$, then for all $\delta \in (0, \sqrt{\mu^*-\mu})$, we have $e^{\delta |x|} \phi \in L^2(\R^3)$. 
	\end{lemma}
	
	\begin{proof} [Proof of Proposition \ref{prop-expo-deca}]
		By Lemma \ref{lem-Linfty}, it remains to show the exponential decay of the ground state. Let us start with the following observation.
		\begin{observation} \label{obse-V-vareps}
			Let $A$ be as in \eqref{defi-A} and $V \in L^r(\R^3)$ for some $r>\frac{3}{2}$. Then for every $\vareps>0$, there exists $R=R(\vareps)>0$ such that 
			\begin{align} \label{est-V-vareps}
				\int |(\nabla+iA)f|^2dx+ \int V|f|^2 dx \geq (1-\vareps)|b| \|f\|^2_{L^2}
			\end{align}
			for all $f \in H^1_A(\R^3)$ with $\supp(f) \subset \R^3 \backslash B(0,R)$.
		\end{observation}
		
		\begin{proof}
			By the H\"older inequality and Sobolev embedding, we have
			\[
			\left|\int V|f|^2 dx\right| \leq \|V\|_{L^r} \|f\|^2_{L^{\frac{2r}{r-1}}} \leq \|V\|_{L^r} \||f|\|^2_{H^1}, 
			\]
			where $2<\frac{2r}{r-1}<6$ as $r>\frac{3}{2}$. By the diamagnetic inequality \eqref{diag-ineq} and \eqref{L2-bound}, we have
			\[
			\||f|\|^2_{H^1} = \|\nabla |f|\|^2_{L^2} + \|f\|^2_{L^2} \leq \left(1+\frac{1}{|b|}\right) \|(\nabla+iA)f\|^2_{L^2}.
			\]
			It follows that
			\begin{align*}
				\int |(\nabla+iA)f|^2dx+ \int V|f|^2 dx \geq \left(1-\|V\|_{L^r} \left(1+\frac{1}{|b|}\right)\right) \|(\nabla+iA)f\|^2_{L^2}.
			\end{align*}
			In particular, for any $f\in H^1_A(\R^3)$ with $\supp(f)\subset \R^3 \backslash B(0,R)$, we have
			\[
			\int |(\nabla+iA)f|^2dx+ \int V|f|^2 dx \geq \left(1-\|V\|_{L^r(|x|\geq R)}\left(1+\frac{1}{|b|}\right)\right) \|(\nabla+iA)f\|^2_{L^2}.
			\]
			As $V \in L^r(\R^3)$, we have $\|V\|_{L^r(|x|\geq R)} \rightarrow 0$ as $R\rightarrow \infty$. Thus for every $\vareps>0$, there exists $R=R(\vareps)>0$ such that 
			\[
			\|V\|_{L^r(|x|\geq R)} \left(1+\frac{1}{|b|}\right) \leq \vareps
			\]
			which proves \eqref{est-V-vareps}. 
		\end{proof}
		
		Now we prove the exponential decay of the ground state by applying Lemma \ref{lem-expo-deca-Raymond} to $V= -|\phi|^\alpha$. By Lemma \ref{lem-Linfty}, we see that $V$ is bounded from below. To see $V \in C^0(\R^3,\R)$, it suffices to show $V \in C^0(B(0,R), \R)$ for any $R>0$. On $B(0,R)$, the equation \eqref{ell-equ} can be written as
		\[
		-\Delta \phi = -b L_z \phi -\frac{b^2}{4} \rho^2 \phi -\omega \phi +|\phi|^\alpha \phi, \quad x \in B(0,R).
		\] 
		Since the right hand side belongs to $L^2(B(0,R))$, the regularity argument (see e.g., \cite[Theorem 10.2]{LL}) shows that $\phi \in C^{0,\alpha}(B(0,R))$ for some $\alpha>0$.   
		
		For $\omega>-|b|$ being given, there exists $\vareps>0$ such that $\omega>-(1-\vareps)|b|$. As $\phi \in L^r(\R^3)$ for all $2\leq r\leq \infty$, by Observation \ref{obse-V-vareps}, we see that \eqref{est-V} holds with $\mu^*=(1-\vareps)|b|$. Applying Lemma \ref{lem-expo-deca-Raymond} with $\mu=-\omega$ the eigenvalue of $-(\nabla+iA)^2 +V$ associated to the eigenfunction $\phi$, namely
		\[
		-(\nabla+iA)^2 \phi - |\phi|^\alpha \phi = -\omega \phi,
		\]
		there exists $\delta>0$ such that $e^{\delta |x|}\phi \in L^2(\R^3)$. The proof is complete.
	\end{proof}
	
	\begin{lemma}
		Let $A$ be as in \eqref{defi-A}, $0<\alpha<4$, $\omega>-|b|$, and $\phi \in \Gc(\omega)$. Then we have
		\[
		K_\omega(\phi) = H(\phi) =0,
		\]
		where $H$ and $K_\omega$ are as in \eqref{defi-H} and \eqref{defi-K-omega} respectively.
	\end{lemma}
	
	\begin{proof}
		Since $\phi \in \Gc(\omega)$, we see that $\phi$ is a solution to \eqref{ell-equ}. By multiplying both sides of \eqref{ell-equ} with $\overline{\phi}$ and integrating over $\R^3$, we have $K_\omega(\phi)=0$. As $\phi \in \Sigma_A(\R^3)$ (see Proposition \ref{prop-expo-deca}), we have from \eqref{mag-norm} that
		\begin{align} \label{poho-prof-1}
		\|\nabla \phi\|^2_{L^2} + b R(\phi) + \frac{b^2}{4} \|\rho\phi\|^2_{L^2} + \omega \|\phi\|^2_{L^2} - \|\phi\|^{\alpha+2}_{L^{\alpha+2}}=0.
		\end{align}
		On the other hand, by \eqref{nabla-A}, we rewrite \eqref{ell-equ} as 
		\begin{align} \label{ell-equ-equi}
		-\Delta \phi + bL_z\phi +\frac{b^2}{4} \rho^2 \phi + \omega \phi - |\phi|^\alpha \phi =0.
		\end{align}
		Multiplying both sides of \eqref{ell-equ-equi} with $x \cdot \nabla \overline{\phi}$, integrating over $\R^3$, and taking the real part, we get
		\begin{align} \label{poho-prof-2}
		-\frac{1}{2}\|\nabla \phi\|^2_{L^2} - \frac{3}{2} b R(\phi) - \frac{5b^2}{8} \|\rho\phi\|^2_{L^2} - \frac{3\omega}{2} \|\phi\|^2_{L^2} +\frac{3}{\alpha+2} \|\phi\|^{\alpha+2}_{L^{\alpha+2}} =0.
		\end{align}
		Here we have used the following identities which can be showed by integration by parts:
		\begin{align*}
		\rea \left( \int x \cdot \nabla \overline{\phi} \Delta \phi dx\right) &= \frac{1}{2}\|\nabla \phi\|^2_{L^2}, \\
		\rea \left( \int x \cdot \nabla \overline{\phi} \rho^2 \phi dx \right) &= -\frac{5}{2} \|\rho \phi\|^2_{L^2}, \\
		\rea \left( \int x \cdot \nabla \overline{\phi} L_z \phi dx\right) &= -\frac{3}{2} R(\phi), \\
		\rea \left( \int x \cdot \nabla \overline{\phi} |\phi|^\alpha \phi dx \right) &= -\frac{3}{\alpha+2} \|\phi\|^{\alpha+2}_{L^{\alpha+2}}.
		\end{align*}
		From \eqref{poho-prof-1} and \eqref{poho-prof-2}, we infer that
		\[
		\|\nabla \phi\|^2_{L^2} -\frac{b^2}{4} \|\rho \phi\|^2_{L^2} -\frac{3\alpha}{2(\alpha+2)} \|\phi\|^{\alpha+2}_{L^{\alpha+2}}=0
		\]
		or $H(\phi)=0$. The proof is complete.
	\end{proof}

	\begin{lemma} \label{lem-key-est}
		Let $A$ be as in \eqref{defi-A}, $\frac{4}{3}<\alpha<4$, $\omega>-|b|$, and $\phi \in \Gc(\omega)$. Assume that $\left. \partial^2_\lambda S_\omega(\phi^\lambda)\right|_{\lambda=1} \leq 0$, where $\phi^\lambda$ is as in \eqref{defi-phi-lambda}. Let $f \in \Sigma_A(\R^3)$ be such that
		\begin{align} \label{cond-f}
		M(f) = M(\phi), \quad R(f) = R(\phi), \quad K_\omega(f) \leq 0, \quad H(f) \leq 0.
		\end{align}
		Then 
		\begin{align} \label{key-est}
		H(f) \leq 2 \left(S_\omega(f) - d(\omega)\right),
		\end{align}
		where $d(\omega)$ is as in \eqref{defi-d-omega}.
	\end{lemma}
	
	\begin{proof}
		The proof is inspired by an idea of M. Ohta \cite{Ohta}. We first consider the case $K_\omega(f)=0$. By the definition of $d(\omega)$ and $H(f) \leq 0$, we have
		\[
		d(\omega) \leq S_\omega(f) \leq S_\omega(f) - \frac{1}{2} H(f)
		\]
		which shows \eqref{key-est}.
		
		We now consider the case $K_\omega(f)<0$. As $f \in \Sigma_A(\R^3)$, we have from \eqref{mag-norm} that
		\[
		K_\omega(f^\lambda) = \lambda^2 \|\nabla f\|^2_{L^2} + b R(f) + \frac{b^2}{4} \lambda^{-2} \|\rho f\|^2_{L^2} + \omega \|f\|^2_{L^2} - \lambda^{\frac{3\alpha}{2}} \|f\|^{\alpha+2}_{L^{\alpha+2}},
		\]
		where $f^\lambda(x):= \lambda^{\frac{3}{2}} f(\lambda x)$. As $K_\omega(f)<0$ and $K_\omega(f^\lambda)>0$ for $\lambda>0$ sufficiently small, there exists $\lambda_0 \in (0,1)$ such that $K_\omega(f_{\lambda_0})=0$. It follows that
		\begin{align*}
		\frac{\alpha}{2(\alpha+2)} \|\phi\|^{\alpha+2}_{L^{\alpha+2}} = S_\omega(\phi) =d(\omega) \leq S_\omega(f_{\lambda_0}) = \frac{\alpha}{2(\alpha+2)} \|f_{\lambda_0}\|^{\alpha+2}_{L^{\alpha+2}} = \frac{\alpha}{2(\alpha+2)} \lambda_0^{\frac{3\alpha}{2}} \|f\|^{\alpha+2}_{L^{\alpha+2}}
		\end{align*}
		which yields
		\begin{align} \label{est-lambda-0}
		\|\phi\|^{\alpha+2}_{L^{\alpha+2}} \leq \lambda_0^{\frac{3\alpha}{2}} \|f\|^{\alpha+2}_{L^{\alpha+2}}.
		\end{align}
		If $\|\rho f\|_{L^2} \geq \|\rho \phi\|_{L^2}$, then we infer from \eqref{cond-f} and \eqref{est-lambda-0} that
		\begin{align*}
		d(\omega) =S_\omega(\phi) &= S_\omega(\phi) -\frac{1}{2} H(\phi) \\
		&= \frac{\omega}{2} \|\phi\|^2_{L^2} + \frac{b}{2} R(\phi) +\frac{b^2}{4}\|\rho \phi\|^2_{L^2} + \frac{3\alpha-4}{4(\alpha+2)} \|\phi\|^{\alpha+2}_{L^{\alpha+2}} \\
		&\leq \frac{\omega}{2} \|f\|^2_{L^2} + \frac{b}{2} R(f) + \frac{b^2}{4} \|\rho f\|^2_{L^2} + \frac{3\alpha-4}{4(\alpha+2)} \lambda_0^{\frac{3\alpha}{2}} \|f\|^{\alpha+2}_{L^{\alpha+2}} \\
		&\leq \frac{\omega}{2} \|f\|^2_{L^2} + \frac{b}{2} R(f) + \frac{b^2}{4} \|\rho f\|^2_{L^2} + \frac{3\alpha-4}{4(\alpha+2)} \|f\|^{\alpha+2}_{L^{\alpha+2}} \\
		&=S_\omega(f)- \frac{1}{2} H(f)
		\end{align*}
		which shows \eqref{key-est}.
		
		Finally we assume that $\|\rho f\|_{L^2} < \|\rho \phi\|_{L^2}$. We consider
		\begin{align*}
		J(\lambda):&= S_\omega(f^\lambda)-\frac{\lambda^2}{2} H(f) \\
		&=\frac{\omega}{2} \|f\|^2_{L^2} + \frac{b}{2} R(f) + \frac{b^2}{8} \left(\lambda^{-2} + \lambda^2\right) \|\rho f\|^2_{L^2} - \frac{1}{\alpha+2} \left(\lambda^{\frac{3\alpha}{2}} -\frac{3\alpha}{4} \lambda^2\right) \|f\|^{\alpha+2}_{L^{\alpha+2}}.
		\end{align*}
		We claim that 
		\begin{align} \label{claim-J}
		J(\lambda_0) \leq J(1).
		\end{align}
		Let us assume \eqref{claim-J} for the moment and complete the proof of Lemma \ref{lem-key-est}. Indeed, we have
		\[
		d(\omega)=S_\omega(f_{\lambda_0}) \leq S_\omega(f_{\lambda_0}) -\frac{\lambda_0^2}{2} H(f) =J(\lambda_0) \leq J(1) = S_\omega(f)-\frac{1}{2} H(f)
		\] 
		which implies \eqref{key-est}.
		
		It remains to show \eqref{claim-J} which is in turn equivalent to show
		\begin{align} \label{claim-J-proof-1}
		\frac{b^2}{8} \left(\lambda_0^{-2} + \lambda_0^2 -2\right) \|\rho f\|^2_{L^2} \leq \frac{1}{\alpha+2} \left( \lambda_0^{\frac{3\alpha}{2}} - \frac{3\alpha}{4} \lambda_0^2 +\frac{3\alpha}{4}-1 \right) \|f\|^{\alpha+2}_{L^{\alpha+2}}.
		\end{align}
		Since $\left. \partial^2_\lambda S_\omega(\phi^\lambda)\right|_{\lambda=1} \leq 0$ and $H(\phi)=0$, we see that
		\[
		b^2 \|\rho \phi\|^2_{L^2} \leq \frac{3\alpha(3\alpha-4)}{4(\alpha+2)} \|\phi\|^{\alpha+2}_{L^{\alpha+2}}.
		\]
		Using \eqref{est-lambda-0}, we infer that
		\begin{align} \label{claim-J-proof-2}
		b^2\|\rho f\|^2_{L^2} < b^2\|\rho \phi\|^2_{L^2} \leq \frac{3\alpha(3\alpha-4)}{4(\alpha+2)} \|\phi\|^{\alpha+2}_{L^{\alpha+2}} \leq \frac{3\alpha(3\alpha-4)}{4(\alpha+2)} \lambda_0^{\frac{3\alpha}{2}}\|f\|^{\alpha+2}_{L^{\alpha+2}}.
		\end{align}
		From \eqref{claim-J-proof-2}, \eqref{claim-J-proof-1} holds provided that
		\[
		\frac{3\alpha(3\alpha-4)}{32(\alpha+2)} \left( \lambda_0^{-2}+\lambda_0^2-2\right) \lambda_0^{\frac{3\alpha}{2}} \leq \frac{1}{\alpha+2} \left(\lambda_0^{\frac{3\alpha}{2}}-\frac{3\alpha}{4} \lambda_0^2 +\frac{3\alpha}{4}-1 \right).
		\]
		The above inequality is equivalent to
		\[
		\beta(\beta-1) (\lambda_0^{-1}-\lambda_0)^2 \lambda_0^{2\beta} \leq 2\left(\lambda_0^{2\beta} - \beta \lambda_0^2 +\beta -1\right)
		\]
		which is $P(\lambda_0^2) \geq 0$, where
		\[
		P(s):= s^\beta-1-\beta(s-1) - \frac{1}{2} \beta(\beta-1) s^{\beta-1} (s-1)^2.
		\]
		We take the Taylor expansion of $s^\beta$ at $s=1$ to get
		\[
		s^\beta =1 + \beta(s-1) +\frac{1}{2} \beta (\beta-1) (s-1)^2 s_0^{\beta-2}
		\]
		for some $s_0 \in [s,1]$. This shows that
		\[
		P(\lambda_0^2) = \frac{1}{2}\beta(\beta-1) (\lambda_0^2-1)^2\left(s_0^{\beta-2} - \lambda_0^{2\beta-2}\right)
		\]
		with $\lambda_0^2 \leq s_0 \leq 1$. Since $\lambda_0^{2\beta-2} \leq s_0^{\beta-1} \leq s_0^{\beta-2}$, we have $P(\lambda_0^2)\geq 0$. This proves \eqref{claim-J-proof-1}, hence \eqref{claim-J}. The proof is complete.
	\end{proof}
	
	We are now able to prove Theorem \ref{theo-insta-grou}.
	
	\begin{proof}[Proof of Theorem \ref{theo-insta-grou}]
		We define the set
		\[
		\Bc(\omega):= \left\{ f \in \Sigma_A(\R^3) \ : \ M(f) = M(\phi), R(f)=R(\phi), S_\omega(f) < d(\omega), K_\omega(f)<0, H(f) <0 \right\}.
		\]
		\begin{observation} \label{obse-inva}
			The set $\Bc(\omega)$ is invariant under the flow of \eqref{mag-NLS}, i.e., if $u_0 \in \Bc(\omega)$, then $u(t) \in \Bc(\omega)$ for all $t\in [0,T^*)$.
		\end{observation}
		\begin{proof}
			In fact, let $u_0 \in \Bc(\omega)$ and $u: [0,T^*) \times \R^3 \rightarrow \C$ be the corresponding solution to \eqref{mag-NLS}. By the conservation of mass and energy, we have $M(u(t))=M(u_0)=M(\phi)$ and $S_\omega(u(t))=S_\omega(u_0)<d(\omega)$ for all $t\in [0,T^*)$. Thanks to Lemma \ref{lem-cons-angu}, we see that $R(u(t))=R(u_0)=R(\phi)$ for all $t\in [0,T^*)$. We will show that $K_\omega(u(t))<0$ for all $t\in [0,T^*)$. Suppose that it does not hold, then there exists $t_0 \in [0,T^*)$ such that $K_\omega(u(t_0))\geq 0$. By the continuity of $t\mapsto K_\omega(u(t))$, there exists $t_1 \in (0,t_0]$ such that $K_\omega(u(t_1))=0$. From the definition of $d(\omega)$, we get $d(\omega) \leq S_\omega(u(t_1))=S_\omega(u_0) <d(\omega)$ which is a contradiction. Finally we prove that $H(u(t))<0$ for all $t\in[0,T^*)$. If it is not true, then arguing as above, there exists $t_2 \in [0,T^*)$ such that $H(u(t_2))=0$. Applying Lemma \ref{lem-key-est} to $f=u(t_2)$, we get
			\[
			0=H(u(t_2)) \leq 2\left(S_\omega(u(t_2))-d(\omega)\right) 
			\]
			which implies 
			\[
			d(\omega) \leq S_\omega(u(t_2))=S_\omega(u_0) <d(\omega).
			\]
			This is again a contradiction. Thus we have
			\[
			M(u(t))=M(u_0)=M(\phi), \quad R(u(t))=R(u_0)=R(\phi), \quad S_\omega(u(t))=S_\omega(u_0)<d(\omega)
			\]
			and $K_\omega(u(t))<0$, $H(u(t))<0$ for all $t\in [0,T^*)$. This shows Observation \ref{obse-inva}. 
		\end{proof}
	
		\begin{observation} \label{obse-lambda}
			We have $\phi^\lambda \in \Bc(\omega)$ for all $\lambda>1$, where $\phi^\lambda$ is as in \eqref{defi-phi-lambda}.
		\end{observation}
		\begin{proof}
		A straightforward computation shows
		\[
		M(\phi^\lambda)= M(\phi), \quad R(\phi^\lambda) = M(\phi), \quad \forall \lambda>0.
		\]
		Next we have
		\begin{align*}
		\partial^2_\lambda S_\omega(\phi^\lambda) &= \|\nabla \phi\|^2_{L^2} + \frac{3b^2}{4} \lambda^{-4} \|\rho \phi\|^2_{L^2} - \frac{3\alpha}{2(\alpha+2)}\left(\frac{3\alpha}{2}-1\right) \lambda^{\frac{3\alpha-4}{2}} \|\phi\|^{\alpha+2}_{L^{\alpha+2}} \\
		&<\|\nabla \phi\|^2_{L^2} + \frac{3b^2}{4} \|\rho\phi\|^2_{L^2} - \frac{3\alpha}{2(\alpha+2)}\left(\frac{3\alpha}{2}-1\right) \|\phi\|^{\alpha+2}_{L^{\alpha+2}} \\
		&=\left. \partial^2_\lambda S_\omega(\phi^\lambda)\right|_{\lambda=1} \leq 0, \quad \forall \lambda>1.
		\end{align*}
		It yields that
		\[
		\partial_\lambda S_\omega(\phi^\lambda) < \left. \partial_\lambda S_\omega(\phi^\lambda)\right|_{\lambda=1} =H(\phi)=0, \quad \forall \lambda>1
		\]
		which shows
		\[
		S_\omega(\phi^\lambda) < S_\omega(\phi), \quad \forall \lambda>1. 
		\]
		We also have
		\[
		H(\phi^\lambda)=\lambda \partial_\lambda S_\omega(\phi^\lambda) <0, \quad \forall \lambda>1.
		\]
		
		It remains to show that $K_\omega(\phi^\lambda)<0$ for all $\lambda>1$. We have
		\begin{align*}
		\partial^3_\lambda K_\omega(\phi^\lambda) &= -6b^2\lambda^{-5} \|\rho \phi\|^2_{L^2} - \frac{3\alpha}{2} \left( \frac{3\alpha}{2}-1\right) \left(\frac{3\alpha}{2}-2\right) \lambda^{\frac{3\alpha}{2}-3} \|\phi\|^{\alpha+2}_{L^{\alpha+2}} <0, \quad \forall \lambda>0.
		\end{align*}
		It follows that
		\[
		\partial^2_\lambda K_\omega(\phi^\lambda) < \left. \partial^2_\lambda K_\omega(\phi^\lambda)\right|_{\lambda=1} = 2 \|\nabla \phi\|^2_{L^2} +\frac{3b^2}{2} \|\rho \phi\|^2_{L^2} -\frac{3\alpha}{2} \left(\frac{3\alpha}{2}-1\right) \|\phi\|^{\alpha+2}_{L^{\alpha+2}}, \quad \forall \lambda>1.
		\]
		By the assumption $\left. \partial^2_\lambda S_\omega(\phi^\lambda)\right|_{\lambda=1} \leq 0$ which is equivalent to
		\[
		\|\nabla \phi\|^2_{L^2} +\frac{3b^2}{4} \|\rho \phi\|^2_{L^2} - \frac{3\alpha}{2(\alpha+2)} \left(\frac{3\alpha}{2}-1\right) \|\phi\|^{\alpha+2}_{L^{\alpha+2}} \leq 0,
		\]
		we infer that
		\[
		\partial^2_\lambda K_\omega(\phi^\lambda) \leq -\frac{3\alpha^2}{2(\alpha+2)} \left(\frac{3\alpha}{2}-1\right) \|\phi\|^{\alpha+2}_{L^{\alpha+2}} <0, \quad \forall \lambda>1.
		\]
		This shows that
		\[
		\partial_\lambda K_\omega(\phi^\lambda) < \left. \partial_\lambda K_\omega(\phi^\lambda)\right|_{\lambda=1} = 2\|\nabla \phi\|^2_{L^2} - \frac{b^2}{2} \|\rho \phi\|^2_{L^2} - \frac{3\alpha}{2} \|\phi\|^{\alpha+2}_{L^{\alpha+2}}, \quad \forall \lambda>1.
		\]
		Using the fact that $H(\phi)=0$, we obtain
		\[
		\partial_\lambda K_\omega(\phi^\lambda) < -\frac{3\alpha^2}{2(\alpha+2)} \|\phi\|^{\alpha+2}_{L^{\alpha+2}}, \quad \forall \lambda>1.
		\]
		This shows that $K_\omega(\phi^\lambda)< K_\omega(\phi)=0$ for all $\lambda>1$. Therefore we prove that $\phi^\lambda \in \Bc(\omega)$ for all $\lambda>1$.
		\end{proof}
	
		Now let $\vareps>0$. As $\phi^\lambda \rightarrow \phi$ strongly in $\Sigma_A(\R^3)$ as $\lambda \rightarrow 1$, there exists $\lambda_0>1$ such that $\|\phi_{\lambda_0}-\phi\|_{\Sigma_A} <\vareps$. Set $u_0 = \phi_{\lambda_0} \in \Bc(\omega)$ and let $u:[0,T^*) \times \R^3 \rightarrow \C$ be the corresponding solution to \eqref{mag-NLS}. By Observation \ref{obse-lambda} and Observation \ref{obse-inva}, $u(t) \in \Bc(\omega)$ for all $t\in [0,T^*)$. Applying Lemma \ref{lem-key-est} to $f=u(t)$ and using the conservation laws of mass and energy, we get
		\[
		H(u(t)) \leq 2\left(S_\omega(u(t))-d(\omega)\right)=2\left(S_\omega(u_0) - d(\omega)\right) <0, \quad \forall t\in [0,T^*).
		\]
		Thanks to Lemma \ref{lem-viri-iden}, we have
		\[
		F''(u(t)) = 8 H(u(t)) \leq -16\left(S_\omega(u_0)-d(\omega)\right)<0, \quad \forall t\in [0,T^*),
		\]
		where $F$ is as in \eqref{defi-F}. The convexity argument shows that $T^*<\infty$. The proof is complete.
		\end{proof}

	\section*{Acknowledgments}
	This work was supported in part by the European Union’s Horizon 2020 Research and Innovation Programme (Grant agreement CORFRONMAT No. 758620, PI: Nicolas Rougerie). V.D.D. would like to express his deep gratitude to his wife - Uyen Cong for her encouragement and support. The authors would like to thank the reviewers for their helpful comments and suggestions.



\begin{thebibliography}{99}
		
		\bibitem{AFF} {C. O. Alves, G. M. Figueiredo, and M. F. Furtado}, {\it Multiple solutions for a nonlinear Schrödinger equation with magnetic fields}, Commun. Partial Differential Equations 36 (2011), no. 9, 1565--1586.
		
		\bibitem{AHS-1} {J. Avron, I. Herbst, and B. Simon}, {\it Schr\"odinger operators with magnetic fields I. General interactions}, Duke Math. J. 45 (1978), no. 4, 847--883.
		
		\bibitem{AHS-2} {J. Avron, I. Herbst, and B. Simon}, {\it Separation of center of mass in homogeneous magnetic fields}, Ann. Phys. 114 (1978), 431--451.
		
		\bibitem{AHS-3} {J. Avron, I. Herbst, and B. Simon}, {\it Schr\"odinger operators with magnetic fields. III. Atoms in homogeneous magnetic fields}, Comm. Math. Phys. 79 (1981), 529--572. 
		
		\bibitem{AS} {G. Arioli and A. Szulkin}, {\it A semilinear Schr\"odinger equations in the presence of a magnetic field}, Arch. Ration. Mech. Anal. 170 (2003), 277--295.
		
		\bibitem{BC} {W. Bao and Y. Cai}, {\it Mathematical theory and numerical methods for Bose-Einstein condensation}, Kinet. Relat. Models 6 (2013), no. 1, 1--135.
		
		\bibitem{BBJV} {J. Bellazzini, N. Boussa\"id, L. Jeanjean, and N. Visciglia}, {\it Existence and stability of standing waves for supercritical NLS with a partial confinement}, Comm. Math. Phys. 353 (2017), no. 1, 229--251.
		
		\bibitem{BL} {H. Br\'ezis and E. H. Lieb}, {\it A relation between pointwise convergence of functions and convergence of functionals}, Proc. Amer. Math. Soc. 88 (1983), no. 3, 486--490.
		
		\bibitem{CE} {T. Cazenave and M. J. Esteban}, {\it On the stability of stationary states for nonlinear Schr\"odinger equations with an external magnetic field}, Mat. Appl. Comp. 7 (1988), 155--168.
		
		\bibitem{Cazenave} {T. Cazenave}, {\it Semilinear Schr\"odinger equations}, Courant Lecture Notes in Mathematics, vol. 10, Courant Institute of Mathematical Sciences, American Mathematical Society, Providence, RI, 2003.
		
		\bibitem{CS} {J. Chabrowski and A. Szulkin}, {\it On the Schr\"odinger equation involving a critical Sobolev exponent and magnetic field}, Topol. Methods Nonlinear Anal., 25 (2005), 3--21.
		
		\bibitem{Cingolani} {S. Cingolani}, {\it Semiclassical stationary states of nonlinear Schrödinger equations with an external magnetic field}, J. Differential Equations 188 (2003), no. 1, 52--79.
		
		\bibitem{CS} {S. Cingolani and S. Secchi}, {\it Semiclassical limit for nonlinear Schrödinger equations with electromagnetic fields}, J. Math. Anal. Appl. 275 (2002), no. 1, 108--130.
		
		\bibitem{CJS} {S. Cingolani, L. Jeanjean, and S. Secchi}, {\it Multi-peak solutions for magnetic NLS equations without non-degeneracy conditions}, ESAIM: Control, Optimisation and Calculus of Variations 15 (2009), no. 3, 653--675.
		
		\bibitem{CCL} {J. Colliander, M. Czubak, and J. Lee}, {\it Interaction Morawetz estimate for the magnetic Schrödinger equation and applications}, Adv. Differ. Equ. 19 (2014), no. 9/10, 805--832.
		
		\bibitem{DFV} {P. D'Ancona, L. Fanelli, and L. Vega}, {\it Endpoint Strichartz estimates for the magnetic Schrödinger equation}, J. Funct. Anal. 258 (2010), no. 10, 3227--3240.
		
		\bibitem{DR} T. Duyckaerts and S. Roudenko, {\it Going beyond the threshold: scattering and blow-up in the focusing NLS equation},
		Comm. Math. Phys. 334 (2015), no. 3, 1573--1615.
		
		\bibitem{EL} {M. J. Esteban and P.-L. Lions}, {\it Stationary solutions of nonlinear Schr\"odinger equations with an external magnetic field}, in Partial differential equations and the calculus of variations I (1989), 401--449. 
		
		\bibitem{Garcia} {A. Garcia}, {\it Magnetic virial identities and applications to blow-up for Schr\"odinger and	wave equations}, Journal of Physics A: Mathematical and Theoretical, 45(2011), 015202.
		
		\bibitem{FV} {L. Fanelli and L. Vega}, {\it Magnetic virial identities, weak dispersion and Strichartz inequalities}, Math. Ann. 344 (2009), no. 2,  249--278. 
		
		\bibitem{FO} {R. Fukuizumi and M. Ohta}, {\it Instability of standing waves for nonlinear Schr\"odinger equations with potentials}, Differ. Integral Equ. 16 (2003), no. 6, 691--706.
		
		\bibitem{HR} {J. Holmer and S. Roudenko}, {\it A sharp condition for scattering of the radial 3D cubic nonlinear Schr\"odinger equation}, Comm. Math. Phys. 282 (2008), 435--467.
		
		\bibitem{Kurata} {K. Kurata}, {\it Existence and semi-classical limit of the least energy solution to a nonlinear Schrödinger equation with electromagnetic fields}, Nonlinear Anal. 41 (2000), no. 5, 763--778.
		
		\bibitem{KL} {T. F. Kieffer and M. Loss}, {\it Non-linear Schr\"odinger equation in a uniform magnetic field}, preprint, available at \url{http://arxiv.org/abs/2010.12961v1}.
		
		\bibitem{LL} {E. H. Lieb and M. Loss}, {\it Analysis}, 2nd Edition, Graduate Studies in Mathematics, vol. 14, American Mathematical Society, Providence, RI, 2001.
		
		\bibitem{LS} {E. H. Lieb and R. Seiringer}, {\it The stability of matter in quantum mechanics}, Cambridge Univ. Press, 2010.
		
		\bibitem{Lions} P.-L. Lions, {\it The concentration-compactness principle in the Calculus of Variation. The	locally compact case, part I and II}, Ann. Inst. H. Poincar\'e Anal. Non Lin\'eaire 1 (1984), no. 2, 109--145 and no. 4, 223--283.
		
		\bibitem{Ohta} {M. Ohta}, {\it Strong instability of standing waves for nonlinear Schr\"odinger equations with harmonic potential}, Funkcial. Ekvac. 61 (2018), no. 1, 135--143.
		
		\bibitem{Raymond} {N. Raymond}, {\it Bound states of the magnetic Schr\"odinger operator}, volume 27. EMS Tracts, 2017.
		
		\bibitem{Ribeiro} {J. M. Gonçalves Ribeiro}, {\it Finite time blow-up for some nonlinear Schr\"odinger equations with an external magnetic field}, Nonlinear Anal. 16 (1991), 941--948.
		
		\bibitem{Ribeiro-insta} {J. M. Gonçalves Ribeiro}, {\it Instability of symmetric stationary states for some nonlinear 			Schr\"odinger equations with an external magnetic field}, Ann. Inst. H. Poincar\'e. Phys. Th\'eor., 54 (1991), 403--433.
	\end{thebibliography}
\end{document}